%% file: main.tex
\documentclass[a4paper]{article}
\usepackage{amsthm,amssymb,amsmath,enumerate,graphicx,epsf}
\usepackage{bbold}
\usepackage{mathtools,enumitem}
\usepackage{authblk}
\usepackage[percent]{overpic}
\usepackage{hyperref}
\RequirePackage{latexsym} \RequirePackage{amsthm}
\RequirePackage{amsmath}
\usepackage{authblk}
\usepackage{tikz}
\usepackage[dvipsnames]{xcolor}
\usepackage{subcaption}
\usepackage{geometry}
\usepackage{enumerate}
\RequirePackage{amssymb} \RequirePackage{makeidx}
\usepackage{dsfont}
\usepackage{mathrsfs}
\usepackage[noabbrev,capitalise,nameinlink]{cleveref}

\newcommand{\COLORON}{1}
\newcommand{\NOTESON}{0}
\newcommand{\Debug}{0} 

\input{defs}

\newcommand{\asm}{\ensuremath{\preceq_\infty}}
\newcommand{\nasm}{\ensuremath{\not\preceq_\infty}}
\newcommand{\thmin}{\ensuremath{\preceq_\mathrm{thin}}}
\newcommand{\cmin}{\ensuremath{\preceq_\mathrm{c}}}

\newcommand{\css}{coarsely self-similar}
\newcommand{\good}{compressible}
\newcommand{\bad}{incompressible}

\newcommand{\blue}[1]{{\textcolor{blue}{#1}}}

\DeclareMathOperator{\dist}{dist}
\DeclareMathOperator{\SO}{S_0}
\DeclareMathOperator{\SI}{S_1}
\DeclareMathOperator{\TO}{T_0}
\DeclareMathOperator{\TI}{T_1}
\DeclareMathOperator{\RR}{R}

\DeclareMathOperator{\SP}{S\Delta}
\DeclareMathOperator{\TP}{T\Delta}
\DeclareMathOperator{\SSS}{\mathbf{SS}}
\DeclareMathOperator{\TTT}{\mathbf{TT}}
\DeclareMathOperator{\RTT}{\mathbf{RTT}}
\DeclareMathOperator{\RRR}{\mathbf{RRR}}

\DeclareMathOperator{\Adh}{Adh}

\DeclareMathOperator{\Ext}{Ext}

\begin{document}

\title{Small counterexamples to the fat minor conjecture}

\author[1]{Sandra Albrechtsen\thanks{Supported by the Alexander von Humboldt Foundation in the framework of the Alexander von Humboldt Professorship of Daniel Král' endowed by the Federal Ministry of Education and Research.}}
\affil[1]{{Institute of Mathematics}\\ {Leipzig University}\\ {Augustusplatz 10}\\ {04109 Leipzig}\\ {Germany}}
\author[2]{Marc Distel\thanks{Supported by Australian Government Research Training Program Scholarship.}}
\affil[2]{{School of Mathematics}\\ {Monash University}\\ {Melbourne}\\ {Australia}}
\author[3]{Agelos Georgakopoulos\thanks{Supported by EPSRC grant  EP/V009044/1.}}
\affil[3]{  {Mathematics Institute}\\ {University of Warwick}\\  {CV4 7AL, UK}}

\date{January 9, 2026}
\maketitle

\begin{abstract}
We narrow the gap between the family of graphs that do and the family of graphs that do not satisfy the fat minor conjecture by obtaining much simpler counterexamples than were previously known, including $K_t, t\geq 6$ and $K_{s,t}, s,t\geq 4$ and $K_{2,2,2}$. 

This is achieved by establishing a `coarse self-similarity' property of the graphs used by Nguyen, Scott and Seymour to disprove the `coarse Menger conjecture'. This property may be of independent interest.
\end{abstract}

{\bf{Keywords:} } coarse graph theory, quasi-isometry, fat minor, self-similarity.\\

{\bf{MSC 2020 Classification:}} 05C83, 05C10, 05C63, 51F30.

\maketitle

\section{Introduction}

We continue the study of \defi{$K$-fat minors}, a geometric analogue of the classical notion of graph minor whereby branch sets are required to be at least some distance $K$ from each other, and the edges connecting them are replaced by long paths, also at distance  $K$ from each other, and from their non-incident branch sets; see \Sr{sec FM} for details. This notion belongs to a more general framework of studying large-scale properties of graphs from a geometric perspective, called `coarse graph theory', paralleling Gromov's approach to geometric group theory \cite{GroAsyInv,DruKapBook}. For more on the motivation of coarse graph theory we refer the reader to \cite{ADGK2t,BBEGLPS,GeoPapMin}. Another central notion of this area is that of a quasi-isometry, a notion of similarity of metric spaces (defined in Section~\ref{subsec:QuasiIsometries}) that has played a  seminal role in geometric group theory. 

In this paper we continue the study of the following influential, though false, conjecture of Papasoglu and the third author connecting fat minors and quasi-isometries, narrowing the gap between graphs known to satisfy it and graphs known to fail it:

\begin{conjecture}[\cite{GeoPapMin}] \label{conj fat min}
For every finite graph $J$ and every $K\in \N$ there exist $M,A\in \N$  such that every graph with no $K$-fat $J$~minor is $(M,A)$-quasi-isometric to a graph with no $J$~minor.
\end{conjecture}

We call a graph $J$ \defi{\good}, if it satisfies this conjecture (for all $K$), and  \defi{\bad} otherwise. We will provide much smaller incompressible graphs than those previously found by Davies, Hickingbotham, Illingworth and McCarty \cite{DHIM} and by Davies and the first author~\cite{ADWeakCounterex}. In particular, we will prove that $K_{2,2,2}$ is \bad, and we will use this to deduce that $K_t, t\geq 6$ and $K_{s,t}, s,t\geq 4$ are \bad. 

It was previously known that the following graphs are \good: cycles \cite{GeoPapMin}, $K_{1,t},t\geq 1$ \cite{GeoPapMin,NgScSeAsyII}
$K^-_4$ \cite{FujPapCoa, AJKW}, $K_{2,3}$ \cite{CDNRV,FujPapCoa},  and $K_4$ \cite{AJKW}. In a companion paper we prove that $K_{2,t},t\geq 1$ are also \good\ \cite{ADGK2t}. Combined with our aforementioned counterexamples, these results single out the Kuratowski graphs $K_5, K_{3,3}$ as `critical' cases for which it is not known whether they are \good. This is essentially the `coarse Kuratowski conjecture' of \cite{GeoPapMin}, which seeks to characterize the graphs quasi-isometric to a planar graph. 

The results we will prove strengthen the ones mentioned above. Our first one is

\begin{theorem} \label{thm:Counterexample:Octahedron}
    For every $M, A \in \N$ there exists a graph \g with no $3$-fat $K_{2,2,2}$ minor \st\ each graph that is $(M,A)$-quasi-isometric to \g has a $2$-fat $K_7$ minor.
\end{theorem}

Theorem~\ref{thm:Counterexample:Octahedron} is stronger than saying that $K_6$ is \bad\ in several ways, because $K_{2,2,2} \subset K_6 \subset K_7$, and the $K_7$ minor we guarantee is $2$-fat. It is also much stronger than saying that $K_{2,2,2}$ is \bad. We point out that $K_{2,2,2}$ is a planar graph --- isomorphic to the 1-skeleton of the octahedron (Figure~\ref{fig:O}). 

Likewise, our next result says that $K_{4,t}$ is \bad\ \fe\ $t\geq 4$, and much more: 

\begin{theorem} \label{thm:FatK4t}
    For every $t \in \N$ and $M, A \in \N$ there exists a graph $G$ with no $3$-fat $K_{4,4}$ minor \st\ each graph that is $(M,A)$-quasi-isometric to \g has a $2$-fat $K_{4,t}$ minor.
\end{theorem}

All known counterexamples to \Cnr{conj fat min} \cite{DHIM,ADWeakCounterex} are based on a construction of Nguyen, Scott and Seymour~\cite{NgScSeCou} (see \fig{fig:NSSGraph} in Section~\ref{sec:constr NSS}), which we will call the \defi{NSS graph(s)}, that disproves the `coarse Menger conjecture' of \cite{AHJKW,GeoPapMin}, and the same holds for our counterexamples. Our proofs of \Trs{thm:Counterexample:Octahedron} and~\ref{thm:FatK4t} are based on a new coarse property of graphs that we introduce that may be of independent interest, which we prove to hold for the NSS graphs: we say that a graph (family) $X$ is \defi{\css}, if $X \preceq Y$ holds for every graph (family) $Y$ quasi-isometric to $X$, where $\preceq$ stands for the minor relation. 
A key result of this paper is  
\begin{theorem} \label{thm css}
    The infinite NSS graph is \css.
\end{theorem}

More precisely, we will prove ---in \Sr{sec css}--- the following finer version, whereby we will provide a 2-fat model. 

\begin{theorem} \label{thm:css:finite}
    For every $M \geq 1$, $A \geq 0$ and $k, d \in \N$, there exist $K, D \in \N$ such that the following holds for all $k'\geq K$ and $d'\geq D$: If $G_{k',d'}$ is $(M,A)$-quasi-isometric to a graph $H$, then $H$ contains $G_{k,d}$ as a $2$-fat minor.
\end{theorem}

\noindent Here $G_{k,d}$ and $G_{k',d'}$ denote NSS graphs (see Section~\ref{sec:constr NSS}). In particular, this implies that $G_{K,D}$ itself contains $G_{k,d}$ as a $2$-fat minor.
\medskip

Interestingly, our proof of this makes use of the 2-path version of the coarse Menger conjecture, which has been proved to be true \cite{AHJKW,GeoPapMin}, applied to graphs used to disprove the 3-path version in~\cite{NgScSeCou}. A further consequence of Theorem~\ref{thm:css:finite} is that NSS graphs are themselves \bad\  (\Sr{sec:Counterex:Proof}). 

\medskip
Another important ingredient for the proof of \Trs{thm:Counterexample:Octahedron} and~\ref{thm:FatK4t} is

\begin{theorem} \label{thm:NSSGraph:NoFatO:Intro}
    $K_{2,2,2}$ is not a $9$-fat minor of the NSS graph. 
\end{theorem}

We prove \Trs{thm:Counterexample:Octahedron} and~\ref{thm:FatK4t} (in \Sr{sec:Counterex:Proof}) by combining the two last theorems with the observation that $K_{2,2,2}$ is a minor of $K_7$ and $K_{4,t}, t\geq 4$, which in turn are minors of the NSS graph (Theorems~\ref{thm:NSSGraph:K7} and \ref{thm:NSSGraph:K4t}), and using a well-known power trick (Theorem~\ref{thm:FromKFatTo3Fat}) to reduce the fatness from 9 to 3. We prove Theorem~\ref{thm:NSSGraph:NoFatO:Intro} in \Sr{sec O fat}.

\medskip
In Section~\ref{sec susp} we prove that  if $J$ is a counterexample to \Cnr{conj fat min} then so is its \defi{suspension $S(J)$}, i.e.\ the graph  obtained from $J$ by adding a new vertex $s_J$ and joining $s_J$ to each $v \in V(J)$  with an edge.
Combining this with Theorem~\ref{thm:Counterexample:Octahedron}, and handling complete bipartite graphs in a similar fashion ---although with additional difficulty--- we deduce 
\begin{corollary} \label{cor Kt}
$K_t,t\geq 6$ and $K_{s,t}, s,t\geq 4$ are \bad.
\end{corollary}

Comparing  Theorem~\ref{thm:FatK4t} with our result that $K_{2,t}$ is \good\  \cite{ADGK2t}, and with the aforementioned `coarse Kuratowski conjecture', makes the following question particularly interesting: 

\begin{question} \label{Q K3t}
    Are there functions $f: \N \rightarrow \N^2$ and $s: \N \rightarrow \N$ such that every graph with no $K$-fat $K_{3,t}$ minor is $f(K)$-quasi-isometric to a graph with no $K_{3,s(t)}$ minor? Can we choose $s(t) = t$?
\end{question}

In \Sr{sec K3t} we will prove the following fact, which shows that if the $K_{3,t}, t\geq 3$ turns out to be \bad, then this cannot be proved by the method we used here to prove that $K_{4,t}, t\geq 4$ is \bad: 

\begin{proposition} \label{prop:K3t:Intro}
    For every $K,t\in \N$, there is $k \in \N$ such that every NSS graph $G_{k',d'}$ with $k' \geq k$ and $d'\geq K$ contains $K_{3,t}$ as a $K$-fat minor.
\end{proposition}

Similarly, we observe that $K_5$ is a $K$-fat minor of $G_{k,K}$ for every sufficiently large $k \in \N$ (Proposition~ \ref{prop:FatK5}), so if $K_5$ is \bad\ (which is an interesting open problem), then this cannot be proved by using NSS graphs.

\section{Preliminaries}

We follow the basic graph-theoretic terminology of~\cite{Bibel}. We do not include $0$ in $\N$. We denote $\N \cup \{0\}$ by $\N_0$.

Let $G$ be a graph $G$ and $X,Y \subseteq V(G)$. An \defi{$X$--$Y$ path} intersects $X$ precisely in its first and $Y$ precisely in its last vertex. 

\subsection{Distances}

Let $G$ be a graph.
We write~\defi{$d_G(v, u)$} for the distance of the two vertices~$v$ and~$u$ in~$G$. 
For two sets~$U$ and~$U'$ of vertices of~$G$, we write~\defi{$d_G(U, U')$} for the minimum distance of two elements of~$U$ and~$U'$, respectively.
If one of~$U$ or~$U'$ is just a singleton, then we omit the braces, writing $d_G(v, U') := d_G(\{v\}, U')$ for $v \in V(G)$.

Given a set~$U$ of vertices of~$G$, the \defi{ball (in~$G$) around~$U$ of radius $r \in \N$}, denoted by~\defi{$B_G(U, r)$}, is the set of all vertices in~$G$ of distance at most~$r$ from~$U$ in~$G$.
If~$U = \{v\}$ for some~$v \in V(G)$, then we omit the braces, writing~$B_G(v, r)$ for the ball (in $G$) around~$v$ of radius~$r$.

Further, the \defi{radius}~\defi{$\rad(G)$} of~$G$ is the smallest number~$k \in \N_0$ such that $V(G) = B_G(v, k)$ for some vertex $v \in V(G)$ or $\infty$ if such a $k \in \N_0$ does not exist.
Additionally, if $U \subseteq V(G)$, then the \defi{radius of $U$ in $G$} is the smallest number $k \in \N_0$ such that there exists some vertex $v$ of $G$ with $U \subseteq B_G(v, k)$ or $\infty$ if such a $k \in \N_0$ does not exist. 

If $Y$ is a subgraph of $G$, then we abbreviate $d_G(U,V(Y))$, $\rad_G(V(Y))$ and $B_G(V(Y),r)$ as \defi{$d_G(U,Y)$}, \defi{$\rad_G(Y)$} and \defi{$B_G(Y,r)$}, respectively.

\subsection{Quasi-isometries} \label{subsec:QuasiIsometries}

For $M \in \R_{\geq 1}$ and $A \in \R_{\geq 0}$, an \defi{$(M, A)$-quasi-isometry} from a graph~$H$ to a graph~$G$ is a map~$\varphi : V(H) \rightarrow V(G)$ such that
\begin{enumerate}[label=\rm{(Q\arabic*)}]
    \item \label{quasiisom:1} $M^{-1} \cdot d_H(g,h) - A \leq d_G(\varphi(g),\varphi(h)) \leq M\cdot d_H(g,h)+A$ for every $g,h \in V(H)$, and
    \item \label{quasiisom:2} for every $v \in V(G)$ there is $h \in V(H)$ such that $d_G(v,\varphi(h)) \leq A$.
\end{enumerate}

The following is a well-known fact. 

\begin{lemma}\label{lem:inversequasiisom}
    If a graph $H$ is $(M,A)$-quasi-isometric to a graph $G$, then $G$ is $(M,3AM)$- quasi-isometric to $H$. \qed
\end{lemma}

\noindent We remark that given an $(M,A)$-quasi-isometry $\varphi$ from $H$ to $G$, one can obtain an $(M, 3AM)$-quasi-isometry from $G$ to $H$ by mapping every vertex $v$ of $G$ to some (arbitrary) vertex $h$ of $H$ such that $d_G(v, \varphi(h)) \leq A$.
\medskip 

The following fact about $(M,A)$-quasi-isometries is immediate from the definition, since for any edge~$uv$ of~$G$ it holds by \ref{quasiisom:1} that $d_H(\varphi(u), \varphi(v)) \leq M+A$. 

\begin{lemma} \label{lem:QIPreservesConn}
    Let $M \geq 1$, $A \geq 0$, and let $\varphi$ be an $(M,A)$-quasi-isometry from a graph $G$ to a graph $H$. If $U \subseteq V(G)$ is connected, then $B_H(\varphi(U), M+A)$ is connected. \qed
\end{lemma}

The next lemma essentially says that quasi-isometries preserve separators (in a coarse sense).

\begin{lemma}[\rm{\cite[Lemma~2.3]{ADWeakCounterex}}] \label{lem:QIPreservesSeps}
    Let $M \geq 1$, $A \geq 0$ and let $\varphi$ be an $(M,A)$-quasi-isometry from a graph $G$ to a graph $H$. If $X \subseteq V(G)$ separates $Y \subseteq V(G)$ and $Z \subseteq V(G)$ in $G$, then $B_H(\varphi(X), r)$ separates $\varphi(Y)$ and $\varphi(Z)$ in $H$ where $r := r(M,A) = M(M(3A+1)+2M^{-1}A+1)$.
\end{lemma}

\subsection{Fat minors} \label{sec FM}

Let $J, G$ be graphs.
A \defi{model} $(\cu,\ce)$ of $J$ in $G$ is a collection $\cu$ of disjoint, connected sets $U_x \subseteq V(G), x\in V(J)$, and a collection $\ce$ of internally disjoint $U_{x}$--$U_{y}$ paths $E_{e}$, one for each edge $e=xy$ of $J$, \st\ $E_{e}$ is disjoint from every $U_z$ with $z \neq x, y$.
The $U_x$ are the \defi{branch sets} and the $E_e$ are the \defi{branch paths} of the model.
A model $(\cu, \ce)$ of $J$ in $G$ is \defi{$K$-fat} for $K \in \N$ if $\dist_G(Y,Z) \geq K$ for every two distinct $Y,Z \in \cu \cup \ce$ unless $Y = E_e$ and $Z = U_x$ for some vertex $x \in V(J)$ incident to $e \in E(J)$, or vice versa.
We say that $J$ is a \defi{($K$-fat) minor} of $G$, and denote it by $J \preceq G$ ($J \preceq_K G$), if $G$ contains a ($K$-fat) model of $X$.
We remark that the $0$-fat minors of $G$ are precisely its minors. 

We say that $J$ is an \defi{asymptotic minor} of $G$, and denote it by $J \preceq_\infty G$, if $J$ is a $K$-fat minor of $G$ for all $K \in \N$.

\subsection{Coarse Menger for two paths}

We need the following `coarse' version of Menger's theorem for two paths, which was proved by Albrechtsen, Huynh, Jacobs, Knappe and Wollan~\cite[Theorem~2]{AHJKW} and, independently, by Georgakopoulos and Papasoglu~\cite[Theorem~8.1]{GeoPapMin}; the version we state here is from the latter paper.

\begin{theorem}[\cite{AHJKW,GeoPapMin}] \label{C Menger}
Let $G$ be a graph, and $S,T$ two subsets of $V(G)$. For every $K>0$, there is either
\begin{enumerate}[label=\rm{(\roman*)}]
\item \label{paths} two paths joining $S$ to $T$ at distance at least $\mu K$ in $G$ from each other, where $\mu\geq 1/544$ is a universal constant, or
\item \label{S} is a set $U \subset V(G)$ with $\rad(U)\leq K$ such that $G - U$ contains no $S$--$T$ path. 
\end{enumerate}
\end{theorem}

\section{Construction of the NSS graph} \label{sec:constr NSS}

For the proof of Theorem~\ref{thm:Counterexample:Octahedron}, we need to construct for every $M,A \in \N$ some graph $G$ that is not $(M,A)$-quasi-isometric to a graph with no $K_7$ minor. As already mentioned in the introduction, these graphs $G$ will be the graphs from~\cite{NgScSeCou}. We recall their construction in what follows (see Figure~\ref{fig:NSSGraph}).
\medskip

\begin{figure}[ht]
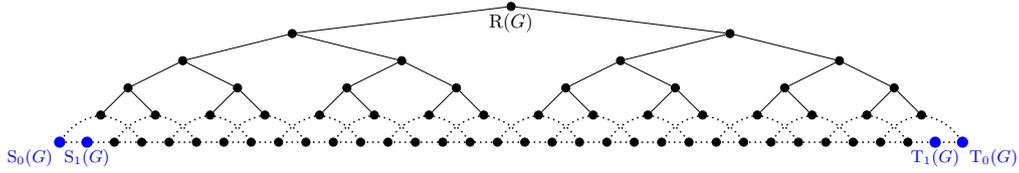

    \centering
    \include{NSSGraph}
    \vspace{-2em}
    \caption{Depicted is the graph $G_{5,d}$. The dotted curves represent paths of length $d+1$.} \label{fig:NSSGraph}
\end{figure}

The \defi{NSS graph} of \defi{height} $k$ and \defi{division} $d$, denoted by \defi{$G_{k,d}$}, is defined inductively. For $k = 1$ and $d \in \N$, $G:=G_{1,d}$ is the graph obtained as follows:

\begin{enumerate}
    \item start with a path on four vertices, labelled $S_0,S_1,T_1,T_0$ in that order, 
    \item add a new vertex $R$ adjacent to $S_0$ and $T_0$, and finally
    \item subdivide each edge into a path of length $d$.
\end{enumerate}

Define $\defim{\SO(G)}:=S_0$, $\defim{\SI(G)}:=S_1$, $\defim{\TO(G)}:=T_0$, $\defim{TI(G)}:=S_1$, $\defim{\RR(G)}:=R$. We call $\RR(G)$ the \defi{root} of $G$. We will maintain the property that for each $k\geq 1$, if $G:=G_{k,d}$, then $\SO(G)$, $\SI(G)$, $\TO(G)$, $\SI(G)$, $\RR(G)$ are defined and they are distinct vertices of~$G$.

We now describe the induction step of the construction. Let $k \in \N$, and let $\Delta_0$ and $\Delta_1$ be copies of $G_{k-1,d}$. 
Then $G := G_{k,d}$ is the graph obtained as follows:

\begin{enumerate}
    \item[(1)] take the disjoint union of $\Delta_0$ and $\Delta_1$,
    \item[(2)] identify $\TI(\Delta_0)$ with $\SO(\Delta_1)$ and identify $\TO(\Delta_0)$ with $\SI(\Delta_1)$, 
    \item[(3)] delete one of the two paths of length $d$ between the two identified vertices, and finally,
    \item[(4)] add a new vertex $R$ adjacent to $\RR(\Delta_0)$ and $\RR(\Delta_1)$.
\end{enumerate}

Set $\SO(G):=\SO(\Delta_0)$, $\SI(G):=\SI(\Delta_0)$, $\TO(G):=\TO(\Delta_1)$, $\TI(G):=\TI(\Delta_1)$, $\RR(G):=R$ (see Figure~\ref{fig:NSSGraph}). 
Moreover, let $\SSS(G) := \{\SO(G), \SI(G)\}$ and $\TTT(G) := \{\TO(G), \TI(G)\}$ and $\RTT(G) := \TTT(G) \cup \{\RR(G)\}$.
We call the set 
\[
\defim{\Adh(G)} := \{\SO(G),\SI(G),\TO(G),\TI(G),\RR(G)\} = \SSS(G) \cup \RTT(G)
\]
the \defi{adhesion} of $G_{k,d}$ (the name will make sense later). 

We also define $\defim{\SP(G)} :=\Delta_0$ and $\defim{\TP(G)} :=\Delta_1$. 
We remark that $\SSS(\SP(G)) = \SSS(G)$ and $\TTT(\SP(G)) = \SSS(\TP(G))$ and $\TTT(\TP(G)) = \TTT(G)$ (see Figure~\ref{fig:NSSGraph:2}). 
\smallskip 

\begin{figure}[ht]
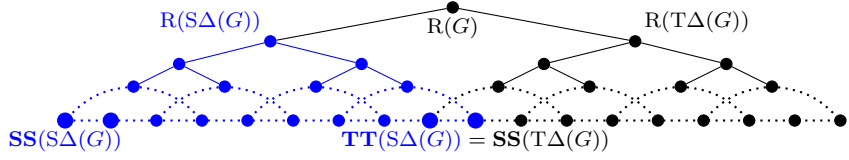

    \centering
    \include{NSSGraph_Subpyramid}
    \vspace{-2em}
    \caption{The blue subgraph of $G := G_{4,d}$ is the trigon $\SP(G)$.} \label{fig:NSSGraph:2}
\end{figure}

If $G=G_{1,d}$, then the \defi{bottom path} of $G$ is the (unique) $\SO(G)$--$\TO(G)$ path that avoids $\RR(G)$.
If $G=G_{k,d}$ for some $k \geq 2$, then the \defi{bottom path} of $G$ is the (unique) $\SO(G)$--$\TO(G)$ path that only uses vertices of $G$ that are only incident with subdivided edges. 
We let \defi{$B(G)$} be the (subdivided) binary tree contained in $G$, i.e.\ $B(G)$ is the subgraph that remains after removing the subdivision vertices of the bottom path from $G$. 
For every $x,y \in V(G)$ that are not subdivision vertices there is a unique $x$--$y$ path in $G$ that is internally disjoint from the bottom path of $G$, and we call it a \defi{tree path}.

For every vertex $x$ of the binary tree $B(G)$ that is neither a leaf nor a subdivision vertex, there is a (unique) copy of some NSS graph in $G$ (of smaller height than $G$ if $x \neq \RR(G)$) whose root is $x$. For example, if $x$ is the `left' successor of $\RR(G)$, then this graph is $\SP(G)$. We call these graphs \defi{trigons} of~$G$. 

Since every trigon is isomorphic to a (unique) NSS graph, all notions defined for NSS graphs, such as e.g.\ $\SO(.)$ and $\Adh(.)$, are also defined for trigons. 
Observe that given a (proper) trigon $\Delta$ of $G$, its adhesion $\Adh(\Delta)$ separates $\Delta$ and $G-V(\Delta)$. 
We call the subgraph $G - (V(\Delta)\setminus \Adh(\Delta))$ of $G$ the \defi{exterior} \defi{$\Ext(\Delta,G)$} of $\Delta$ in $G$. If the underlying graph $G$ is clear from the context, then we abbreviate $\Ext(\Delta) := \Ext(\Delta, G)$. 

For later use, let us remark that given any trigon $\Delta$ of $G$ with $\SSS(\Delta) = \SSS(G)$, its root $\RR(\Delta)$ together with $\TTT(\Delta)$ separate $\Delta$ from $G-V(\Delta)$.
\medskip

In the remainder of this section we collect some properties of the NSS graphs that we need later.

\begin{lemma}[\rm{\cite[2.1]{NgScSeCou}}] \label{lem:NSSGraph:NoSmallSep}
    Let $\ell \geq 1$, and let $G:=G_{k,d}$ for some $k > 2\ell + 2$ and $d > 2\ell$. If $X \subseteq V(G)$ with $|X|\leq 2$, then there is a path $P$ in $G$ between $\SSS(G)$ and $\TTT(G)$ such that $d_G(X,P) > \ell$.
\end{lemma}

\subsection{\texorpdfstring{$\SSS(G)$}{SS(G)}--\texorpdfstring{$\TTT(G)$}{TT(G)} Linkages} \label{sec link}

In this subsection we introduce some terms and observations that will be used in the proof of \Tr{thm css}, only.

Let $G:=G_{k,d}$ for some $k,d \in\N$. A \defi{linkage} of $G$ is a pair of disjoint paths $P,Q$ from $\SSS(G)$ to $\TTT(G)$. If one of $P,Q$ joins $\TI(G)$ with $\SO(G)$ (and thus the other one joins $\TO(G)$ with $\SI(G)$), we call $P,Q$ a \defi{crossing linkage}, otherwise we call it a \defi{nested linkage} (see Figure~\ref{fig:Linkages}).

\begin{figure}[ht]
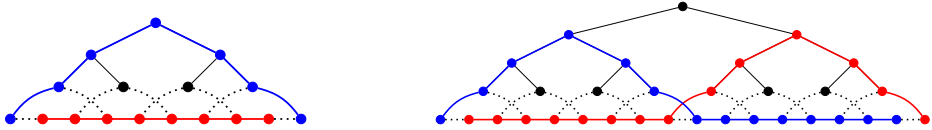

    \centering
    \begin{subfigure}[b]{0.4\linewidth} 
        \centering
        \include{LinkagesNested}
    \end{subfigure}
    \hspace{1em}
    \begin{subfigure}[b]{0.55\linewidth}
        \centering
        \include{LinkagesCrossing}
    \end{subfigure}
    \vspace{-2em}
    \caption{A nested linkage (left) and a crossing linkage (right).}
    \label{fig:Linkages}
\end{figure}

\begin{observation} \label{obs NL}
    For every $k,d \in \N$, the NSS graph $G_{k,d}$ has a nested linkage $P,Q$ with $d_G(P,Q)=d$. \qed
\end{observation}

\noindent Indeed, such a linkage is shown in Figure~\ref{fig:Linkages}.

\begin{lemma} \label{Lem cross}
Let $\Delta$ be a trigon of an NSS graph~$G$, and let $\Delta_0,\Delta_1$ be its children trigons. Let $P,Q$ be a linkage of $\Delta$ that avoids $\RR(\Delta)$. 
Then 
\begin{enumerate}[label=\rm{(\roman*)}]
    \item \label{li} for each $i\in \{0,1\}$ the paths $P^i := P \cap \Delta_i$ and $Q^i := Q \cap \Delta_i$ form a linkage of~$\Delta_i$; 
    \item \label{lii} at least one
of the three  linkages $P,Q$ and $P^0,Q^0$ and $P^1,Q^1$ is crossing.
\end{enumerate}
\end{lemma}

\begin{proof}
    \noindent \ref{li}: Let $W^i := W \cap \Delta_i$ for $W \in \{P,Q\}$ and $i \in \{0,1\}$. Since $\{\RR(\Delta)\} \cup \TTT(\Delta_0)$ separates $\Delta_0$ and $\Delta_1$ in $\Delta$, and because avoid $\RR(\Delta)$, the subpaths $P^0, Q^0$ and $P^1, Q^1$ are linkages of $\Delta_0, \Delta_1$, respectively. 
    \medskip

    \noindent \ref{lii}: Without loss of generality let $P$ start in $\SO(\Delta)$. If both $P^0,Q^0$ and $P^1,Q^1$ are nested, then $P^0$ ends in $\TO(\Delta_0)$ by the definition of nested. Since $\TO(\Delta_0) = \SI(\Delta_1)$, it follows that $P^1$ ends in $\TI(\Delta_1)$. Hence, $P$ starts in $\SO(\Delta)$ and ends in $\TI(\Delta)$, so $P,Q$ is crossing.
\end{proof}

\subsection{More about paths in the NSS graph}

In this subsection we make further observations about NSS graphs that will be used in \Sr{sec O fat}, only.

Recall that Nguyen, Scott and Seymour~\cite{NgScSeCou} constructed the NSS graphs as a counterexample to the `coarse Menger conjecture' \cite{AHJKW,GeoPapMin}. For this, one of the two main properties of $G:=G_{k,d}$ was that whenever $P,Q$ are paths between $\SSS(G)$  and $\TTT(G)$, then either $P,Q$ have distance at most $2$ from each other, or one of them contains the root $\RR(G)$. More precisely

\begin{lemma}{\rm{\cite[2.2]{NgScSeCou}}} \label{lem:TwoPathsInNSSGraph}
For every $k \geq 1$ and $d \in \N$, if $P,Q$ are paths in $G := G_{k,d}$ between $\SSS(G)$ and $\TTT(G)$, then either
\begin{enumerate}[label=\rm{(\roman*)}]
    \item $P,Q$ have distance at most $2$ from each other, or
    \item one of $P,Q$ is the (unique) path of $B(G)$ between $\SO(G)$ and $\TO(G)$.
\end{enumerate}
\end{lemma}

We also need the following variant of Lemma~\ref{lem:TwoPathsInNSSGraph}:

\begin{corollary}
\label{cor:TwoPAthsInNSSGraph:TopAndRight}
    For every $k \geq 2$ and $d \in \N$, if $P,Q$ are paths in $G := G_{k,d}$ between $\SSS(G)$ and $\RTT(G)$, then either
\begin{enumerate}[label=\rm{(\roman*)}]
    \item \label{itm:SStoRTT:i} $P,Q$ have distance at most $2$ from each other, or
    \item \label{itm:SStoRTT:ii} one of $P,Q$ is the (unique) path in $B(G)$ between $\SO(G)$ and $\RR(G)$ or between $\SO(G)$ and $\TO(G)$. In particular, it contains $\RR(G)$.
\end{enumerate}
\end{corollary}

\begin{proof}
    We proceed by induction on the height $k$. For $k = 1$ the assertion is easy to check because $G_{1, d}$ is just a subdivided $K_3$. 
    So we may assume that $k > 1$ and that the assertion holds for all $k' < k$. 
    Set $G := G_{k,d}$, let $P,Q$ be given such that (without loss of generality) $P$ starts in $\SO(G)$ (and hence $Q$ starts in $\SI(G)$), and assume that $d_G(P, Q) \geq 3$, i.e.\ \ref{itm:SStoRTT:i} does not hold. 

    If both $P,Q$ end in $\TTT(G)$, then \ref{itm:SStoRTT:ii} holds by Lemma~\ref{lem:TwoPathsInNSSGraph}. Hence, we may assume that one of $P,Q$ ends in $\RR(G)$. In particular, $P,Q$ are $\SSS(G)$--$\RTT(G)$ paths.
    Now suppose for a contradiction that \ref{itm:SStoRTT:ii} does not hold either, i.e.\ $P$ is not the (unique) path in $B(G)$ between $\SO(G)$ and $\RR(G)$. 
    Let $v_0v_1$ be the first edge on the (unique) path in $B(G)$ from $\SO(G)$ to $\RR(G)$ which is not contained in $P$. 
    Let $\Delta_0, \Delta_1$ be the unique trigons of $G$ with $\RR(\Delta_0) = v_0$ and $\RR(\Delta_1) = v_1$. In particular, $\RR(\Delta_0) = v_0 \in V(P)$ and $\Delta_0 = \SP(\Delta_1)$. Since $\TTT(\Delta_0)$ and the edge $\RR(\Delta_0)\RR(\Delta_1)$ separate $\SSS(G)$ from $\TTT(G)$, both $P,Q$ meet $\TTT(\Delta_0)$. Since $d_G(P,Q) \geq 3$ and $\SO(G) \in V(P)$, it follows by Lemma~\ref{lem:TwoPathsInNSSGraph} that $\TO(\Delta_0), \RR(\Delta_0) \in V(P)$ and $\TI(\Delta_0) \in V(Q)$. 
    
    Now consider $\Delta'_0 := \TP(\Delta_1)$. As $\RTT(\Delta'_0)$ separates $\TO(\Delta_0) = \SO(\Delta'_0)$ and $\TI(\Delta_0) = \SI(\Delta'_0)$ from $\RTT$ in $\Ext(\Delta_0)$, 
    both $P,Q$ meet $\RTT(\Delta'_0)$. Applying the induction hypothesis to $\Delta'_0$ yields that $Q$, which starts in $\SO(\Delta'_0)$, contains $\RR(\Delta'_0)$. But then $d_G(P,Q) \leq d_G(\RR(\Delta_0), \RR(\Delta'_0)) = 2$, a contradiction. 
\end{proof}

A similar result as in Corollary~\ref{cor:TwoPAthsInNSSGraph:TopAndRight} is also true if $\Delta$ is a trigon of $G$ and the paths $P,Q$ are contained in $\Ext(\Delta)$ (see Figure~\ref{fig:PathsInNssGraph}):

\begin{lemma} \label{lem:TwoPathsInNSSGraph:ExteriorOfSubpyramid}
    For every $k \geq 2$ and $d \in \N$ and every trigon $\Delta$ of $G := G_{k,d}$, if $P,Q$ are paths in $\Ext(\Delta)$ between $\SSS(\Delta)$ and $\RTT(\Delta)$, then either
    \begin{enumerate}[label=\rm{(\roman*)}]
        \item $P,Q$ have distance at most $2$ from each other, or
        \item one of $P,Q$ is the (unique) path in $B(G)$ between either $\SI(\Delta)$ and $\TI(\Delta)$ or between $\SI(\Delta)$ and $\RR(\Delta)$. In particular, it has distance at most~$1$ from $\RR(\Delta)$.
    \end{enumerate}
\end{lemma}

\begin{figure}[ht]
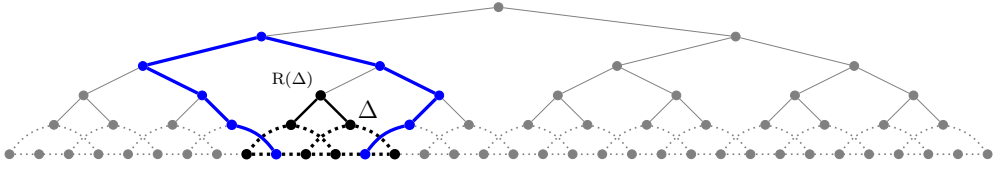

    \centering
    \include{PathsInNSSGraph}
    \vspace{-2em}
    \caption{Indicated in blue is the unique path in $B(G)$ between $\SI(\Delta)$ and $\TI(\Delta)$, where $\Delta$ is indicated in black.}
    \label{fig:PathsInNssGraph}
\end{figure}

\begin{proof}[Proof of Lemma~\ref{lem:TwoPathsInNSSGraph:ExteriorOfSubpyramid}]
    Let $k'$ be the height of $\Delta$. We proceed by induction on $n := k - k'$. For $n = 0$, we have $\Delta = G$, and hence there is no path in $G[\Adh(G)] = \Ext(G)$ between $\SSS(\Delta)$ and $\RTT(\Delta)$, so the assertion holds trivially. For $n = 1$ there are no two disjoint paths in $\Ext(\Delta)$ between $\SSS(\Delta)$ and $\RTT(\Delta)$ because either $\SO(\Delta)$ and $\SI(\Delta)$ or $\TO(\Delta)$ and $\TI(\Delta)$ are isolated vertices in $\Ext(\Delta)$.
    Hence, we may assume $n \geq 2$. Let $\Delta_1, \Delta_2$ be the (unique) trigons of~$G$ of height $k'+1$ and $k'+2$, respectively, that contain~$\Delta$. 
    We assume that \ref{itm:SStoRTT:i} does not hold, and show that then \ref{itm:SStoRTT:ii} must hold.
    \smallskip
    
    Let us first assume that $\Delta = \SP(\Delta_1)$. Set also $\Delta' := \TP(\Delta_1)$.
    Without loss of generality let $P$ start in $\SI(\Delta)$. Since $\RTT(\Delta_1)$ separates $\RTT(\Delta)$ in $\Ext(\Delta)$ from $\SSS(\Delta)$, both $P,Q$ have to (re-)enter $\Delta_1$ through $\RTT(\Delta_1)$. 
    Let $P', Q'$ be the (unique) subpaths of $P, Q$, respectively, which start in $\SSS(\Delta)$, end in $\RTT(\Delta_1)$, and are internally disjoint from $\Delta_1$.
    Since the height of $\Delta_1$ is one larger than the height~$k'$ of $\Delta$, we may apply the induction hypothesis to $\Delta_1$, which yields that $P'$ is the (unique) path in $B(G)$ between $\SI(\Delta_1)$ and $\RR(\Delta_1)$ or $\TI(\Delta_1)$, and that $Q'$ ends in $\TTT(\Delta_1)$. In particular $\RR(\Delta_2) \in V(P)$, which implies that $Q$ does not end in $\RR(\Delta)$ as $d_G(\RR(\Delta), \RR(\Delta_2)) = 2$.
    
    If $P'$ ends in $\RR(\Delta_1)$ and the next edge on $P$ is $\RR(\Delta_1)\RR(\Delta)$, then $P$ is the (unique) path in $B(G)$ between $\SI(\Delta)$ and $\RR(\Delta)$ as desired. Thus, we may assume that if $P'$ ends in $\RR(\Delta_1)$, then $P$ contains the edge $\RR(\Delta_1)\RR(\Delta')$, and the next vertex on $P$ is $\RR(\Delta')$.
    This implies that in either case the first vertex of $P$ contained in $\Delta'$ is $\RR(\Delta')$ or $\TI(\Delta')$. 

    Let us first assume that $P$ enters $\Delta'$ through $\RR(\Delta')$ (and recall that $Q$ enters $\Delta'$ through $\TTT(\Delta')=\TTT(\Delta_1)$ as shown earlier).
    Since $\RTT(\Delta')$ has size $3$ and each of its vertices has only one neighbour in $\Ext(\Delta')$, the subgraphs $P \cap \Delta'$ and $Q \cap \Delta'$ of $P,Q$ are paths (which have one endvertex in $\RTT(\Delta')$ and the other in $\TTT(\Delta) = \SSS(\Delta')$).
    By applying Corollary~\ref{cor:TwoPAthsInNSSGraph:TopAndRight} to $\Delta'$, it follows that $P\cap \Delta'$ is the (unique) path in $B(G)$ between $\SO(\Delta') = \TI(\Delta)$ and $\RR(\Delta')$, and hence $P$ is the (unique) path in $B(G)$ between $\SI(\Delta)$ and $\TI(\Delta)$, so \ref{itm:SStoRTT:ii} holds.

    Thus, we may assume that $P$ enters $\Delta'$ through $\TI(\Delta)$ (and thus $Q$ enters $\Delta'$ through $\TO(\Delta')$). 
    We now again analyse $P \cap \Delta'$ and $Q \cap \Delta'$, which are again paths for the same reasons as above. First, observe that their other endvertices (the ones not in $\TTT(\Delta')$) have to be contained in $\SSS(\Delta') \cup \RR(\Delta')$. 
    By the symmetry of the NSS graph, we may apply Corollary~\ref{cor:TwoPAthsInNSSGraph:TopAndRight} to a flipped version of $\Delta'$, which yields that $Q \cap \Delta'$ is the (unique) path in $B(G)$ between $\TO(\Delta')$ and either $\SO(\Delta')$ or $\RR(\Delta')$. In both cases, $Q$ contains $\RR(\Delta')$, and hence $d_G(P,Q) \leq d_G(\RR(\Delta_2), \RR(\Delta')) = 2$, a contradiction. 
    \smallskip

    Let us now assume that $\Delta = \TP(\Delta_1)$. (This case uses similar ideas to the previous one.) 
    Set $\Delta' := \SP(\Delta_1)$. 
    Without loss of generality let $P$ start in $\SI(\Delta)$. Since $\SSS(\Delta') \cup \{\RR(\Delta')\}$ separates $\SSS(\Delta)$ in $\Ext(\Delta)$ from $\RTT(\Delta)$, both $P,Q$ have to exit $\Delta'$ through $\SSS(\Delta') \cup \{\RR(\Delta')\}$.
    In particular, since $\SSS(\Delta') \cup \{\RR(\Delta')\}$ has size $3$ and each of its vertices has only one neighbour in $\Ext(\Delta')$, both $P \cap \Delta'$ and $Q \cap \Delta'$ are paths (between $\SSS(\Delta) = \TTT(\Delta')$ and $\SSS(\Delta') \cup \RR(\Delta')$). 
    By the symmetry of the NSS graph, we can apply Corollary~\ref{cor:TwoPAthsInNSSGraph:TopAndRight} to a flipped version of $\Delta'$, which implies that $P \cap \Delta'$ is the unique path in $B(G)$ between $\TO(\Delta') = \SI(\Delta)$ and $\RR(\Delta')$ or $\SO(\Delta')$ (so $P \cap \Delta'$ is a path in $\Delta'$ from `right' to `left' or `top'); in particular, $\RR(\Delta') \in V(P)$. 
    
    Let us first assume that $P \cap \Delta'$ ends in $\SO(\Delta')$. Then $Q\cap \Delta'$ ends in $\SI(\Delta')$. Hence, $P\cap \Ext(\Delta_1)$ and $Q \cap \Ext(\Delta_1)$ are paths, and they go from $\SSS(\Delta_1) = \SSS(\Delta')$ to $\RTT(\Delta_1)$. Since the height of $\Delta_1$ is one larger than the height~$k'$ of $\Delta$, we may apply the induction hypothesis. But then $Q\cap \Ext(\Delta_1)$, which starts in $\SI(\Delta_1)$, contains $\RR(\Delta_2)$, and hence has distance at most~$2$ from $\RR(\Delta') \in P$, a contradiction.
    
    Thus, $P \cap \Delta'$ ends in $\RR(\Delta')$, and it is the unique path in $B(G)$ between $\TO(\Delta') = \SI(\Delta)$ and $\RR(\Delta')$. Then the next vertex on $P$ after $\RR(\Delta')$ has to be $\RR(\Delta_1)$. If $P$ also contains the edge $\RR(\Delta_1)\RR(\Delta)$ and thus ends in $\RR(\Delta)$, then $P$ is the unique $\SI(\Delta)$--$\RR(\Delta)$ path in $B(G)$ as desired. 
    So we may assume that $P$ does not contain this edge; in particular, it ends in $\TTT(\Delta)$ (because it can no longer reach $\RR(\Delta)$ by only using vertices in $\Ext(\Delta)$). Similarly, $Q$ must end in $\TTT(\Delta)$ (because $\RR(\Delta_1) \in P$, and since $P$ and $Q$ are disjoint, so $Q$ cannot reach $\RR(\Delta)$ by only using vertices in $\Ext(\Delta)$). 
    Hence, $P \cap \Ext(\Delta_1)$ and $Q \cap \Ext(\Delta_1)$ are paths from $\SSS(\Delta_1) \cup \{\RR(\Delta_1)\}$ to $\TTT(\Delta_1)$. By the symmetry of the NSS graph and because the height of $\Delta_1$ is one larger than the height~$k'$ of~$\Delta$, we can apply the induction hypothesis to a flipped version of $G$, which implies that $P \cap \Ext(\Delta_1)$ is the (unique) path in $B(G)$ between $\RR(\Delta_1)$ and $\TI(\Delta_1)=\TI(\Delta)$. Hence, $P$ is the unique path in $B(G)$ between $\SI(\Delta)$ and $\TI(\Delta)$, so \ref{itm:SStoRTT:ii} holds. 
\end{proof}

\section{Counterexamples to Conjecture~\ref{conj fat min}} \label{sec:Counterex:Proof}

In this section we prove Theorems~\ref{thm:Counterexample:Octahedron} and~\ref{thm:FatK4t} assuming Theorems~\ref{thm:css:finite} and~\ref{thm:NSSGraph:NoFatO:Intro} (which we will prove in the next two sections). 
We start by finding 2-fat models of $K_7$ and~$K_{4,t}$.

\begin{theorem} \label{thm:NSSGraph:K7}
    $K_7$ is a $2$-fat minor of $G_{k,d}$ for every $k \geq 13$ and $d \geq 2$.
\end{theorem}

\begin{proof}
    Figure~\ref{fig:K7} shows a ($0$-fat) model of $K_7-\{16,26\}$ in $G_{7,d}$. 
    To obtain a model of $K_7$ in $G_{9,d}$, one  repeats the idea of the bottom picture, except that this time it is the sixth branch set instead of the fifth that uses the root to connect to the first and second branch set. 

    We can now deduce the existence of a 2-fat model of $K_7$ abstractly, by combining the $0$-fat model that we just found in $G_{9,d}$ with the fact that $G_{9,d}$ is a 2-fat minor of some $G_{k,d'}$ by Theorem~\ref{thm:css:finite}. But we can also adapt the above construction to explicitly construct a $2$-fat $K_7$ minor in $G_{13,d}$: by increasing the height of the NSS graph in the top picture by $2$, we can make the model of $K_7-\{15,25,16,26\}$ $2$-fat, and if we increase the height of the NSS graph in the bottom picture by $1$ (with respect to the gray areas), then we can also make the model of $K_7-\{16,26\}$ $2$-fat). (See also Figure~\ref{fig:K4t:2Fat}.)
\end{proof}

\begin{figure}[ht]
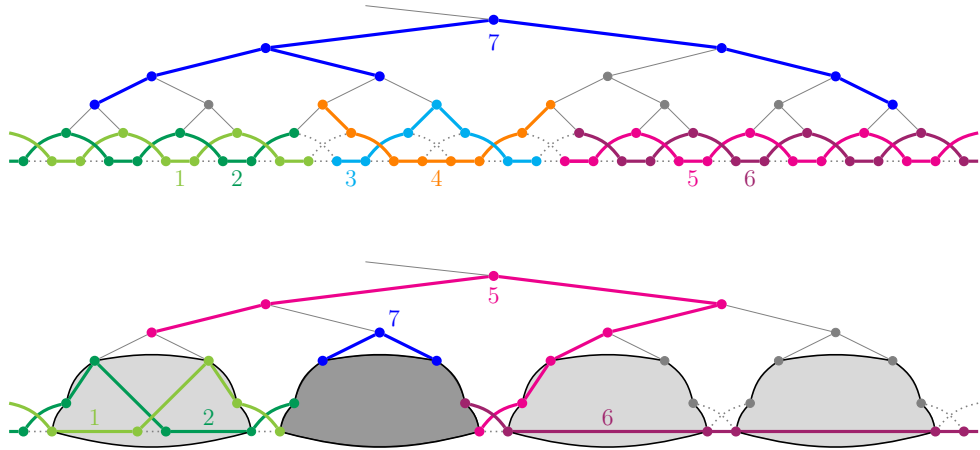

    \centering
    \begin{subfigure}{1\linewidth}
        \centering
        \include{K7}
    \end{subfigure}
    \begin{subfigure}{1\linewidth}
        \centering
        \include{K7_2}
    \end{subfigure}
    \vspace{-2em}
    \caption{Depicted are the branch sets of a model of a $K_7-\{16,26\}$ in $G_{7,d}$. The dark gray area in the bottom picture represents the graph from the top picture.}
    \label{fig:K7}
\end{figure}

\begin{theorem} \label{thm:NSSGraph:K4t}
    $K_{4,t}$ is a $2$-fat minor of $G_{k,d}$ for every $t \in \N$ and $k \geq 3+2t$, $d \geq 2$.
\end{theorem}

\begin{proof}
    Figure~\ref{fig:K4t} shows a model of $K_{4,2}$ in $G_{5,d}$ where the branch sets of the vertices on the side of $K_{4,t}$ of size $4$ have distance at least $2$ from each other. It is straightforward to extend the picture to find a model of $K_{4,t}$ in $G_{1+2t,d}$ for every $t \geq 3$ with the same property: we increase the height by 2 for each new branch set we want to introduce, we prolong the blue  and green branch sets to the left and right respectively, and form a new branch set using the root similar to the pink one. 

    As in Theorem~\ref{thm:NSSGraph:K7}, we can now obtain a 2-fat $K_{4,t}$ model via Theorem~\ref{thm:css:finite}. Alternatively, if we use $G_{3+3t,d}$ instead of $G_{1+2t,d}$, then it is easy to adapt the above construction so that all branch sets of the model of $K_{4,t}$ are pairwise at least $2$ far apart. Indeed, by modifying this model of $K_{2,t}$ in $G_{3+3t,d}$ as indicated in Figure~\ref{fig:K4t:2Fat}, we obtain a model of $K_{2,t}$ that is $2$-fat. 
\end{proof}

\begin{figure}[ht]
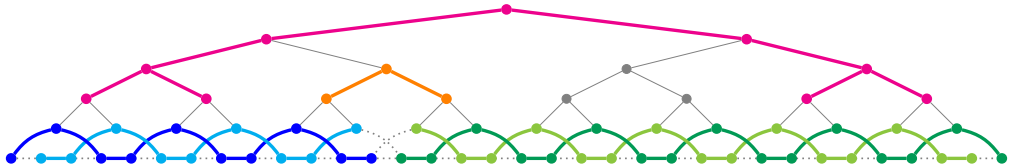

    \centering
    \include{K4t}
    \vspace{-2em}
    \caption{Depicted are the branch sets of a model of $K_{4,2}$ in $G_{5,d}$. The blue and green sets are the branch sets corresponding to the four vertices on `$4$ side' of $K_{4,2}$, and the orange and pink sets are the branch sets corresponding to the two vertices on the `$2$~side'.}
    \label{fig:K4t}
\end{figure}

\begin{figure}[ht]
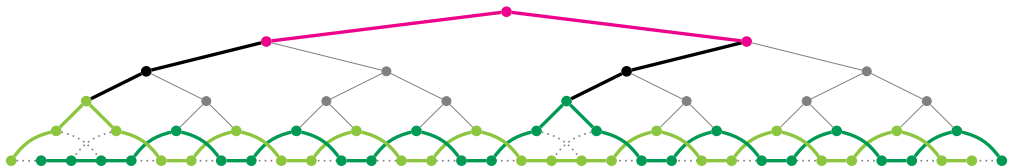

    \centering
    \include{K4tFat}
    \vspace{-2em}
    \caption{A part of a $2$-fat model of $K_{4,t}$. Indicated in black are the branch paths between a (pink) branch set of the `$t$ side' of $K_{4,t}$ and two (green) branch sets of the `$4$~side' of $K_{4,t}$.}
    \label{fig:K4t:2Fat}
\end{figure}

Given a graph $G$ and $K \in \N$, the \defi{$K$-th power $G^K$} is obtained from~$G$ by adding an edge between every pair of distinct vertices at distance at most $K$. 
We will use the following well-known observation about graph powers \cite[Proof of Theorem~3]{DHIM}:

\begin{theorem} \label{thm:FromKFatTo3Fat}
    Let $K \in \N$, let $J$ be a graph, and let $G$ be a graph that does not contain~$J$ as $K$-fat minor. Then $G^K$ does not contain $J$ as a $3$-fat minor. 
    
    Moreover, $G$ is $(K,0)$-quasi-isometric to $G^K$.
\end{theorem}

\begin{proof}[Proof of Theorems~\ref{thm:Counterexample:Octahedron} and ~\ref{thm:FatK4t}]

Let $G_\infty:=\dot{\bigcup}_{\nin} G_{n,n}$ denote the infinite NSS graph. We will show that $G_\infty^9$ has the desired properties. 

Recall that $G_\infty$ does not contain $O$ as a $9$-fat minor by Theorem~\ref{thm:NSSGraph:NoFatO},  and so $G_\infty^9$ does not contain $O$ as a $3$-fat minor by \Tr{thm:FromKFatTo3Fat}. Since $O$ is a minor of $K_{4,4}$, we deduce that $G_\infty^9$ does not contain $K_{4,4}$ as a $3$-fat minor.

Now let $H$ be a graph such $G_\infty^9$ is $(M,A)$-quasi-isometric to $H$ for some $M,A \in \N$. Then $G$ is $(9M, A)$-quasi-isometric to $H$ by Theorem~\ref{thm:FromKFatTo3Fat}. By Theorem~\ref{thm:css:finite}, $H$ contains $G_\infty$ as a $2$-fat minor. Since $K_7$ is a minor of $G_\infty$ by Theorem~\ref{thm:NSSGraph:K7}, we deduce that $K_7$ is a $2$-fat minor of $H$ as claimed by Theorem~\ref{thm:Counterexample:Octahedron}. Similarly, $K_{4,t}$ is a minor of $G_\infty$ by Theorem~\ref{thm:NSSGraph:K4t}, and so it is a $2$-fat minor of $H$ as claimed by Theorem~\ref{thm:FatK4t}.
\end{proof}

Our proof in fact yields the following more general statement that could be useful in finding additional counterexamples to \Cnr{conj fat min}. Let again $G_\infty :=\dot{\bigcup}_{\nin} G_{n,n}$ denote the infinite NSS graph, and $G_\infty^9$ its $9$-th power.

\begin{corollary} \label{cor:Counterexample:MainThm}
     $G_\infty^9$ does not contain $O$ as a $3$-fat minor, but every graph $H$ that is quasi-isometric to $G_\infty^9$ contains $G_\infty$, and hence any minor of $G_\infty$, as a $2$-fat minor. 
\end{corollary}

This implies, for example, that $G_{k,d}$ is itself \bad\ \fe\ $k$ \leth\ $O \preceq G_{k,d}$. Likewise, it implies that $G_\infty$ is \bad.

\section{Coarse self-similarity of the NSS graph} \label{sec css}

We say that a graph (family) $X$ is \defi{\css}, if $X \preceq Y$ holds for every graph (family) $Y$ quasi-isometric to $X$.
As above, we define $G_\infty:=  \dot{\bigcup}_{\nin} G_{n,n}$. 
The aim of this section is to prove that $G_\infty$ is \css,  
\comment{
\begin{theorem} \label{thm css}
    The NSS graph $G_\infty$ is \css. 
\end{theorem}
}
and in fact strengthen this fact by requiring 2-fatness as in Theorem~\ref{thm:css:finite}, which we restate here for convenience: 

\begin{theorem} \label{thm:css:finite:maintext}
    For every $M \geq 1$, $A \geq 0$ and $k, d \in \N$, there exist $K, D \in \N$ such that the following holds for all $k'\geq K$ and $d'\geq D$: If $G_{k',d'}$ is $(M,A)$-quasi-isometric to a graph $H$, then $H$ contains $G_{k,d}$ as a $2$-fat minor.
\end{theorem}

\begin{proof}
    Let $M \geq 1$ and $A \geq 0$ be arbitrary but fixed throughout the proof. Let $r := r(M,A)+1$ be given by Lemma~\ref{lem:QIPreservesSeps} for $M,A$. 
    We start with some definitions. 
    
    Let $\Delta$ be some (arbitrary) trigon of some $G=G_{k',d'}$, and let $\varphi$ be an $(M,A)$-quasi-isometry from~$G$ to a graph $H$. Set $\nabla := H[B_H(\varphi(\Delta), r)]$. We call $\nabla$ the \defi{trigon of $H$ corresponding to $\Delta$}.
    It is straightforward to extend the notations $\SO(\Delta), \SI(\Delta)$, $\TO(\Delta),\TI(\Delta)$, $\RR(\Delta)$ to $\nabla$, except that now they denote balls in $H$ of radius $r$ instead of single vertices; for example, $\SO(\nabla):= B_H(\varphi(\SO(\Delta)), r)$. 
    We call these sets the \defi{adhesion balls} of $\nabla$.
    We also let $\SSS(\nabla) := \SO(\nabla) \cup \SI(\nabla)$ and $\TTT(\nabla) := \TO(\nabla) \cup \TI(\nabla)$ and $\Adh(\nabla) := \SSS(\nabla) \cup \TTT(\nabla) \cup \{\RR(\nabla)\}$.
    We remark that whenever we use these notations later in the proof, the map $\varphi$ and the graph $H$ will be clear from the context.

    We say that a model $(\cu,\ce)$ of $G_{k,d}$ in $\nabla$ is \defi{natural} (for $\nabla$) if each adhesion ball of~$\nabla$ is contained in the branch set of $(\cu,\ce)$ corresponding to the `right' adhesion vertex of~$G_{k,d}$; for example $\SO(\nabla) \subseteq U_{\SO(G_{k,d})}$.

    We remark for later use that if $(\cu, \ce)$ is natural, then all its branch paths have distance at least $2$ from the neighbourhood of $H-\nabla$. Indeed, since every branch path is internally contained in $\nabla-\Adh(\nabla)$, this follows from Lemma~\ref{lem:QIPreservesSeps} as $r$ is one larger than $r(M,A)$. 
    \medskip

    We proceed with the proof of the statement. We will provide concrete values for the constants $K,D$ that satisfy our requirements, but the reader can choose to ignore these values; what matters is that we can choose $K,D$ large in comparison to $k,d,M$ and~$A$, more concretely, large enough compared to values that come out of applications of Lemmas~\ref{lem:QIPreservesSeps}, ~\ref{lem:NSSGraph:NoSmallSep} and Theorem~\ref{C Menger}. The values that we obtain are
    \begin{align} 
        K &= K_k := \max\{2\ell+3, \hspace*{.25cm} K_{k-1}+\lceil M(2r+A+3)\rceil\}, \text{ and} \label{eq:K}\\ 
        D & =D_d:= \max\{2\ell+1, \hspace*{.25cm} \lceil M(4d+2r+A)\rceil\} \label{eq:D}
    \end{align}
    with $\ell := \lceil \max\{M\lceil M\rceil(4r+2A+5)+2r+A, \hspace*{.25cm} 544\cdot (2r+1)\} + r(M,3AM)\rceil$ where $r(M,3AM)$ is provided by Lemma~\ref{lem:QIPreservesSeps}.
    We will later apply Lemma~\ref{lem:NSSGraph:NoSmallSep} with parameter~$\ell$, and hence our choice of $K,D$ ensures that $\SSS(G_{K,D})$ and $\TTT(G_{K,D})$ cannot be separated by two balls of radius at most $\ell$.
    \smallskip
    
    We claim that the following holds  for every $k' \geq K$ and $d' \geq D$:
    \labtequtag{indhyp}{For every trigon $\Delta$ of $G_{k',d'}$ of height  $h\geq K$ and for every $(M,A)$-quasi-isometry $\varphi$ from $G$ to a graph $H$, the trigon $\nabla$ of $H$ corresponding to $\Delta$ contains a $2$-fat, natural, model of $G_{k,d}$.}{$\ast$}
    This immediately implies, and strengthens, our statement. 
    \smallskip
    
    We prove \eqref{indhyp} by induction on $k$. For $k=1$, we obtain the desired model of $G_{1,d}$ in~$H$ as follows. 
    For every $x \in \SSS(G_{1,d}) \cup \TTT(G_{1,d})$ we define its branch set as $U_x := B_H(\varphi(x'), r)$ where $x'$ is the adhesion vertex of $\Delta$ `in the same position' as~$x$; e.g.\ $x'=\SO(\Delta)$ if $x=\SO(G_{1,d})$. Moreover, for $x = \RR(G_{1,d})$, we let $U_x := B_H(\varphi(B), r)$ where $B$ is the component of $\Delta$ containing $\RR(\Delta)$ after removing the interiors of all dotted paths (i.e.\ all subdivision vertices) from $\Delta$.
    By Lemma~\ref{lem:QIPreservesConn}, every $U_x$ is connected. 

    Moreover, by Lemma~\ref{lem:QIPreservesConn}, for each of the five subdivided edges $Q$ of $G_{1,d}$, there exists a (unique) component of $\nabla - \bigcup_{x \in \Adh(G_{1,d})} U_x$ that has neighbours in both branch sets $U_x$ and $U_y$ of the endvertices $x,y$ of $Q$,
    and by Lemma~\ref{lem:QIPreservesSeps} (applied to $G_{k',d'}$ and~$H$) these five components are distinct. 
    Therefore, we may choose in each component a shortest path $W_{xy}$ from $U_x$ to $U_y$, where $x,y$ are the two endvertices of $Q$ in $G_{1,d}$. Since $d' \geq M(4d+2r+A)$ by \eqref{eq:D}, the adhesion balls of $\nabla$ have distance at least $4d$ from each other, and therefore the paths $W_{xy}$ are at least $4d$ long. As the $W_{xy}$ are shortest paths and thus induced, it is straightforward to divide them into $d-1$ branch sets and $d$ branch paths so that any two of them have distance at least~$2$ from each other and from $U_x, U_y$, unless they represent and incident vertex/edge pair.
    This yields the desired $2$-fat, natural, model of $G_{1,d}$ in~$\nabla$.
    \smallskip

    \begin{figure}[ht]
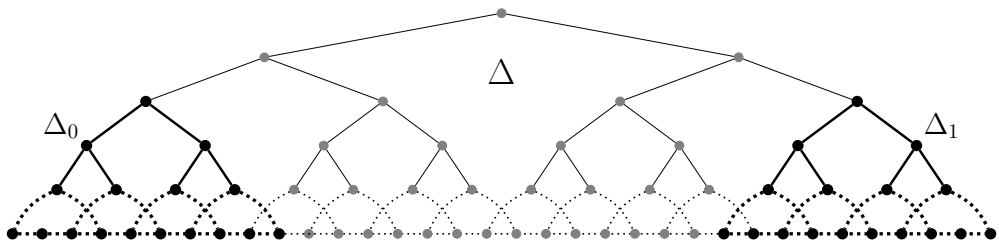

        \centering
        \include{CSSProof1}
        \vspace{-2em}
        \caption{The trigons $\Delta_0$ and $\Delta_1$ inside $\Delta$ in the proof of Theorem~\ref{thm:css:finite}.}
        \label{fig:CSS:Proof:1}
    \end{figure}
    For $k > 1$, we construct our model of $G_{k,d}$ in $H$, roughly speaking, as follows. We start by applying the induction hypothesis to a pair of subtrigons $\Delta_0,\Delta_1$ of $\Delta$ lying in the leftmost and rightmost extreme of $\Delta$ (see Figure~\ref{fig:CSS:Proof:1}) to obtain $2$-fat, natural models $M_0 := (\cu^0, \ce^0)$ and $M_1 := (\cu^1, \ce^1)$ of $G_{k-1,d}$. Recall that we inductively constructed $G_{k,d}$ from two copies of $G_{k-1,d}$ by identifying two pairs of vertices and adding a new vertex adjacent to their root vertices. We will use $M_0,M_1$ as our two copies of $G_{k-1,d}$, and extend them into a $2$-fat, natural $G_{k,d}$-model within $\nabla$ as follows (see Figure~\ref{fig:CSS:Proof:2}). 
    \begin{figure}[ht]
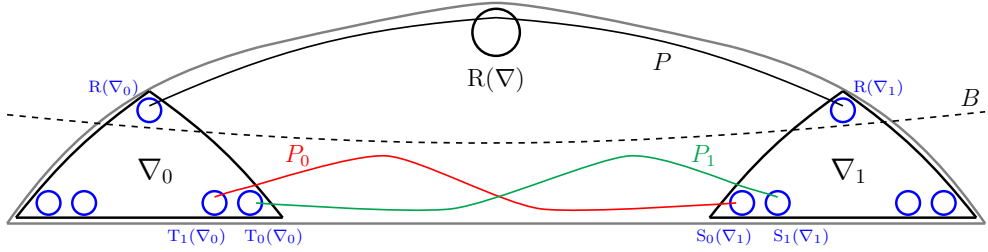

        \centering
        \include{CSSProof2}
        \vspace{-2em}
        \caption{A sketch of the model of $G_{k,d}$ in $\nabla$ in the proof of Theorem~\ref{thm:css:finite}. Indicated in blue are the adhesion balls of $\nabla_0, \nabla_1$. The paths $P_0,P_1$ are disjoint.}
        \label{fig:CSS:Proof:2}
    \end{figure}
 
    Firstly, we connect the root branch sets of $M_0,M_1$ by a path $P$ in $\nabla$ through $R(\nabla)$. Then, we find a pair of disjoint paths $P_0$ from $\TI(\nabla_0)$ to $\SO(\nabla_1)$ and $P_1$ from $\TO(\nabla_0)$ to $\SI(\nabla_1)$ in $\nabla$, which have distance at least $2$ from each other and also from $P$ and the rest of $\nabla_0, \nabla_1$. (This step is the most challenging part of our proof.) Next, we contract each $P_i$, thereby joining the branch set of $M_0$ corresponding to $\TO(G_{k,d})$ (respectively, $\TI(G_{k,d})$) to the branch set of $M_1$ corresponding to $\SI(G_{k,d})$ (respectively $\SO(G_{k,d})$). (For this note that $M_0,M_1$ are natural, and therefore for example $\TO(\nabla_0) \subseteq U^0_{\TO(G_{k-1,d})}$.) Finally, we form a new branch set containing $R(\nabla)$, and use edges of $P$ to connect it to the root branch sets of $M_0,M_1$. These new branch sets and paths, together with the branch sets and paths of $M_0,M_1$ that we did not modify, form the branch sets and paths of the desired model of $G_{k,d}$ in $\nabla$. 
    \medskip
    
    Let us make this construction precise. 
    We start the construction of our model by applying the induction hypothesis \eqref{indhyp} to the two subtrigons $\Delta_0, \Delta_1$ of $\Delta$ of height $h':= h-\lceil M(2r+A+3)\rceil$ lying in the leftmost and rightmost extreme of $\Delta$, as indicated in Figure~\ref{fig:CSS:Proof:1}. Note that $h' \geq  K_{k-1}$ by \eqref{eq:K}.
    In this way, we obtain $2$-fat, natural models $M_0 := (\cu^0, \ce^0)$ and $M_1 := (\cu^1, \ce^1)$ of $G_{k-1,d}$ in $\nabla_0$ and $\nabla_1$, respectively.
    Note that our choice of $d' > M(r+A+1)$ and $h > h'+M(2r+A+2)$ by \eqref{eq:K} and \eqref{eq:D} implies that $\nabla_0,\nabla_1$ have distance at least $2$ from each other. 
 
    Let $P^-$ be the unique $\RR(\Delta_0)$--$\RR(\Delta_1)$~path in the binary tree $B(\Delta)$. Note that $B_H(\varphi(P^-), M+A)$ is connected by Lemma~\ref{lem:QIPreservesConn}, 
    and therefore $B_H(\varphi(P^-),M+A)$ contains an $\RR(\nabla_0)$--$\RR(\nabla_1)$~path $P$; we choose a shortest such path, which implies that $P$ is induced, and that at most its first two vertices send edges to $\nabla_0$ and at most its last two vertices send edges to $\nabla_1$. Note that $P$ must intersect $\RR(\nabla)$ by Lemma~\ref{lem:QIPreservesSeps} (applied to $G_{k',d'}, H$ and $X := \Adh(\Delta) \cup \TTT(\SP(\Delta))$ where we used that $P^-$ has distance at least~$h'$ from $X \setminus \RR(\Delta)$, and hence $P \subseteq B_H(\varphi(P^-), M+A)$ is disjoint to $B_H(\varphi(X), r) \setminus \RR(\nabla)$ as $h' \geq M(M+A+r+1)+A$).
    Moreover, since $d_G(\RR(\Delta_0), \RR(\Delta_1)) \geq \min\{2d', 2(h-h')\} \geq 2M(r+A+3)$ by \eqref{eq:K} and \eqref{eq:D}, we deduce that $P$ has length at least $6$.
    
    Let $q:=\lceil M\lceil M\rceil(4r+2A+5)+2r+A\rceil$ and $B:=B_H(\RR(\nabla), q)$. Roughly speaking, $q$ is large with respect to $M, A$, but $K,D$ are chosen after $q$, and they are even larger than~$q$ (more precisely, $K,D \geq 2\ell \geq 2q)$). 
    We claim that there are two $\TTT(\nabla_0)$--$\SSS(\nabla_1)$~paths $P_0,P_1$ in $\nabla - B$ with $d_H(P_0,P_1) \geq 2$, such that $P_0$ starts at $\TI(\nabla_0)$ and ends at $\SO(\nabla_1)$ and $P_1$  starts at $\TO(\nabla_0)$ and ends at $\SI(\nabla_1)$ (see Figure~\ref{fig:CSS:Proof:2}).
    We postpone the proof of this claim for later, and first show how to conclude the proof given $P_0, P_1$.

    For this, we follow the above sketch to combine $M_0, M_1$ into a model of $G_{k,d}$: each~$P_i$ joins a branch set of $M_0$ to the `right' branch set of $M_1$ (as indicated in Figure~\ref{fig:CSS:Proof:2}), and is otherwise disjoint from $\nabla_0, \nabla_1$ by Lemma~\ref{lem:QIPreservesSeps} and because $\RR(\nabla_0), \RR(\nabla_1) \subseteq B$ by our choice of~$q$. For each $i\in \{1,2\}$, we form a new branch set comprising the union of these two branch sets and $P_i$, i.e.\ the branch set corresponding to $\TI(\SP(G_{k,d}))$ is defined as $U^0_{\TI(G_{k-1,d})} \cup V(P_0) \cup U^1_{\SO(G_{k-1,d})}$ and analogously for the branch set corresponding to $\TO(\SP(G_{k,d}))$. 
    
    We form a new branch set corresponding to the root $\RR(G_{k,d})$ consisting of $\RR(\nabla)$ together with the vertices of $P$ except for its first and last two vertices, and let the branch paths corresponding to the two edges of $G_{k,d}$ joining $\RR(G_{k,d})$ to its children in $B(G_{k,d})$ consist of the subpath of~$P$ of length~$2$ containing its first or last two vertices, respectively.
    These branch sets and branch paths, together with the branch sets of $M_0,M_1$ that we did not modify, form the branch sets and paths of a model of $G_{k,d}$ in $\nabla$. By construction, this model is natural in $\nabla$. Moreover, it is $2$-fat since $d_H(\nabla_0, \nabla_1) \geq 2$ and by the choice of $P, P_0, P_1$, where we used that every branch path of $M_0,M_1$ has distance at least $2$ from $P$ by our remark after the definition of natural minor-models. 
    \medskip

    Therefore, it suffices to find paths $P_0, P_1$ as above.     
    We first show that there are two $\SSS(\nabla)$--$\TTT(\nabla)$ paths $Q_0,Q_1$ in $\nabla-B$ that have distance at least $2r+1$ from each other (in $\nabla-B$). If not, then by the coarse version of Menger's Theorem for two paths, Theorem~\ref{C Menger}, there is a vertex $w \in \nabla$ such that $B_\nabla(w, 544\cdot(2r+1))$ intersects all $\SSS(\nabla)$--$\TTT(\nabla)$ paths in $\nabla-B$. Let $v \in V(\Delta)$ such that $d_H(w, f(v)) \leq A$. By Lemmas~\ref{lem:inversequasiisom} and~\ref{lem:QIPreservesSeps}, $B_G(v, M\cdot 544\cdot(2r+1)+ 3AM + R)$ separates $\SSS(\Delta)$ and $\TTT(\Delta)$ in $\Delta-B_G(\RR(\Delta), Mq+3AM+R)$ where $R := r(M,3AM)$ is provided by Lemma~\ref{lem:QIPreservesSeps}. 
    Since $h \geq 2\ell+3$ and $d \geq 2\ell+1$ by \eqref{eq:K} and \eqref{eq:D}, and because $\ell = M\cdot\max\{q,544\cdot(2r+1)\}+3AM+R$, this contradicts Lemma~\ref{lem:NSSGraph:NoSmallSep}. Therefore, there are two $\SSS(\nabla)$--$\TTT(\nabla)$ paths $Q_0, Q_1$ in $\nabla-B$ that have distance at least $2r+1$ from each other (in $\nabla-B$). Their distance in $\nabla$ is much smaller by \cite{NgScSeCou}, but we only claim that this distance is at least $2$, which follows as $\nabla-B$ is an induced subgraph of~$\nabla$.
    
    By Lemma~\ref{lem:QIPreservesSeps} and because $\RR(\nabla_0), \RR(\nabla_1) \subseteq B$ by our choice of $q$, the paths $Q_0, Q_1$ both meet $\TTT(\nabla_0)$ and $\SSS(\nabla_1)$. As the adhesion balls of $\nabla_0, \nabla_1$ are balls of radius~$r$ in $\nabla$, and hence have diameter at most $2r$ in $\nabla-B$ as they avoid $B$, and $Q_0, Q_1$ have distance at least $2r+1$ from each other in $\nabla-B$, precisely one of $Q_0, Q_1$ meets $\TI(\nabla_0)$ and the other meets $\TO(\nabla_0)$, and analogously for $\SO(\nabla_1), \SI(\nabla_1)$. 
    Let $P_0, P_1$ be $\TTT(\nabla_0)$--$\SSS(\nabla_1)$ subpaths of $Q_0, Q_1$, respectively. Note that by Lemma~\ref{lem:QIPreservesSeps}, $P_0, P_1$ meet $\nabla_0, \nabla_1$ only at their endvertices.  
    \smallskip

    Suppose \obda\ that $P_0$ starts at $\TI(\nabla_0)$ and $P_1$ starts at $\TO(\nabla_0)$. 
    If $P_0$ ends at $\SO(\nabla_1)$ (and so $P_1$ ends at $\SI(\nabla_1)$), then $P_0, P_1$ are as desired. 
    In the other case where the endvertices of $P_0,P_1$ are the wrong way round, we will reroute them locally inside a suitable subtrigon of $\nabla$ to exchange their endvertices (see Figure~\ref{fig:CSS:Proof:3}).
    
    Let $\Delta'$ be a subtrigon of $\Delta$ next to $\Delta_1$, i.e.\ such that $\TTT(\Delta') = \SSS(\Delta_1)$,
    chosen so that the distance between $\RR(\Delta')$ and $\RR(\Delta_1)$ is precisely $\lceil M(2r+A+1)\rceil$. This is possible because in $\Delta$ there is a trigon next to $\Delta_1$ with any integer distance between their roots ranging from $2$ to $h+1$.  
    Then $d_G(\RR(\Delta), \RR(\Delta'')) = (h-h')+\lceil M(2r+A+1)\rceil -2 \leq M^{-1}(q-2r-A)$ by our choice of~$q$, and hence $d_H(\RR(\nabla), \RR(\nabla')) \leq q-2r$, which implies that $\RR(\nabla') \subseteq B$. Because of this and 
    by Lemma~\ref{lem:QIPreservesSeps}, $P_0, P_1$ both meet $\SSS(\nabla')$. Since the adhesion balls of $\nabla'$ have diameter at most $2r$ in $\nabla-B$ and $P_0, P_1$ have distance at least $2r+1$ from each other in $\nabla-B$,  
    precisely one of $P_0, P_1$ meets $\SO(\nabla')$ and the other one meets $\SI(\nabla')$.
    We may assume \obda\ that if some $P_i$ intersects $\SO(\nabla')$ or $\SI(\nabla')$, then their intersection is a path, for otherwise we could shortcut $P_i$ inside that ball. Indeed, $P_1, P_2$ will still have distance at least $2r+1$ from each other in $\Delta-B$ as $d_H(\SO(\nabla'), \SI(\nabla')) \geq M^{-1}\cdot d_G(\SO(\Delta), \SI(\Delta))-A-2r$ and $d_G(\SO(\Delta), \SI(\Delta)) \geq d' \geq M(4r+A+1)$. 
    For the same reason, we may also assume that if one of $P_0$ or $P_1$ intersects $\TTT(\SP(\nabla'))$, then this intersection is a path as well. 
    In particular, the intersection of $P_0$ (and $P_1$) with $\nabla'$ (and with $\SP(\nabla')$ and $\TP(\nabla')$) is a path.

    \begin{figure}[ht]
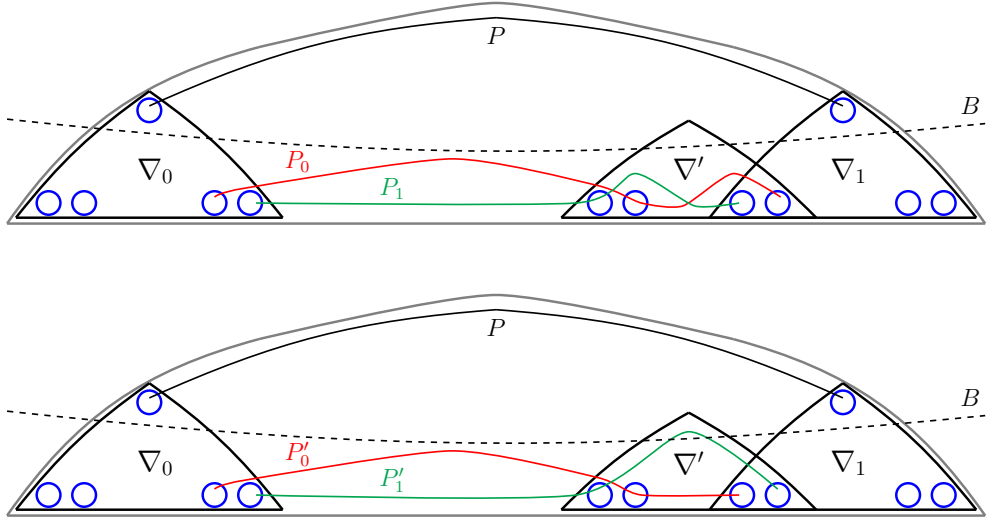

        \centering
        \begin{subfigure}[b]{1\linewidth}
            \centering
            \include{CSSProof3}
        \end{subfigure}
        \begin{subfigure}[b]{1\linewidth}
            \centering
            \include{CSSProof4}
        \end{subfigure}
        \vspace{-2em}
        \caption{A sketch of the paths $P_0, P_1$ (top) and the rerouted paths $P'_0, P'_1$ (bottom). In this case we have $\nabla'' = \nabla'$ and $\cp'' = \cp'$.}
        \label{fig:CSS:Proof:3}
    \end{figure}

    We now reroute $P_0, P_1$ inside $\nabla'$ to exchange their endvertices in $\SSS(\nabla_1)$ (see Figure~\ref{fig:CSS:Proof:3}). 
    For this, note that it is straightforward to extend the definitions of (nested and crossing) linkages of \Sr{sec link} to trigons in $H$ (since the definitions only refer to how the paths intersect $\SSS(G') \cup \TTT(G')$ for some NSS graph $G'$). 
    In particular, as the intersections of $P_0,P_1$ with $\nabla', \SP(\nabla')$ and $\TP(\nabla')$ are paths, they induce a linkage $\cp'$ on $\nabla'$ and two linkages $\cp_0, \cp_1$ on its two child-trigons $\SP(\nabla')$ and $\TP(\nabla')$ by \Lr{Lem cross}~\ref{li} (again, Lemma~\ref{Lem cross} and its proof transfer directly to $H$ because we modified the paths $P_0, P_1$ so that their intersections with $\TTT(\SP(\nabla'))$ are paths). 
    At least one of them is crossing by \Lr{Lem cross}~\ref{lii}. 
    We can substitute one such crossing linkage by a nested one within the same (sub)trigon using \Or{obs NL} (see Figure~\ref{fig:CSS:Proof:3}). More precisely, let $\cp'' \in \{\cp', \cp_0, \cp_1\}$ be crossing, and let $\nabla'' \in \{\nabla', \SP(\nabla'), \TP(\nabla')\}$ be the trigon containing $\cp''$. We apply \Or{obs NL} to the `preimage' $\Delta''$ of $\nabla''$ in $G$ (i.e.\ the trigon $\Delta''$ of $G$ such that $\nabla'' = H[B_H(\varphi(\Delta''), r)]$) to obtain a nested linkage $Q''_0, Q''_1$ in~$\Delta''$. By Lemma~\ref{lem:QIPreservesConn}, $B_H(\varphi(Q''_0), M+A)$ and $B_H(\varphi(Q''_1), M+A)$ are connected, so we may choose $\SSS(\nabla'')$--$\TTT(\nabla'')$ paths $Z_0, Z_1$ inside these sets. Since $d_G(Q''_0, Q''_1) \geq d'$ and $d' \geq M(2M+3A+2)$ by \eqref{eq:D}, it follows that $Z_0, Z_1$ have distance at least $2$ from each other. 

    Recall that the intersections of $P_0$ and $P_1$ with $\SSS(\nabla'') \cup \TTT(\nabla'')$ are paths, and that $P_0, P_1$ avoid $\RR(\nabla') \subseteq B$. 
    Thus, by replacing $\cp''$ by $Z_0,Z_1$, and appropriately adapting the paths within each adhesion ball in $\SSS(\nabla'') \cup \TTT(\nabla'')$, we have modified $P_0,P_1$ into a pair of paths $P'_0,P'_1$, so that $P'_0$ ends at $\SO(\nabla_1) \subseteq U^1_{\SO(G_{k,d})}$ and $P'_1$ ends at $\SI(\nabla_1) \subseteq U^1_{\SI(G_{k,d})}$. Moreover, $P'_0$ and $P'_1$ have distance at least $2$ (in $H$) from each other.
    Indeed, $P_0$ and $P_1$ have large distance in $H-B$ from each other by their choice, and hence distance at least~$2$ in $H$. Moreover, $Z_0$ and $Z_1$ have distance at least~$2$ (in $H$) from each other by the previous paragraph. Finally, $P'_i - Z_i$ and $Z_{1-i}$, for $i \in \{0,1\}$, have distance at least~$2$ from each other (in $H$) since $Z_{1-i}$ is contained in $B_H(\varphi(\Delta''), M+A)$ and $P'_i - Z_i$ avoids $\nabla''$, and hence in particular $B_H(\varphi(\Delta''), M+A+1) \subseteq \nabla''$, except for its adhesion balls $\SSS(\nabla'') \cup \TTT(\nabla'')$ (this implicitly uses that $P_0, P_1$ avoid $\RR(\nabla')$). 

    This completes the (formal) construction of our model of $G_{k,d}$.
    Notice that $P'_0$, $P'_1$ meet $\nabla_0 \cup \nabla_1$ only in their endvertices, and $P$ not at all, because of our choice of $\nabla'$. Moreover, note that $Z_0,Z_1$, and hence $P'_0, P'_1$, might intersect~$B$ (in fact, one of them will intersect $B$ in $\RR(\nabla'')$) but they do not intersect $\RR(\nabla)$ or $B_H(P,2)$ (cf.\ Figure~\ref{fig:CSS:Proof:3}).
    In particular, our model of $G_{k,d}$ in $H$ is $2$-fat.
\end{proof}

\section{Non-fat minors of the NSS graph} \label{sec O fat}

Observe that $K_{2,2,2}$ is isomorphic to the $1$-skeleton of the octahedron. Therefore, we will denote this graph from now on by $O$ to shorten notation. Equivalently, $O$ is the graph obtained from $K_6$ by removing a perfect matching.
The latter definition easily implies

\begin{proposition} \label{prop:O:TwoEdgesBetweenTwoPairsOfVertices}
    For every four distinct vertices $v_0, v_1, u_0, u_1$ of $O$ there is a perfect matching between $\{v_0, v_1\}$ and $\{u_0, u_1\}$ in $O$. \qed
\end{proposition}

\noindent Note that~$O$ is planar and vertex-transitive. 
\medskip

\begin{figure}[ht]
    \centering
    \begin{tikzpicture}[scale=0.6]
        \draw (-4,2)--(0,-2)--(4,2)--(-4,2);
        \draw (-4,2)--(0,1)--(-1,0)--(-4,2);
        \draw (0,-2)--(-1,0)--(1,0)--(0,-2);
        \draw (4,2)--(0,1)--(1,0)--(4,2);
        \draw[fill,black] (0,1) circle (.1);
        \draw[fill,black] (-1,0) circle (.1);
        \draw[fill,black] (1,0) circle (.1);
        \draw[fill,black] (0,-2) circle (.1);
        \draw[fill,black] (4,2) circle (.1);
        \draw[fill,black] (-4,2) circle (.1);
        \draw[black] (0,0.5) node {\footnotesize $1$};
        \draw[black] (0.5,-2) node {\footnotesize $6$};
        \draw[black] (-4.5,2) node {\footnotesize $2$};
        \draw[black] (4.5,2) node {\footnotesize $4$};
        \draw[black] (-1.5,0) node {\footnotesize $3$};
        \draw[black] (1.5,0) node {\footnotesize $5$};
    \end{tikzpicture}
    \caption{The graph $O$.}
    \label{fig:O}
\end{figure}
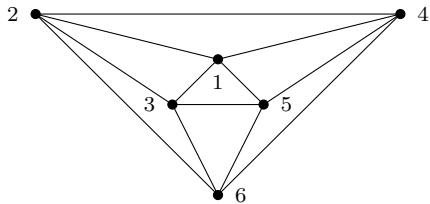

In this section we prove Theorem~\ref{thm:NSSGraph:NoFatO:Intro} which we restate here for convenience: 

\begin{theorem} \label{thm:NSSGraph:NoFatO}
    $O$ is not a $9$-fat minor of $G_{k,d}$ for any $k,d \in \N$. 
\end{theorem}

\noindent Together with our result that the NSS graphs are coarsely self-similar (Theorem~\ref{thm css}) and the fact that the NSS graphs contain $O$ as a ($2$-fat) minor, this implies that $O$ is a counterexample to Conjecture~\ref{conj fat min} (see Section~\ref{sec:Counterex:Proof} for details).

We believe that the NSS graphs in fact exclude $O$ even as a $3$-fat minor, but the proof was more difficult. Therefore, we decided to optimise for simplicity of the proof rather than for the fatness. 
\medskip

To prove that $O$ is not a $9$-fat minor of any NSS graph, we will employ the following tree-decomposition:

\begin{construction} \label{constr:TreeDecompOfNSSGraph}
    Let $G := G_{k,d}$ for some $k,d \in \N$.
    Let $B'(G) \subseteq B(G)$ be the component of $G$ that contains $r$ after removing the interiors of all `dotted' paths (i.e.\ those paths in $G$ which arose by subdividing an edge in the construction of $G_{k,d}$).
    For every node $x$ of $B'(G)$ we set
    \[
    V_x := \Adh(\Delta_x) \cup \TTT(\SP(\Delta_x)) \cup \RRR(\Delta_x)
    \]
    where $\Delta_x$ is the (unique) trigon of $G$ with root $x$, except that if $x$ is a leaf of $B'(G)$ (in which case $\SP(\Delta_x)$ does not exists), then we let $V_x := \Adh(\Delta_x)$. See Figure~\ref{fig:TreeDecomp}. 

    Let $T$ be the tree obtained from $B'(G)$ by adding for every dotted path $P$ in $G$ a vertex $v_P$ to $B'(G)$ and joining it to a node $x$ of $B'(G)$ whose bag $V_x$ contains the endvertices of $P$. (Note that such a node exists for every such $P$.) We set $V_{v_P} := V(P)$.

    It is easy to check that $(T, \cv)$ is a tree-decomposition of $G$. Every adhesion set $V_e$ of an edge $e=xy \in B'(G)$ where $y$ is the successor of $x$ in $B'(G)$ is of the form
    \[
    V_e = \{y\} \cup \SSS(\Delta_y) \cup \TTT(\Delta_y) = \Adh(\Delta_y),
    \]
    and the separation of $G$ it induces is $(V(\Delta_y), V(\Ext(\Delta_y))$. (For this, recall that by \cite[Lemma~12.3.1]{Bibel} every edge $xy$ of $T$ in any tree-decomposition $(T, \cv)$ of any graph $G$ \defi{induces} a separation \defi{$\{A_e^x, A_e^y\}$} of $G$ where $A_e^z := \bigcup_{t \in V(T_z)} V_t$, for $z \in \{x,y\}$, is the union over all bags $V_t$ of nodes $t$ of the unique component $T_z$ of $T-xy$ containing $z$.) 
\end{construction}

\begin{figure}[h]
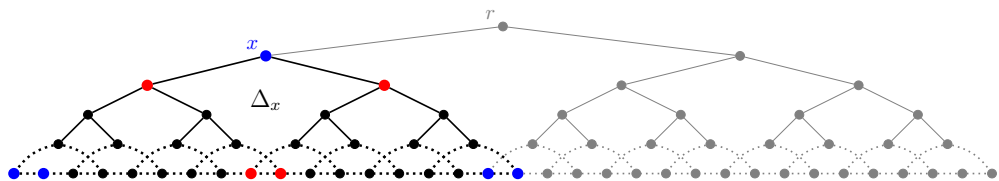

    \centering
    \include{TreeDecomp}
    \vspace{-2em}
    \caption{The adhesion set $V_e$ of the edge $e=rx$ is indicated in blue. The red vertices are those contained in the bag $V_x$ but not in $V_r$. The black trigon is $\Delta_x$.}
    \label{fig:TreeDecomp}
\end{figure}

Before we start with the proof of Theorem~\ref{thm:NSSGraph:NoFatO}, let us first, as a toy example, show that $K_8$ is not a $3$-fat minor of any NSS graph, to demonstrate the usefulness of the tree-decomposition from Construction~\ref{constr:TreeDecompOfNSSGraph}:

\begin{proposition} \label{prop:NSSGraph:NoFatK7}
    $K_7$ is not a $3$-fat minor of $G_{k,d}$ for any $k,d \in \N$.
\end{proposition}

\begin{proof}
    Suppose towards a contradiction that $G := G_{k,d}$ contains a $3$-fat model $(\cu, \ce)$ of $K_7$. Let $(T, \cv)$ be the tree-decomposition of $G$ as in Construction~\ref{constr:TreeDecompOfNSSGraph}. Since $(T, \cv)$ has adhesion~$5$ but~$K_7$ is $6$-connected, there exists for every edge $e=xy \in E(T)$ a unique side of $\{A^x_e,A^y_e\}$ which meets all branch sets in~$\cu$. Hence, orienting every edge of $T$ towards the side meeting all branch sets defines an orientation of $E(T)$. Since $T$ is finite, there is a sink $x \in V(T)$, i.e.\ a node $x$ of $T$ all whose incident edges are oriented towards $x$. Then the bag~$V_x$ intersects every branch set in~$\cu$.

    Clearly, $x$ cannot be a node of $T - B'(G)$ (of the form $v_P$), because their bags are only a path, and the adhesion set corresponding to their unique incident edge has size $2$. Hence, $V_x = \Adh(\Delta_x) \cup \TTT(\SP(\Delta_x)) \cup \RRR(\Delta_x)$. 
    Since the vertices in $\RRR(\Delta_x)$ have pairwise distance at most $2$ but $(\cu, \ce)$ is $3$-fat, at most one, and hence precisely one, branch set $U_i$, with $i=1$ say, meets $\RRR(\Delta_x)$. 
    Moreover, the six other vertices of $V_x-\RRR(\Delta_x)$ are contained in distinct branch sets $U_i$ with $i \neq 1$. 
    Since $K_7$ is complete and the model $(\cu, \ce)$ of $K_7$ is $3$-fat, it follows that there are two $\SSS(\Delta_x)$--$\TTT(\SP(\Delta_x))$ paths $Q_1,Q_2$ in~$G$ that avoid $\RTT(\Delta_x)$ and that are at least $3$ apart. But then $Q_1, Q_2$ must be contained in $\SP(\Delta_x)$ as they avoid $\RTT(\Delta_x)$. By Lemma~\ref{lem:TwoPathsInNSSGraph}, one of them has to meet $\RR(\SP(\Delta_x)) \in \RTT(\Delta_x)$, a contradiction.
\end{proof}

We remark that Proposition~\ref{prop:NSSGraph:NoFatK7}, together with Theorems~\ref{thm:NSSGraph:K7} and \ref{thm:css:finite:maintext}, implies that $K_7$ is \bad.

We would like to imitate the proof of Proposition~\ref{prop:NSSGraph:NoFatK7} to prove Theorem~\ref{thm:NSSGraph:NoFatO}, but we can no longer guarantee that there is a bag of our tree-decomposition that intersects every branch set of our hypothetical model of $O$. 
But we can still use the same strategy to find either a bag that meets every branch set of our hypothetical model of $O$, or find an adhesion that meets four branch sets so that the other two lie on distinct sides of the corresponding separation.
Either way, there will be essentially only one configuration which will be hard to deal with, and we define this configuration in Configuration~\ref{setting} below.
\medskip

In the following sections we will often make use of the fact that if $W$ is a path in~$O$, and we are given a model $(\cu, \ce)$ of $O$ in a graph $G$, then we can find an $A$--$B$ path in $G$ in $X :=  \bigcup_{w\in V(W)} G[U_w] \cup \bigcup_{e \in E(W)} E_w$ for all $A,B \subseteq V(X)$. Hence, we will say that $W$ \defi{induces} an $A$--$B$ path. For this, note that if we assume that the branch sets in $\cu$ are trees, then, in what follows, the $A$--$B$ path induced by $W$ will `usually' be unique. (In fact, below, it will always be unique, but this is not important for our arguments.)

If $(\cu, \ce)$ is $3$-fat in $G$, and $P, P'$ are paths in $G$ induced by disjoint paths $W, W'$ in~$O$, respectively, then $P, P'$ have distance at least~$3$ from each other.

\subsection{No fat model of \texorpdfstring{$O$}{O} intersects an adhesion \texorpdfstring{$\Adh(P)$}{Adh(P)} nicely}

For convenience, we number the vertices of $O$ as follows.
We let $V(O) = [6]$ and $E(O) := \{ij \in [6]^2 : i \neq j \text{ and } i+j \neq 7\}$. Thus, $O$ contains all edges in $[6]^2$ except for $16, 25$ and $34$.
\smallskip

\begin{configuration} \label{setting}
    Let $\Delta$ be a trigon of $G:=G_{k,d}$ for some $k,d \in \N$, and suppose that $G$ contains a $3$-fat model $(\cu, \ce)$ of $O$. Suppose further that there are distinct $i, j, h \in \{3,4,5\}$ such that (\fig{fig:setting})
    \begin{enumerate}[label=\rm{(\roman*)}]
        \item \label{itm:setting:LeftAndTop} $\RR(\Delta) \in U_2$, $\SO(\Delta) \in U_i$, $\SI(\Delta) \in U_j$, 
        \item \label{itm:setting:Right} $\TTT(\Delta) \cap U_h \neq \emptyset$,
        \item \label{itm:setting:U1U6} $U_1 \cap V(\Delta) \neq \emptyset$ and $U_6 \cap V(\Ext(\Delta)) \neq \emptyset$, and
        \item \label{itm:setting:BranchPaths} each branch path of $(\cu, \ce)$ is contained in either $\Delta$ or $\Ext(\Delta)$.
    \end{enumerate}
\end{configuration}

\noindent We remark that even though Configuration~\ref{setting} is quite specific, it is, up to the vertex-transitivity of~$O$ and the symmetry of the NSS graph, in fact the only difficult case which we have to consider.
\medskip

\begin{figure}[ht]
    \centering
    \include{Setting}
    \vspace{-2em}
    \caption{A sketch of Configuration~\ref{setting}. The branch set $U_2$ contains $R(\Delta)$, the branch sets $U_i, U_j$ contain $\SO(\Delta)$ and $\SI(\Delta)$, respectively, and $U_h$ contains and least one of $\TO(\Delta)$ and $\TI(\Delta)$. Moreover, $U_1$ meets $V(\Delta)$ and $U_6$ meets $\Ext(\Delta)$. Note that one of $U_1, U_6$ can meet both $V(\Delta)$ and $\Ext(\Delta)$.}
    \label{fig:setting}
\end{figure}

The main effort in this section goes into the following:

\begin{lemma} \label{lem:SettingCannotOccur}
    Configuration~\ref{setting} cannot occur.
\end{lemma}

We prepare its proof with additional auxiliary lemmas forbidding certain refinements of our setting.

\begin{lemma} \label{lem:NoFatO:LeftDoesNotTouchRight}
    Let $G, \Delta$ and $(\cu, \ce)$ be as in Configuration~\ref{setting}. Then none of $U_i, U_j$ meets $\TTT(\Delta)$.
\end{lemma}

\begin{proof}
    Suppose for a contradiction that for some $n \in \{i,j\}$ the set $U_n$ meets $\TTT(\Delta)$, and let $m$ be the other element of $\{i,j\}$. 
    Then none of $U_1, U_6$ meets $\Adh(\Delta)$, which by \ref{itm:setting:U1U6} implies that $U_1 \subseteq V(\Delta)\setminus \Adh(\Delta)$ and $U_6 \subseteq V(\Ext(\Delta))\setminus \Adh(\Delta)$.

    Since $U_n$ intersects both $\SSS(\Delta)$ and $\TTT(\Delta)$ but avoids $\RR(\Delta) \in U_2$, there is an $\SSS(\Delta)$--$\TTT(\Delta)$ path $P$ in $G[U_n]$ that is contained in either $\Delta$ or $\Ext(\Delta)$. Let $H \in \{\Delta, \Ext(\Delta)\}$ be such that $P \subseteq H$, and let $\ell \in \{1,6\}$ such that $U_\ell \subseteq V(H)\setminus \Adh(\Delta)$. Then the $\SSS(\Delta)$--$\TTT(\Delta)$ path $Q$ induced by $m\ell k$ avoids $\RR(\Delta) \in U_2$ (as $2 \notin \{m,l,k\}$, and thus $Q$ is also contained in $H$ because $U_\ell \subseteq V(H)\setminus \Adh(\Delta)$. Since $(\cu, \ce)$ is $3$-fat, $P,Q$ have distance at least~$3$ from each other and from $\RR(\Delta) \in U_2$. But this contradicts Lemma~\ref{lem:TwoPathsInNSSGraph} or~\ref{lem:TwoPathsInNSSGraph:ExteriorOfSubpyramid}.
\end{proof}

\begin{lemma} \label{lem:NoFatO:PathsInsidePyramid}
    Let $G, \Delta$ and $(\cu, \ce)$ be as in Configuration~\ref{setting}, and suppose that $U_1 \subseteq V(\Delta) \setminus \Adh(\Delta)$. 
    Then any of $E_{ih}$, $E_{2j}$, $E_{i6}, E_{j6}$ that exist are contained in $\Ext(\Delta)$, and if both $E_{2i}, E_{jh}$ exist, then at most one of them is contained in $\Delta$. 
\end{lemma}

\begin{proof}
    \noindent \textbf{$\mathbf{E_{ih}, E_{i6}}$:} Suppose for a contradiction that $E_{i6}$ (or $E_{ih}$, if it exists) is not contained in $\Ext(\Delta)$, which by \ref{itm:setting:BranchPaths} implies that it is contained in $\Delta$. 
    In particular, if $E_{i6} \subseteq \Delta$, this implies by \ref{itm:setting:U1U6} that $U_6$ meets $\TTT(\Delta)$. By \ref{itm:setting:BranchPaths} and because $U_1 \subseteq V(\Delta)\setminus \Adh(\Delta)$ by assumption, we also have $E_{12} \cup E_{1j} \subseteq \Delta$. This implies that the $\SI(\Delta)$--$\RTT(\Delta)$ path $P$ induced by $j12$ is contained in $\Delta$. Similarly, by Lemma~\ref{lem:NoFatO:LeftDoesNotTouchRight} and because $E_{i6}$ (or $E_{ih}$) is contained in $\Delta$ and $2 \notin \{i,h,6\}$, the $\SO(\Delta)$--$\TTT(\Delta)$ path $Q$ induced by $i6$ (or by $ih$) is contained in $\Delta$. 
    Since $(\cu, \ce)$ is $3$-fat and $\RR(\Delta) \in U_2$, the path $Q$ has distance at least $3$ from $P$ and from $\RR(\Delta) \in U_2$. But this contradicts Corollary~\ref{cor:TwoPAthsInNSSGraph:TopAndRight}.
    \medskip

    \noindent \textbf{$\mathbf{E_{2j}, E_{j6}}$:} Suppose for a contradiction that $E_{j6} \subseteq \Delta$ (or $E_{2j} \subseteq \Delta$). This case is similar to the previous one, except that we replace $E_{i6}$, $E_{ih}$, $E_{12}$ by $E_{j6}$, $E_{2j}$, $E_{16}$, respectively, and $j12$ by $j6$ (or by $j2$) and $i6$ (or $ih$) by $i1h$. 
    \smallskip

    \noindent \textbf{$\mathbf{E_{2i}, E_{jh}}$:} 
    Suppose for a contradiction that $E_{2i}, E_{jh} \subseteq \Delta$. Then by Lemma~\ref{lem:NoFatO:LeftDoesNotTouchRight}, the $\SSS(\Delta)$--$\RTT(\Delta)$ paths $P,Q$ induced by $i2$ and $jh$ are contained in $\Delta$. Since $(\cu, \ce)$ is $3$-fat, $P,Q$ are at least $3$ apart.
    By Corollary~\ref{cor:TwoPAthsInNSSGraph:TopAndRight}, it follows that $P$ is the (unique) $\SO(\Delta)$--$\RR(\Delta)$ path in $B(G)$. 
    Since $U_1 \subseteq V(\Delta)\setminus \Adh(\Delta)$ also the $\SO(\Delta)$--$\RTT(\Delta)$ path $P'$ induced by $i12$ is contained in $\Delta$.
    Moreover, $P'$ has distance at least~$3$ from $Q$ because $(\cu, \ce)$ is $3$-fat. By applying Corollary~\ref{cor:TwoPAthsInNSSGraph:TopAndRight} again, to $P', Q$, we obtain that $P'$ is also the (unique) $\SO(\Delta)$--$\RR(\Delta)$ path in $B(G)$. But then $P = P'$, which is a contradiction because $P'$ intersects $U_1$ but $P$ does not, and so $P\neq P'$.
\end{proof}

The next lemma is similar to the previous one, the difference being that instead of $U_1 \subseteq V(\Delta)\setminus \Adh(\Delta)$ we are assuming $U_6 \subseteq V(G-\Delta)$:

\begin{lemma} \label{lem:NoFatO:PathsOutsidePyramid}
    Let $G, \Delta$ and $(\cu, \ce)$ be as in Configuration~\ref{setting}, and suppose that $U_6 \subseteq V(G-\Delta)$. Then any of $E_{2i}$, $E_{jh}$, $E_{1i}$, $E_{1j}$ that exist are contained in $\Delta$, and if both $E_{2j}, E_{ih}$ exist, then at most one of them is contained in $\Ext(\Delta)$.
\end{lemma}

\begin{proof}
    The proof is analogous to the proof of Lemma~\ref{lem:NoFatO:PathsInsidePyramid} except that we interchange $1$ and $6$ as well as $i$ and $j$ and also $\Delta$ and $\Ext(\Delta)$. In the first two cases, we then obtain two paths $P,Q \subseteq \Ext(\Delta)$ between $\SSS(\Delta)$ and $\RTT(\Delta)$ such that $Q$ starts in $\SI(\Delta)$ (instead of $\SO(\Delta)$) and has distance at least $3$ from $P$ and from $\RR(\Delta)$, which contradicts Lemma~\ref{lem:TwoPathsInNSSGraph:ExteriorOfSubpyramid} (instead of Corollary~\ref{cor:TwoPAthsInNSSGraph:TopAndRight}). 
    In the third case, we obtain two $\SI(\Delta)$--$\RTT(\Delta)$ paths $P,P'$ in $B(G)$ such that $P$ is induced by $j2$ and $P'$ is induced by $j62$. Since $P'$ meets $U_6$ but $P$ does not, but both of them are contained in $B(G)$, it follows that $P'$ contains all of $P$ except for its last vertex. This implies that the predecessor of $\RR(\Delta)$ in $B(G)$ is contained in $U_j$, and thus $U_j$ has distance $1$ from $\RR(\Delta) \in U_2$, a contradiction.
\end{proof}

We are now ready to prove that Configuration~\ref{setting} does not occur.

\begin{proof}[Proof of Lemma~\ref{lem:SettingCannotOccur}]
    By \ref{itm:setting:U1U6}, either $U_1 \subseteq V(\Delta) \setminus \Adh(\Delta)$ and $U_6 \subseteq V(G-\Delta)$, or $U_6 \cap \Adh(\Delta) \neq \emptyset$, or $U_1 \cap \Adh(\Delta) \neq \emptyset$. We consider these cases separately:
    \medskip
    
    \noindent \textbf{Case 1:} \textit{$U_1 \subseteq V(\Delta) \setminus A(\Delta)$ and $U_6 \subseteq V(G-\Delta)$.}  
    
    By Proposition~\ref{prop:O:TwoEdgesBetweenTwoPairsOfVertices}, there exists a perfect matching in $O$ between $\{2,h\}$ and $\{i,j\}$, i.e.\ either $2i, hj \in E(O)$ or $2j, hi \in E(O)$. 
    If $2i,hj \in E(O)$, then, by Lemma~\ref{lem:NoFatO:PathsOutsidePyramid}, $E_{2i}, E_{hj}$ are contained in $\Delta$, which contradicts Lemma~\ref{lem:NoFatO:PathsInsidePyramid}. If $2j, hi \in E(O)$, then, by Lemma~\ref{lem:NoFatO:PathsInsidePyramid}, $E_{2j}, E_{hi}$ are contained in $\Ext(\Delta)$, which constradicts Lemma~\ref{lem:NoFatO:PathsOutsidePyramid}.
    \medskip

    \noindent \textbf{Case 2:} \textit{$U_6 \cap \Adh(\Delta) \neq \emptyset$}
    
    By \ref{itm:setting:LeftAndTop}, we have $U_6 \cap \TTT(\Delta) \neq \emptyset$. 
    Moreover, $U_1 \subseteq V(\Delta) \setminus V(\Delta)$ by \ref{itm:setting:U1U6}, and every branch path is contained in either $\Delta$ or $\Ext(\Delta)$ by \ref{itm:setting:BranchPaths}. 

    We first show that $E_{ih}, E_{jh} \subseteq \Delta$ (if they exist). For this, suppose for a contradiction that $E_{nh}$, with $\{n,m\} = \{i,j\}$, exists and is contained in $\Ext(\Delta)$. Then the $\SSS(\Delta)$--$\TTT(\Delta)$ path $P$ induced by $nh$ is contained in $\Ext(\Delta)$ (as it avoids $\RR(\Delta) \in U_2$).
    Since $m \neq 1$ and $U_6$ meets $\TTT(\Delta)$, the path $m6$ in $O$ also induces an $\SSS(\Delta)$--$\TTT(\Delta)$ path $Q$, which is contained in $\Ext(\Delta)$ because $E_{m6} \subseteq \Ext(\Delta)$ by Lemma~\ref{lem:NoFatO:PathsInsidePyramid} (and because it avoids $\RR(\Delta) \in U_2$). As $(\cu, \ce)$ is $3$-fat, $P,Q$ have distance at least~$3$ from each other and from $\RR(\Delta) \in U_2$. But this contradicts Lemma~\ref{lem:TwoPathsInNSSGraph:ExteriorOfSubpyramid}. 
    
    By Lemma~\ref{lem:NoFatO:PathsInsidePyramid}, it follows that $ih \notin E(O)$. Hence, $\{i,h\} = \{3,4\}$; by symmetry of~$O$, we may assume that $i = 3$ and $h = 4$. In particular, $j = 5$. Since $jh = 45 \in E(O)$, it follows by the previous argument that $E_{jh} \subseteq \Delta$. But then $E_{2i}$, which exists as $2i = 23 \in E(O)$, is contained in $\Ext(\Delta)$ by Lemma~\ref{lem:NoFatO:PathsInsidePyramid}. Hence, the $\SO(\Delta)$--$\RR(\Delta)$ path $P$ induced by $i2$ is contained in $\Ext(\Delta)$. By Lemma~\ref{lem:NoFatO:PathsInsidePyramid}, also the $\SI(\Delta)$--$\TTT(\Delta)$ path $Q$ induced by $j6$ is contained in $\Ext(\Delta)$ (as it avoids $\SO(\Delta) \in U_i$ and $\RR(\Delta) \in U_2$). Since $P,Q$ are at least $3$ apart because $(\cu,\ce)$ is $3$-fat, this contradicts Lemma~\ref{lem:TwoPathsInNSSGraph:ExteriorOfSubpyramid}. 
    \medskip

    \noindent \textbf{Case 3:} \textit{$U_1 \cap \Adh(\Delta) \neq \emptyset$}
    
    The proof is analogous to the proof of Case 2 except that we interchange $1$ and $6$ as well as $\Delta$ and $\Ext(\Delta)$ and also Lemma~\ref{lem:TwoPathsInNSSGraph:ExteriorOfSubpyramid} and Corollary~\ref{cor:TwoPAthsInNSSGraph:TopAndRight}. 
    In particular, we obtain $E_{ih}, E_{jh} \subseteq \Ext(\Delta)$. It then follows by Lemma~\ref{lem:NoFatO:PathsOutsidePyramid} that $jh \notin E(O)$, and thus $\{j,h\} = \{3,4\}$; by symmetry, we may assume $j = 3$ and $h = 4$. In particular, $i = 5$.
    It follows that $ih = 45, 2j=23 \in E(O)$, which implies $E_{2j} \subseteq \Delta$ by Lemma~\ref{lem:NoFatO:PathsOutsidePyramid}. Hence, the $\SI(\Delta)$--$\RR(\Delta)$ path $P$ induced by $j2$ and the $\SO(\Delta)$--$\TTT(\Delta)$ path $Q$ induced by $i1$ are contained in~$\Delta$.
    Since $(\cu, \ce)$ is $3$-fat, and thus $P,Q$ have distance at least~$3$ from each other, this contradicts Corollary~\ref{cor:TwoPAthsInNSSGraph:TopAndRight}. 
\end{proof}

\subsection{\texorpdfstring{$O$}{O} minors and the tree-decomposition}

In this section we prove two corollaries which follow from the fact that Configuration~\ref{setting} cannot occur. Both corollaries describe how a fat model of $O$ can interact with the tree-decomposition $(T, \cv)$ from Construction~\ref{constr:TreeDecompOfNSSGraph}. For this, recall that the `relevant' adhesion sets of $(T, \cv)$ are of the form $\Adh(\Delta)$, and that the `relevant' bags are of the form $V_\Delta = \Adh(\Delta) \cup \TTT(\SP(\Delta)) \cup \RRR(\Delta)$ for some trigon~$\Delta$ of the NSS graph.

\begin{corollary} \label{cor:NoFatO:NoAdhesionMeetsFiveBranchSets}
    Let $\Delta$ be a trigon of some NSS graph $G$, and suppose that $G$ contains a $3$-fat model $(\cu, \ce)$ of~$O$. Then $\Adh(\Delta)$ intersects at most four branch sets in $\cu$.
\end{corollary}

\begin{proof}
    Suppose for a contradiction that $\Adh(\Delta)$ intersects five branch sets in $\cu$. Since $|\Adh(\Delta)| = 5$ and $|O| = 6$, there is (precisely) one branch set in $\cu$ which does not meet $\Adh(\Delta)$, and all other branch sets intersect $\Adh(\Delta)$ exactly once. By the vertex-transitivity of $O$, we may assume that $U_1 \subseteq V(\Delta) \setminus \Adh(\Delta)$ or $U_6 \subseteq V(\Ext(\Delta))\setminus\Adh(\Delta)$. 
    We consider two cases:
    \medskip

    \noindent \textbf{Case 1:} $\RR(\Delta) \notin U_1 \cup U_6$.

    By the vertex-transitivity of $O$, we may assume that $\RR(\Delta) \in U_2$. Furthermore, by the symmetry of $G$, we may assume that $U_1$ and $U_6$ do not meet $\SSS(\Delta)$. Hence, \ref{itm:setting:LeftAndTop} and \ref{itm:setting:Right} hold. Moreover, \ref{itm:setting:U1U6} holds because one of $U_1,U_6$ meets $\Adh(\Delta)$ and the other is contained in $\Delta$ or $\Ext(\Delta)$, respectively. Finally, \ref{itm:setting:BranchPaths} holds because all of $\Adh(\Delta)$ is contained in branch sets, and therefore no branch path can traverse it. Thus we are in Configuration~\ref{setting}, contradicting Lemma~\ref{lem:SettingCannotOccur}.
    \medskip

    \noindent \textbf{Case 2:} $\RR(\Delta) \in U_1 \cup U_6$.

    Let $\{m,n\} = \{1,6\}$ such that $\RR(\Delta) \in U_n$.
    Let $i,j,k,\ell$ such that $\SSS(\Delta) \subseteq U_i \cup U_j$ and $\TTT(\Delta) \subseteq U_k \cup U_\ell$.
    By Proposition~\ref{prop:O:TwoEdgesBetweenTwoPairsOfVertices} and without loss of generality, $ik, j\ell \in E(O)$. Then $ik, j\ell$ induce two $\SSS(\Delta)$--$\TTT(\Delta)$ paths $Q_{ik}, Q_{j\ell}$, which have distance at least~$3$ from each other and from $\RR(\Delta) \in U_1 \cup U_6$ since $(\cu,\ce)$ is $3$-fat. By Lemmas~\ref{lem:TwoPathsInNSSGraph} and~\ref{lem:TwoPathsInNSSGraph:ExteriorOfSubpyramid}, one of them is contained in $\Delta$ and one is contained in $\Ext(\Delta)$. Without loss of generality let $Q_{ik}$ be contained in that subgraph $H \in \{\Delta,\Ext(\Delta)\}$ which also contains~$U_m$.
    Since $n \notin \{j,\ell\}$, also $jm\ell$ induces an $\SSS(\Delta)$--$\TTT(\Delta)$ path $Q'_{j\ell}$ in $H$, which has distance at least $3$ from $Q_{ik}$ and from $\RR(\Delta) \in U_n$. But this contradicts Lemma~\ref{lem:TwoPathsInNSSGraph} or~\ref{lem:TwoPathsInNSSGraph:ExteriorOfSubpyramid}.
\end{proof}

\begin{corollary} \label{lem:NoFatO:NoBagMeetsAllBranchSets}
    Let $(T, \cv)$ be the tree-decomposition from Construction~\ref{constr:TreeDecompOfNSSGraph} of some NSS graph $G$. Suppose that $G$ contains a $5$-fat model $(\cu, \ce)$ of $O$. Then no bag $V_x$ of $T$ intersects all six branch sets in $\cu$.
\end{corollary}

\begin{proof}
    Suppose towards a contradiction that some bag $V_x$ meets all six branch sets. Clearly, $x$ cannot be a leaf of $T$ (of the form $v_y$), because their bags are only a path, and the adhesion set corresponding to their unique incident edge $yv_y$ has size $2$. Let $\Delta:=\Delta_x$. 
    We consider two cases according to whether $(\cu,\ce)$ intersects $\mathbf{RRR}(\Delta)$: 
    \medskip
    
    \noindent \textbf{Case 1:} \textit{$\RRR(\Delta) \cap U_i \neq \emptyset$ for some $i \in [6]$.}

    We replace $U_i$ in $\cu$ by $B_G(U_i, 2)$ and shorten each path $E_{ij}$ accordingly, i.e.\ we replace each branch path $E_{ij}$ in $\ce$ with some $B_G(U_i, 2)$--$U_j$ subpath of $E_{ij}$. Then $\RRR(\Delta)$ is a subset of the new branch set $U_i$, and since $(\cu, \ce)$ is $5$-fat, the new model of $O$ is $3$-fat. As every $U_j$ still intersects $V_x$, it follows that $V_x = \Adh(\Delta) \cup \TTT(\SP(\Delta))$ meets all branch sets of the new model $(\cu, \ce)$; in particular $V_x \setminus \RRR(\Delta)$ meets all branch sets except for $U_i$.  
    Recall that $V_x \setminus \RRR(\Delta) = \SSS(\Delta)  \cup \TTT(\SP(\Delta)) \cup \TTT(\Delta)$. Combining this with the previous argument yields that two of the three sets in this union are contained in $\bigcup_{j \in [6]\setminus \{i\}} U_i$.  
    Since any such combination is either of the form $\Adh(\Delta) \setminus \RR(\Delta)$ or $\Adh(\SP(\Delta)) \setminus \RR(\SP(\Delta))$ or $\Adh(\TP(\Delta)) \setminus \RR(\TP(\Delta))$, and because $\RRR(\Delta) \subseteq U_i$, some adhesion $\Adh(\Delta)$, $\Adh(\SP(\Delta))$ or $\Adh(\TP(\Delta))$ meets five branch sets of $(\cu, \ce)$,  contradicting Corollary~\ref{cor:NoFatO:NoAdhesionMeetsFiveBranchSets}. 
    \medskip

    \noindent \textbf{Case 2:} \textit{No branch set meets $\mathbf{RRR}(\Delta)$.}

    Then all six vertices in $V_\Delta -\RRR(\Delta)$ are contained in distinct branch sets in $\cu$; let $\SSS(\Delta) \subseteq U_{i_1} \cup U_{i_2}$, $\TTT(\SP(\Delta)) \subseteq U_{i_3} \cup U_{i_4}$ and $\TTT(\Delta) \subseteq U_{i_5} \cup U_{i_6}$.
    By Proposition~\ref{prop:O:TwoEdgesBetweenTwoPairsOfVertices} and without loss of generality, $i_1i_3, i_2i_4, i_3i_5, i_4i_6 \in E(O)$. 
    Let $P,Q$ be the $\SSS(\Delta)$--$\TTT(\SP(\Delta))$ paths induced by $i_1i_3, i_2i_4$, respectively. Since $(\cu,\ce)$ is $3$-fat, $P,Q$ have distance at least~$3$ from each other. Then either one of $P,Q$ meets $\RR(\Delta)$, or they are contained in $\SP(\Delta)$, in which case by Lemma~\ref{lem:TwoPathsInNSSGraph}, one of $P,Q$ contains $\RR(\SP(\Delta))$. 
    Thus, one of $P,Q$, say $P$, intersects $\RR(\Delta) \cup \RR(\SP(\Delta)) \subseteq \RRR(\Delta)$.
    By the same argument, one of the $\SSS(\TP(\Delta))$--$\TTT(\Delta)$ paths induced by $i_3i_5, i_4i_6$, call it $P'$, meets $\RR(\Delta) \cup \RR(\TP(\Delta)) \subseteq \RRR(\Delta)$. 
    Since $(\cu, \ce)$ is $3$-fat, the intersections of $\RRR(\Delta)$ with $P$ and $P'$ are either contained in the same branch path in $\ce$, or some branch set in $\cu$ meets $\RRR(\Delta)$ (in the intersection with $P$ or $P'$). The former case is not possible because $P, P'$ do not intersect the same branch path, and the latter case is a contradiction to the assumption of Case~2.
\end{proof}

\subsection{Balancing the model} 

We need one last lemma for our proof of Theorem~\ref{thm:NSSGraph:NoFatO}, which given a fat enough model of $O$ finds a less fat model that is more `well-behaved' around the adhesion $\Adh(\Delta)$ of some trigon $\Delta$, in the sense that it comes somewhat (but not yet all the way) close to the assumptions in Configuration~\ref{setting}. 

\begin{lemma} \label{lem:SomeAdhesionIntersectsONicely}
    Let $K \geq 3$, and suppose that for some $k, d \in \N$ the NSS graph $G := G_{k,d}$ contains a $3K$-fat model of~$O$. 
    Then there is a $K$-fat model $(\cu, \ce)$ of~$O$ in~$G$ and a trigon $\Delta$ of $G$ such that
    \begin{enumerate}[label=\rm{(\arabic*)}]
        \item \label{itm:NiceModelOfO:1}  $U_{1} \subseteq V(\Delta) \setminus \Adh(\Delta)$ and $U_{6} \subseteq V(G-\Delta)$,
        \item \label{itm:NiceModelOfO:2} $\Adh(\Delta) \cap U_{i} \neq \emptyset$ for all $i \in \{2,3,4,5\}$, and
        \item \label{itm:NiceModelOfO:3} each branch path of $(\cu, \ce)$ is contained in either $\Delta$ or $\Ext(\Delta)$.
    \end{enumerate}
\end{lemma}

\begin{proof}
    Suppose that $G$ contains a $3K$-fat model $(\cu, \ce)$ of~$O$. 
    Let $(T, \cv)$ be the tree-decomposition of $G$ provided by Construction~\ref{constr:TreeDecompOfNSSGraph}. Note that, for every edge~$e = xy$ of $B'(G)\subseteq T$, there is a trigon $\Delta$ of $G$ such that $V_e = \Adh(\Delta)$ and $(A^x_e, A^y_e) = (V(\Delta), V(\Ext(\Delta)))$. We distinguish two cases.
    \medskip

    \noindent \textbf{Case 1:} \textit{For every edge $e = xy$ of $T$, there is a side $A^z_e$ such that every $U_i$ intersects $B_G(A^z_e, K-1)$.}

    Since $|V_e| \leq 5$ but $|O| = 6$ and $d_G(U_i, U_k) > 2K-2$ for all $i \neq k \in [6]$, there is for every edge $e=xy$ of $T$ a unique side $A^z_e$ such that $B_G(A^z_e, K-1) \cap U_i \neq \emptyset$ for all $i \in [6]$.
    Hence, orienting each edge $e$ towards its end $z$ defines an orientation of $E(T)$. Since $T$ is finite, there is a sink $t \in V(T)$, i.e.\ a node $t$ of $T$ all whose incident edges are oriented towards $t$. Then $B_G(V_t, K-1) \cap U_i \neq \emptyset$ for all $i \in [6]$. We set $U'_i := B_G(U_i, K-1)$ for all $i \in [6]$ and shorten the branch paths $E_{ij}$ to $U'_i$--$U'_j$ paths $E'_{ij} \subseteq E_{ij}$. Then $(\cu', \ce')$ is a $(K+2)$-fat model of $O$ since $(\cu, \ce)$ was $3K$-fat, and all its branch sets meet $V_t$. But since $K\geq 3$, this contradicts Corollary~\ref{lem:NoFatO:NoBagMeetsAllBranchSets}.
    \medskip

     \noindent \textbf{Case 2:} \textit{There is an edge $e=xy$ of $T$ and two vertices $k,\ell \in V(O)$ such that $U_k \subseteq A^x_e \setminus B_G(A^y_e, K-1)$ and $U_\ell \subseteq A^y_e \setminus B_G(A^x_e, K-1)$.}
    \medskip

    \noindent \textbf{Case 2a:} $k\ell \notin E(O)$.

    By the vertex-transitivity of $O$, we may assume that $k = 1$ and $\ell = 6$, so \ref{itm:NiceModelOfO:1} holds. 
    There are four paths $Q_j \subseteq E_{1j} \cup G[U_j] \cup E_{j6}$ between $U_1$ and $U_6$, for $j \in \{2,3,4,5\}$. 
    For all $j \in \{2, 3, 4, 5\}$ we define $U'_j$ as follows. If $U_j \cap V_e \neq \emptyset$, then we set $U'_j := U_j$.
    If $U_j \subseteq A^y_e \setminus A^x_e$, then $E_{1j}$ meets $V_e$. Since $d_G(U_1, V_e) \geq K$, the (unique) subpath $Q'_j$ of $E_{1j}$ between $U_j$ and $V_e$ that is internally disjoint from $V_e$ has distance at least $K$ from $U_1$. We set $U'_j := U_j \cup V(Q'_j)$ and let $E'_{1j}$ be the remaining subpath of $E_{1j}$ that starts in $U_1$ and ends in the first vertex of $Q'_j$. 
    Otherwise, if $U_j \subseteq A^x_e \setminus A^y_e$, we define $U'_j$ and $E'_{j6}$ analogously.
    Then $(\cu', \ce')$ (with $U'_1 = U_1$, $U'_6 := U_6$ and $E'_{ij} := E_{ij}$ for all $E_{ij}$ not yet defined) is still $K$-fat and every branch set $U'_j$ with $j \in \{2,3,4,5\}$ meets $V_e$, so \ref{itm:NiceModelOfO:2} holds.
    \smallskip

    We now further modify our model $(\cu',\ce')$ to satisfy \ref{itm:NiceModelOfO:3} while maintaining \ref{itm:NiceModelOfO:1} and \ref{itm:NiceModelOfO:2}. 
    For this, assume that some $E'_{ij}$ is neither contained in $\Delta$ nor in $\Ext(\Delta)$, for otherwise we are done. In particular, $E'_{ij}$ meets $\Adh(\Delta)$ in an internal vertex $v$.
    Note that since $|\Adh(\Delta)| = 5$ and four branch sets in $\cu'$ meet $\Adh(\Delta)$ by \ref{itm:NiceModelOfO:2}, the branch path $E'_{ij}$ and the vertex $v$ are unique. 
    
    Since $16 \notin E(O)$, we may assume, without loss of generality, that $i \neq \{1,6\}$. 
    Let $Q$ be the (unique) subpath of $E'_{ij}$ between $U'_i$ and $v \in \Adh(\Delta)$. If $d_G(Q, U'_j) \geq K$, which in particular is the case if $j = \{1,6\}$, then we may replace $U'_i$ in the model $(\cu',\ce')$ by $U'_i \cup V(Q)$ and $E'_{ij}$ by the remaining subpath of $E'_{ij}$ that starts in $U'_j$ and ends in an endvertex of $Q$. 
    This is still a $K$-fat model of $O$ by the same argument as above, and it additionally satisfies \ref{itm:NiceModelOfO:3}.
    
    Otherwise, let $u$ be the first vertex on $Q$ that has distance $< K$ from $U'_j$. Let $Q'$ be the subpath of $Q$ between $U'_i$ and $u$,  and let $Q''$ be a $u$--$U'_j$ path of length $K-1$. We then update $E'_{ij}$ to be $E''_{ij} := Q' \cup Q''$. 
    Since $i,j \notin \{1,6\}$, we have $E'_{ij} = E_{ij}$, and thus $E'_{ij}$ has distance $\geq 3K-2$ from all branch sets and paths other than $U'_i, U'_j$. Hence, as $Q'$ has length at most $K-1$, the path $E''_{ij}$ has distance at least $2K-1$ from all other branch sets and paths. Thus, the new model of $O$ (using $E''_{ij}$) is also $K$-fat, and additionally satisfies \ref{itm:NiceModelOfO:3}. 
    \medskip

    \noindent \textbf{Case 2b:} \textit{For every pair $k', \ell' \in V(O)$ as in Case 2 we have $k'\ell' \in E(O)$.}
    
    In particular, $k\ell \in E(O)$, so by the vertex-transitivity of $O$, we may assume that $k=2$ and $\ell=3$. 
    We now modify $(\cu, \ce)$ as follows. We first replace every $U_i$ that intersects $B_G(V_e, K-1)$ with the ball $B_G(U_i, K-1)$ around it (while leaving all other $U_j$'s unchanged). Afterwards, we replace each branch path $E_{ij}$ with a suitable subpath. This yields a new model $(\cu', \ce')$ of $O$, which is $K$-fat since $(\cu, \ce)$ was $3$-fat.
    
    Second, for every pair $k'\ell'$ as in Case~2, we have by the assumption in Case~2b that $k'\ell' \in E(O)$, and hence $E'_{k'\ell'}$ intersects $V_e$. We now modify $E'_{k'\ell'}$ and $U'_{k'}, U'_{\ell'}$ similarly as in Case~2a. For this, let $\{i,j\} = k'\ell'$ be such that $i \neq 2,3$ unless $k'\ell'=k\ell$. Since $d_G(U_j, V_e) \geq K$ by the assumption on $k'\ell'$, the (unique) subpath $Q_i$ of $E'_{ij}$ between $U'_i$ and $V_e$ that is internally disjoint from $V_e$ has distance at least $K$ from $U'_j$. We set $U''_i := U'_i \cup V(Q_i)$ and let $E''_{ij}$ be the remaining subpath of $E'_{ij}$ that starts in the last vertex of $Q_i$ and ends $U'_j$. This yields a new model $(\cu'', \ce'')$ of $O$ (where we left all other branch sets/paths unchanged), which is still $K$-fat. 

    By construction, one of $U''_2$ and $U''_3$ meets $V_e$ (by the second step of the construction as $23$ is an edge as in Case~2). Moreover, as every other vertex of $O$ sends an edge to at least one of $2,3$, every branch set $U''_i$, for $i \in \{1,4,5,6\}$, also meets $V_e$. Indeed, if $U_i$ had distance at most $K-1$ from $V_e$, then $U''_i$ meets $V_e$ because of the first step in the construction. Otherwise, $U_i$ was contained in one side of $\{A_e^x, A_e^y\}$ without the $(K-1)$-ball around $V_e$, and then, by the assumption of Case~2b, $i$ must send an edge to that vertex $j$ in $\{k,\ell\}$ whose branch set $U_j$ is contained in the opposite side of $\{A_e^x, A_e^y\}$. Thus, $U''_i$ meets $V_e$ because of the second step of the construction. 
    Hence, five branch sets of $(\cu'', \ce'')$ intersect $V_e$. As $(\cu'', \ce'')$ is $K$-fat, this contradicts Corollary~\ref{cor:NoFatO:NoAdhesionMeetsFiveBranchSets}. 
\end{proof}

\subsection{Completing the proof of Theorem~\ref{thm:NSSGraph:NoFatO}} \label{sec O proof fin}

We have now gathered enough ingredients to complete our proof: 

\begin{proof}[Proof of Theorem~\ref{thm:NSSGraph:NoFatO}]
    Suppose for a contradiction $G := G_{k,d}$, for some $k,d \in \N$, contains a $9$-fat model $(\cu, \ce)$ of $O$. By Lemma~\ref{lem:SomeAdhesionIntersectsONicely}, $G$ contains a $3$-fat model $(\cu', \ce')$ of $O$ and a trigon $\Delta$ such that \ref{itm:NiceModelOfO:1} to \ref{itm:NiceModelOfO:3} hold. We distinguish two cases.
    \medskip

    \noindent \textbf{Case 1:} \textit{Some branch set in $\cu'$ contains $\RR(\Delta)$.}
    
    By the vertex-transitivity of $O$, we may assume $\RR(\Delta) \in U'_2$. Moreover, by the symmetry of~$G$, we may assume that two of $U'_3, U'_4, U'_5$ intersect $\SSS(\Delta)$. But this yields Configuration~\ref{setting}, which contradicts Lemma~\ref{lem:SettingCannotOccur}.
    \medskip

    \noindent \textbf{Case 2:} \textit{No branch set in $\cu'$ contains $\RR(\Delta)$.}
    
    Let $U_{i}, U_{j}, U_{k}, U_{\ell}$ with $\{i,j,k,\ell\} = \{2,3,4,5\}$ such that $\SO(\Delta), \SI(\Delta) \in U_i \cup U_j$ and $\TI(\Delta), \TO(\Delta) \in U_k \cup U_\ell$. By Proposition~\ref{prop:O:TwoEdgesBetweenTwoPairsOfVertices} and without loss of generality, $ik, j\ell \in E(O)$. Since $(\cu', \ce')$ is $3$-fat, the $\SSS(\Delta)$--$\TTT(\Delta)$ paths $Q_{ik},Q_{j\ell}$ induced by $ik,j\ell$ have distance at least~$3$ from each other. Moreover, $Q_{ik}, Q_{j\ell}$ are each contained in either $\Delta$ or $\Ext(\Delta)$ because the only way for them to cross the separation $\{\Delta, \Ext(\Delta)\}$ would be through $\RR(\Delta)$, but this is prevented by \ref{itm:NiceModelOfO:3} and the assumption of Case 2.

    Suppose first that $Q_{ik}, Q_{j\ell}$ are contained in the same subgraph $H \in \{\Delta, \Ext(\Delta)\}$. 
    Then by Lemma~\ref{lem:TwoPathsInNSSGraph} or~\ref{lem:TwoPathsInNSSGraph:ExteriorOfSubpyramid}, one of them, say $Q_{ik}$, is the unique $\SSS(\Delta)$--$\TTT(\Delta)$ path in $B(G) \cap H$. By \ref{itm:NiceModelOfO:1}, the $\SSS(\Delta)$--$\TTT(\Delta)$ path $Q'_{ik}$ induced by $i1k$ if $H=\Delta$ or by $i6k$ if $H=\Ext(\Delta)$ is also contained in $H$. Since $(\cu', \ce')$ is $3$-fat, also $Q'_{ik}$ has distance at least~$3$ from $Q_{j\ell}$. But then Lemma~\ref{lem:TwoPathsInNSSGraph} or~\ref{lem:TwoPathsInNSSGraph:ExteriorOfSubpyramid} yields that $Q_{ik} = Q'_{ik}$, a contradiction. 
    In particular, we may assume that $Q_{ik} \subseteq \Delta$, and $Q_{j\ell} \subseteq \Ext(\Delta)$. 
    By \ref{itm:NiceModelOfO:1}, the $\SSS(\Delta)$--$\TTT(\Delta)$ path $Q'_{j\ell}$ induced by $j1\ell$ is contained in~$\Delta$ and the $\SSS(\Delta)$--$\TTT(\Delta)$ path $Q'_{ik}$ induced by $i6k$ is contained in $\Ext(\Delta)$. By Lemmas~\ref{lem:TwoPathsInNSSGraph} and~\ref{lem:TwoPathsInNSSGraph:ExteriorOfSubpyramid} and because $(\cu', \ce')$ is $3$-fat, either both of $Q_{ik}, Q_{j\ell}$ (if $\SO(\Delta) \in U_i$) or both of $Q'_{ik}, Q'_{j\ell}$ (if $\SI(\Delta) \in U_i$) have distance at most~$1$ from $\RR(\Delta)$.
    But this contradicts that $(\cu', \ce')$ is $3$-fat.
\end{proof}

\section{Suspension preserves incompressibility} \label{sec susp}

Given a graph \G, we define its \defi{suspension $S(G)$} by adding a vertex $s_G$ and joining $s_G$ to each $v\in V(G)$ with an edge.

\begin{lemma} \label{lem:susp}
Let $G$ be a graph, and $v\in V(S(G))$. Then $G$ is a subgraph of $S(G)-v$. 
\end{lemma}

\begin{proof}
We specify a bijection $b: V(G) \to V(S(G)) \setminus \{v\}$ that preserves edges. If $v$ is the suspension vertex $s_G$, then we take $b$ to be the identity. If not, we set $b(v):= s_G$, and $b(u)=u$ \fe\ $u\neq v \in V(G)$. 
\end{proof}

\begin{theorem} \label{thm:susp}
Let $H$ be a (possibly infinite or disconnected) graph with minimum degree $\delta(H)\geq 3$. If $H$ is \bad, then $S(H)$ is \bad. 
\end{theorem}

\begin{proof}
Let $X$ be a graph witnessing that $H$ is \bad, i.e.\ $H\nasm X$ but $H \preceq X'$ \fe\ graph $X'$ quasi-isometric to $X$.

For each $n\in \N$, we obtain a graph $X_n$ from $X$ by adding a new vertex $a_n$ and joining each vertex $x$ of $X$ to $a_n$ by a path $P_x$ of length \blue{$n$}, so that these paths only intersect at~$a_n$.  

Let $Y:= \bigcup_{\nin} X_n$. We claim that $Y$ witnesses that $S(H)$ is \bad, i.e.\ it satisfies the following two claims: 
\labtequ{SasminY}{$S(H) \nasm Y$.}
and 
\labtequ{SminY}{$S(H)\preceq Y'$ holds \fe\ graph $Y'$ quasi-isometric to $Y$}
To prove \eqref{SasminY}, suppose $(\cu, \ce)$ is a $K$-fat model of $S(H)$ in $Y$ for some $K \in \N$. Since $S(H)$ is connected, $(\cu,\ce)$ lives in one of the components $X_n$ of $Y$. Notice that no branch set $U_v \in \cu$ can be contained in the interior of one of the $X$--$a_n$ paths in the construction of $Y$ because $d(v)\geq 3$ by our assumption on $H$, and so every branch set in $\cu$ contains $a_n$ or a vertex of $X$. If some $U_v$ contains $a_n$, remove it, along with its incident branch paths, to obtain a model $(\cu',\ce')$ of $S(H)-v$ in $Y-a_n$. Similarly, if a branch path contains $a_n$, then we can let $v$ be one of the endvertices of the corresponding edge of $S(H)$, and again we obtain a model $(\cu',\ce')$ of $S(H)-v$ in $Y-a_n$. Note that we can now remove all vertices of $Y-X$ from the branch sets in $\cu'$ without violating their connectedness or the property of being a model of $S(H)-v$.  
This results into a model $(\cu'',\ce'')$ of $S(H)-v$ in $X\subseteq X_n$. Moreover, $(\cu'',\ce'')$ is still $K$-fat in $X$, because $X$ is a subgraph of $X_n$, and each branch set/path of $(\cu'',\ce'')$ is contained in one of $(\cu,\ce)$. We deduce that $S(H)-v$ is a $K$-fat minor of $X$. By \Lr{lem:susp}, $H$ is a $K$-fat minor of $X$ too. As $H\nasm X$ by assumption, we deduce that this cannot hold \fe\ $K\in \N$, proving \eqref{SasminY}.
\medskip

To prove \eqref{SminY}, let $\varphi: V(Y) \to V(Y')$ be an $(M,A)$-quasi-isometry for some $M \geq 1$ and $A \geq 0$. We pick $n\in \N$ large compared to $M,A$; how large will become clear below. 
Let $\iota$ denote the identity map from $X$ to $X_n$. Note that $\iota$ is an \defi{$n$-local isometry}, i.e.\ it satisfies $d_X(x,y)=d_{X_n}(x,y)$ \fe\ $x,y\in X$ with $d_X(x,y)\leq n$ (\fig{figXn} left). Let $\varphi':= \varphi \circ \iota$ and $X':=Y'[\varphi'(X)]$.

Note that \fe\ $xy\in E(X_n)$ we have $d_Y(\varphi(x),\varphi(y))\leq M+A=:C$. Therefore, we can add a $\varphi'(x)$--$\varphi'(y)$ path of length at most $C$ for each $xy\in E(X)$ to $X'$ to obtain a subgraph $X'' \supseteq X'$ of $Y'$ satisfying $d_{X''}(\varphi'(x),\varphi'(y))\leq C$ \fe\ $xy\in E(X)$. Note that below we can choose $n$ large enough so that $\varphi(a_n) \notin X''$, and, moreover, so that $X''$ is disconnected if $X$ is disconnected. 
We claim that 
\labtequ{phip}{$\varphi'$ is a quasi-isometry from $(X,d_X)$ to $(X'',d_{X''})$.} 
The quasi-surjectivity property \ref{quasiisom:2} is straightforward. To check \ref{quasiisom:1}, in one direction we have $$d_{X''}(\varphi'(x),\varphi'(y)) \leq C \cdot d_{X}(x,y)$$ for every $x,y\in V(X)$ by the construction of $X''$.  For the other direction, let $x':=\varphi'(x), y':=\varphi'(y)$, and let $P=(x'=)z_0 \ldots z_k(=y')$ be a \pth{x'}{y'}\ of length $d_{X''}(x',y')$ in $X''$. By the construction of $X''$, we can choose for each vertex $z_i$ of $P$ a vertex $z'_i$ in $X'$ \st\ $d_{Y'}(z_i,z'_i)\leq d_{X''}(z_i,z'_i) \leq C$, and we can thereby choose $z'_0=x'$ and $z'_k=y'$. In particular, since $z'_i \in V(X')$, there exists some $v_i \in V(X)$ such that $\varphi(v_i) = z'_i$ (where we set $v_0 := x$ and $v_k := y$.
Then $d_{Y'}(z'_i,z'_{i+1})\leq C':=2C+1$ \fe\ $i<k$ by the triangle inequality, from which we deduce $d_Y(v_i, v_{i+1}) \leq M(C'+A) =: C''$. 
Choosing $n>C''$, and recalling that $\iota$ is a $n$-local isometry, and applying the triangle inequality $k$ times, we obtain 
$$d_X(x,y)=d_X(v_0, v_k) \leq C'' k= C'' d_{X''}(x',y').$$ 
This completes the proof of \eqref{phip}.

Let $a:=\varphi(a_n)$, and let $W$ be the component of $Y'\sm X''$ containing $a$, which exists when $n$ is larger than $M(C+A)$, because $a_n$ is then too far from $X$ for $a$ to be contained in $X''$ (as each vertex in $X''$ is at distance at most $C$ ($\lfloor C/2\rfloor$) from the the $\varphi'$-image of a vertex of $X$). Let $Z$ be the set of vertices of $X''$ sending an edge to $W$. We claim that 
\labtequ{dvZ}{$d_{Y'}(v,Z) \leq M^2(2C + A) + A + 2C$ \fe\ $v\in V(X'')$.}
(It only matters that $d(v,Z)$ is bounded.)
Let us first check this for the case that $v=\varphi'(u)$ for some $u\in V(X)$. Let $P$ be the $u$--$a_n$~path of length at least $n$ in the construction of $X_n$. Then $B_{Y'}(\varphi(V(P)), C)$ is connected by Lemma~\ref{lem:QIPreservesConn}, and thus there exists a $v$--$a$~path $P'$ in $Y'$ within distance $C$ from $\varphi(V(P))$.

\begin{figure}[ht] 
\begin{center}
\begin{overpic}[width=1\linewidth]{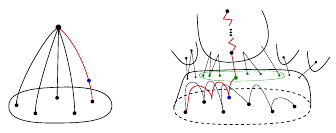} 
\put(60,31){$W$}
\put(67,37.5){$a$}
\put(16,33.5){$a_n$}
\put(82,14){\color{ForestGreen} $Z$}
\put(24,24){$P$}
\put(70,30){$P'$}
\put(71,16){$p$}
\put(67,24){$q$}
\put(28,15.5){$b$}
\put(67.1,7.8){$\varphi(b)$}
\put(27,7){$u$}
\put(56,5){\footnotesize $v = \varphi'(u)$}
\put(0,8){$X$}
\put(4,28){$X_n$}
\put(85,32){$Y'$}
\put(93,6){$X'$}
\put(93,13){$X''$}
\end{overpic}
\end{center}
\caption{The graph $X_n$ (left), and its image under $\varphi$ in the proof of \eqref{dvZ} (right). The green vertices depict $Z$, and the red curve depicts $P'$.} 
\label{figXn}
\end{figure}
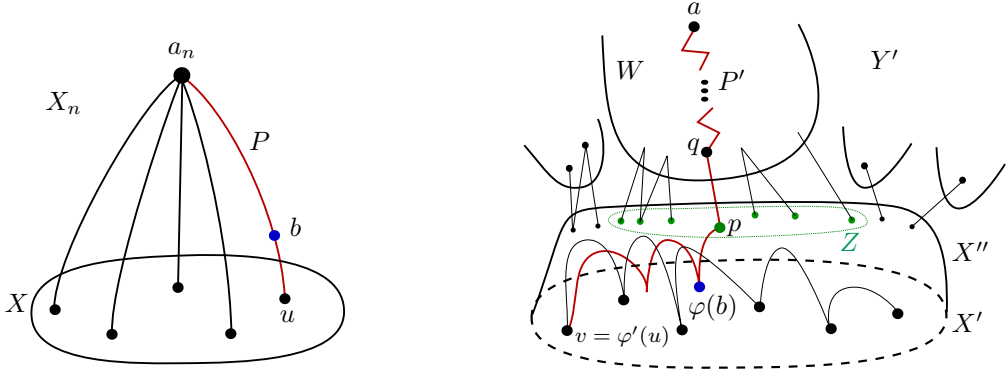

Let $q$ be the first vertex of $P'$ in $W$, which exists since $P'$ ends in $a\in W$. 
As $P'$ starts in $X''$, the vertex $q$ is not the first vertex of $P'$. Let $p$ be the previous vertex of~$P'$, which thus lies in $X''$. In fact, $p$ lies in $Z$ as witnessed by the edge $pq$. 

Note that $d_{Y'}(p,\varphi(P))\leq C$ since $p\in V(P')$. Thus we can pick $b\in V(P)$ with 
\labtequ{dpb}{$d_{Y'}(p,\varphi(b))\leq C$.}
Moreover, $d_{Y'}(p,X')\leq C$ since $p \in X''$.
Combining these two inequalities we deduce $d_{Y'}(\varphi(b),X')\leq 2C$. 
Since $X'=\varphi'(X)$, and $\varphi|_{X_n}$ is an $(M,A)$-quasi-isometry, it follows that $d_{X_n}(b,X)\leq M(2C +A)$. 

In fact, it follows that $d_{X_n}(b,u) \leq M(2C+A)$: since $b \in P_u$, every $b$--$X$ path in $X_n$ that avoids $u$ has to contain $a_n$ and a complete path $P_x$ for some $x \in X_n$. As we chose $n$ larger than $M(2C+1)$, each such path has length at least $M(2C+1)$, and thus $d_{X_n}(b,u) = d_{X_n}(b,X)$.

Applying again the quasi-isometry bound in the reverse direction, we further obtain 
$d_{Y'}(\varphi(b),v=\varphi(u)) \leq M^2(2C +A)+A$, and combining this with \eqref{dpb} yields $d_{Y'}(p,v) \leq M^2(2C +A)+A+ C$.

For the general case $v\in V(X'')$, recall that, by the definition of $X''$, each $v\in V(X'')$ lies within distance $C$ from $X'$. Thus \eqref{dvZ} follows by adding $C$ to the right-hand side of the above bound. 
\smallskip

By \eqref{dvZ} we can now decompose $X''$ into a family of connected subgraphs of bounded diameters, with each subgraph containing an element of $Z$ using an idea similar to the Voronoi decomposition. For this, we construct a breadth-first spanning forest $F$ of  $X''$ using $Z$ as the set of roots of the trees in the forest. To make this precise, we construct $F$ recursively as follows. We start with $F_0$ being the graph with vertex set $Z$ and no edges. Having constructed $F_i$, we obtain $F_{i+1}$ by adding, for each $v\in V(X'')$ with $d_{X''}(v,F_i)=1$, an edge from $v$ to some vertex of $F_i$ chosen arbitrarily. Note that this process terminates after a bounded number of steps by \eqref{dvZ}, returning a spanning forest $F$ of  $X''$, each component of which contains exactly one vertex of $Z$, and has bounded radius in $X''$ (the bound is $M'B+A'+C$ where $B$ is the bound from \eqref{dvZ} and $M', A' \in \N$ with $d_{X''}(x,y) \leq M'\cdot d_{Y'}(x,y) + A'$ for all $x,y \in X'$ can be obtained by observing that $X''\subseteq B_{Y'}(X',C)=B_{Y'}(\varphi'(X),C)$ and $d_{X''}(\varphi'(x),\varphi'(y)) \leq C\cdot d_X(x, y)+C$ for all $x,y \in X$ since $\varphi$ is a $(C,C)$-quasi-isometry, and $d_X(x,y) \leq M\cdot d_{Y'}(\varphi(x), \varphi(y))+A$). 

Let $\kreis{X}$ be the minor of $X''$ obtained by contracting each component of $F$ into a point. Since each contracted set has bounded diameter, it follows that $\kreis{X}$ is quasi-isometric to $X''$, see e.g. \cite[{Lemma~2.2}]{ADGK2t}. 

Thus $\kreis{X}$ is quasi-isometric to $X$ too by \eqref{phip}, from which it follows that $H \preceq \kreis{X}$ since $X$ witnesses that $H$ is \bad. Pick a model \cb\ of  $H$ in $\kreis{X}$, and expand it into a model $\cb'$ of  $H$ in $X''$ by uncontracting each component of $F$ corresponding to a vertex of $\kreis{X}$ in \cb. 
We can further extend $\cb'$  into a model of  $S(H)$ in $Y'$ by using $W$ as the new branch set, recalling that each vertex in $Z$, and therefore each branch set in $\cb'$, sends an edge to $W$. This completes the proof of \eqref{SminY}.
\end{proof}

Combining this with Theorem~\ref{thm:Counterexample:Octahedron} easily implies the first half of Corollary~\ref{cor Kt}, i.e.\ that $K_t$ is \bad\ \fe\ $t\geq 6$. A small adaptation of the proof of Theorem~\ref{thm:susp} handles the second half: 

\begin{corollary}
$K_{s,t}$ is \bad\  \fe\ $s,t\geq 4$. 
\end{corollary}

\begin{proof}
    Let $X:=G_\infty (=\dot{\bigcup}_{n \in \N} G_{n,n})$ denote the infinite NSS graph as above, and recall that $K_{4,t}, t\geq 4$ is a minor, but not a $9$-fat minor, of $X$ by Theorems~\ref{thm:NSSGraph:K4t} and~\ref{thm:NSSGraph:NoFatO}, which settles the case $s=4$ by Theorem~\ref{thm:css:finite:maintext}. Thus, it remains to show that $K_{s'+4,t'+4}$ is \bad\ for $t'\geq s' \geq 1$. 

    We repeat the construction of $Y$ based on $X$ from the proof of Theorem~\ref{thm:susp} $n$ times. More precisely, let $Y_1$ be the graph $Y$ constructed as in the proof of Theorem~\ref{thm:susp} for $X=G_\infty$, and given $Y_i$, for $1 < i < s'$, let $Y_{i+1}$ be the graph $Y$ constructed as in the proof of Theorem~\ref{thm:susp} for $X=Y_i$. Since $K_{4,t'+4} \preceq X'$ for every graph $X'$ quasi-isometric to $X$, it follows by applying \eqref{SminY} $s'$ times (first to $Y_1$, then to $Y_2$, and so on) that $S^{s'}(K_{4,t'+4}) \preceq Y'$ holds for every graph $Y'$ quasi-isometric to $Y_{s'}$ where $S^{s'}(K_{4,t'+4})$ is the $s'$-fold suspension of $K_{4,t'+4}$, i.e.\ $S^{s'}(K_{4,t'+4}) = S(S(\ldots S(K_{4,t'+4})))$, nested $s'$-times. Since $K_{s'+4,t'+4} \preceq S^{s'}(K_{4,t'+4})$, this shows that $K_{s'+4,t'+4} \preceq Y'$ for every graph $Y_{s'}$ quasi-isometric to $Y_{s'}$.

    On the other hand, similarly to the proof of \eqref{SasminY}, if $K_{s'+4,t'+4} \asm Y_{s'}$ were true, then removing the branch sets or branch paths using the $s'$ `apex' vertices from a $9$-fat model of $K_{s'+4,t'+4}$ in $Y_{s'}$, we would obtain a $9$-fat model of (a supergraph of) $K_{4,t'+4-s'}$ in $X$, which is impossible since $t'+4-s'\geq 4$ and $K_{4,4}$ is not a $9$-fat minor of $X$. 
\end{proof}

\section{Asymptotic \texorpdfstring{$K_{3,t}$}{K3t}  and \texorpdfstring{$K_{5}$}{K5} minors} \label{sec K3t}

As mentioned in the introduction, we know that $K_{2,t}$ are compressible \cite{ADGK2t}, i.e.\ satisfy \Cnr{conj fat min}, but $K_{4,t}$, with $t\geq 4$, are incompressible (Theorem~\ref{thm:FatK4t}). The aim of this section is to show that if $K_5$ or $K_{3,t}$ are incompressible, i.e.\ fail to satisfy the conjecture, which is an interesting open problem, then this cannot be proved by using the NSS graphs, as we did for $K_{4,t}$.  
The following proposition makes Proposition~\ref{prop:K3t:Intro} precise. 

\begin{proposition} \label{prop:K3t}
    For every $K,t\in \N$ \ti\ $k \in \N$ \st\ $G_{k,K}$ contains $K_{3,t}$ as a $K$-fat minor.
\end{proposition}

We remark that it follows easily from the proof that Proposition \ref{prop:K3t} is still true if we make the side of $K_{3,t}$ of size $3$ complete (this can be done already in the first step of the construction).

\begin{proof}
    Fixing $K\in \N$, we prove the statement by induction on $t$. Our inductive hypothesis strengthens the statement by imposing some conditions on the three branch sets $B_1,B_2,B_3$ corresponding to the partition side of $K_{3,t}$ of size 3, namely 
    \labtequ{indhypK3t}{$G_t:= G_{k(K,t),K}$ contains a $K$-fat model of $K_{3,t}$ \st\ (Figure~\ref{fig:K3t:1})}
    \begin{enumerate}
        \item \label{it i} $B_1$ contains $\SO(G_t)$;
        \item \label{it ii} $B_2$ contains $\TI(G_t)$; and
        \item \label{it iii} $B_3$ contains the (unique) path of the binary tree $B(G_t)$ between $\RR(G_t)$ and $\TO(G_t)$. 
    \end{enumerate}
    \begin{figure}[ht]
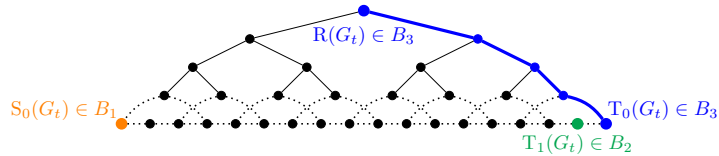

        \centering
        \include{K3t_1}
        \vspace{-3em}
        \caption{$G_t$ and part of our $K$-fat model of $K_{3,t}$ in the inductive hypothesis of \Prr{prop:K3t}}
        \label{fig:K3t:1}
    \end{figure}
    We start the induction with $t=0$, in which case our model consists of the three branch sets $B_1,B_2,B_3$ only. We choose $k=k(K,0)=1$ (and $d=K$), and define the $B_i$ by replacing `contains' by `consists of' in \ref{it i}--\ref{it iii}.

    For the inductive step, assuming we have found $G_t$ as in \eqref{indhypK3t}, we proceed to construct $G_{t+1}$ as follows. 
    We increase the height $k$ by $4K+1$, i.e.\ we let $k(K,t+1):=k(K,t)+4K+1$. By our recursive construction, $G_{t+1}$ is obtained by combining two copies $\Delta_0 := \SP(G_{t+1})$ and $\Delta_1 := \TP(G_{t+1})$ of $G_{k(K,t+1)-1,K}$. We identify $G_t$ with the leftmost trigon $\Delta'_1$ of $\Delta_1$ of height $k(K,t)$, i.e.\ we identify $\SO(G_t)$ with $\SO(\Delta_1)$ (indicated in grey in Figure~\ref{fig:K3t:2}). We will construct our model of $K_{3,t+1}$ by expanding the model of $K_{3,t}$ inside $\Delta'_1$ obtained from identifying the latter with $G_t$, except that we `flip' the model of $K_{3,t}$ vertically inside $\Delta'_1$, i.e.\ we compose the model with the automorphism of $G_t$ that exchanges $\SO(G_t)$ with $\TO(G_t)$. After this flipping we have $\TO(\Delta'_1) \in B_1, \SI(\Delta'_1) \in B_2$ and $\SO(\Delta'_1) \in B_3$ (see Figure~\ref{fig:K3t:2}).
    
    To construct our model of $K_{3,t+1}$, we will keep the branch sets corresponding to the $t$ side intact, we will introduce one more branch set $B^{t+1}$ in that side, and we will expand the branch sets $B_1,B_2,B_3$ corresponding to the side of size three, and join each $B_i, i\leq 3$ to $B^{t+1}$ by a branch path. This is summarized in Figure~\ref{fig:K3t:2}. 

    \begin{figure}[ht]
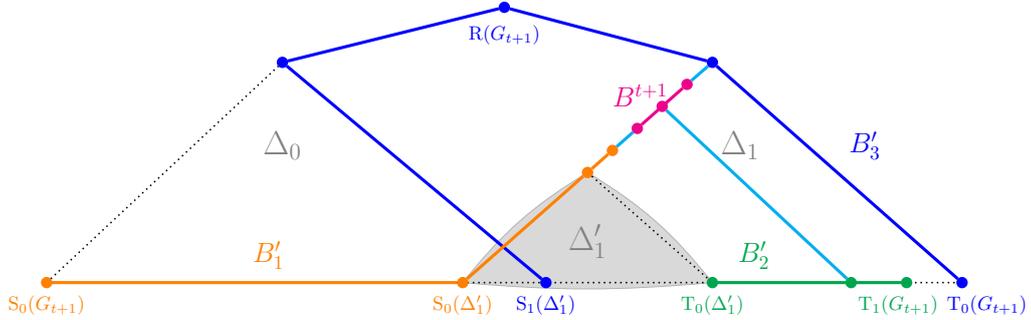

        \centering
        \include{K3t_2}
        \vspace{-3em}
        \caption{Performing the inductive step to construct $G_{t+1}$. The grey area represents a flipped (and recoloured) copy of $G_t$ containing a model of $K_{3,t}$.}
        \label{fig:K3t:2}
    \end{figure}

    Let $N$ be the tree path in $\Delta_1$ from $\RR(\Delta'_1)$ to $\RR(\Delta_1)$. Notice that $N$ has length $5K$, and let $m,n$ be the $K$-th and $3K$-th vertex of $N$, respectively.
    
    Recall that $\TO(\Delta'_1) \in B_1$  by \ref{it i} after flipping. We expand $B_1$ by the subpath $P_1$ of the bottom path of $G_{t+1}$ to $\TI(G_{t+1})$. 
    The new branch set $B'_2:= B_1 \cup P_1$ thus obtained will play the role of $B_2$ in our model of $K_{3,t+1}$ in $G_{t+1}$; this satisfies \ref{it ii} since $\TI(G_{t+1})\in B'_2$. 
    Recall that $\SI(\Delta'_1) \in B_2$.  We expand $B_2$ by the tree path $P_2^1$ in $P$ from $\SI(\Delta'_1)= \TO(\Delta_0)$ to $\RR(\Delta_0)$, the edge $e$ from $\RR(\Delta_0)$ to its parent $\RR(G_{t+1})$, and the tree path $P_2^2$ in $\Delta_0$ from $\RR(G_{t+1})$ to $\TO(G_{t+1})$. The new branch set $B'_3:= B_2\cup P_2^1 \cup \{e\} \cup P_2^2$ will play the role of $B_3$ in our new model; this satisfies \ref{it iii} because $P_2^2\subset B'_3$. 
    Recall that $\SO(Q') \in B_3$. We expand $B_3$ by the subpath $P_3$ of the bottom path of $P$ from $\SO(\Delta'_1)$ to $\SO(G_{t+1})$ and by the subpath $P'_3$ of $N$ from $\RR(\Delta'_1)$ to $m$. The new branch set $B'_1:= P_3 \cup B_3 \cup P'_3$ will play the role of $B_1$ in our new model; it satisfies \ref{it i}.

    To complete our model, it remains to define the new branch set $B^{t+1}$ and its incident branch paths. We let $B^{t+1}$ be the subpath of $N$ of length $2K$ having $n$ as its middle vertex. The two subpaths of $N$ (of length $K$ each) from $m$ to $B^t$ and from $B^t$ to $\RR(\Delta_1)$ are two of the desired branch paths: one of them joins $B^{t+1}$ to $m \in B'_1$, and the other joins $B^{t+1}$ to $\RR(G_t)\in B'_3$. Our last branch path is the rightward tree path in $\Delta_1$ from $n$ to the bottom path, which finishes at $B'_2$. 
    
    This completes the definition of our model of $K_{3,t+1}$ in $G_{t+1}$, which by construction satisfies our inductive hypothesis. Indeed, all distances between the various parts of this model are at least $K$. 
\end{proof}

\subsection{Asymptotic \texorpdfstring{$K_{5}$}{K5} minors} \label{sec:K5}

As mentioned in the introduction, we do not know whether $K_5$ is compressible, i.e.\ satisfies Conjecture~\ref{conj fat min}. The following proposition, similar to Proposition~\ref{prop:K3t}, shows that if $K_5$ is incompressible, then again this cannot be proved by using NSS graphs.

\begin{proposition} \label{prop:FatK5}
    For every $K \in \N$ there exists some $k \in \N$ such that $G_{k,K}$ contains $K_5$ as a $K$-fat minor.
\end{proposition}

\begin{proof}[Proof (Sketch)]
    Figure~\ref{fig:FatK5} shows the branch sets of a $K$-fat model of $K_5$ in $G_{k,K}$ for every large enough $k \in \N$. Each of the five branch sets is depicted with a different colour. Note that the dashed lines represent paths of $B(G_{k,K})$ of length at least $K$. The branch paths of the $K$-fat model of $K_5$ are not explicitly shown, but they can be chosen so that each of them is some dashed or dotted path in Figure~\ref{fig:FatK5}. 
\end{proof}

\begin{figure}[h]
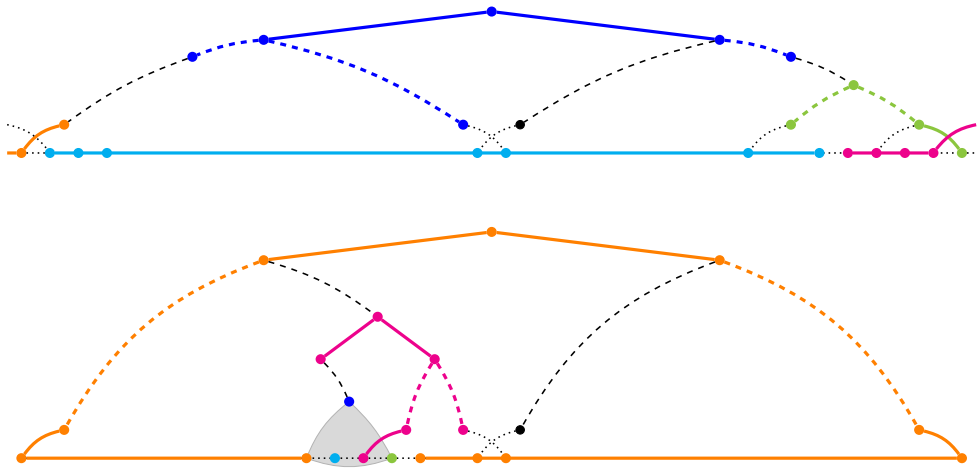

    \centering
    \begin{subfigure}{1\linewidth}
    \centering
    \include{FatK5_1}
    \end{subfigure}
    \begin{subfigure}{1\linewidth}
    \centering
    \include{FatK5_2}
    \end{subfigure}
    \vspace{-2em}
    \caption{Depicted are the branch sets of a $K$-fat model of $K_5$ in some large enough NSS graph $G$. The dashed lines represent paths in $B(G)$ of length at least $K$. The grey area in the bottom picture represents the graph from the top picture.}
    \label{fig:FatK5}
\end{figure}

\section{Further problems}

Let $J$ be a finite or infinite graph, and $X$ an infinite graph or a family of finite (or infinite) graphs. If $X$ is a family, we say that $J$ is a ($K$-fat) minor of $X$ if $J$ is a ($K$-fat) minor of some element of $X$. 

Recall that $J$ is an \defi{asymptotic minor} of $X$, if $J$ is a $K$-fat minor of $X$ \fe\ $K\in\N$. In this case we write $J \asm X$.
We say that $J$ is a \defi{coarse minor} of $X$, and write $J \cmin X$, if $J\preceq X'$ holds \fe\ graph (family) $X'$ quasi-isometric to $X$. Finally, we say that  $J$ is a \defi{thin minor} of $X$, and write $J \thmin X$, if $J\cmin X'$ but $J\nasm X$. Thus  
\labtequ{bad}{$J$ is \bad\ ---i.e.\ a counterexample to \Cnr{conj fat min}--- \iff\ there is $X$ \st\ $J \thmin X$.}
As already mentioned, all known incompressible graphs are proved to be such using the NSS graphs. The following question asks whether this must be the case. Recall that $G_\infty$ denotes the infinite NSS graph $\dot{\bigcup}_{\nin} G_{n,n}$.

\begin{question} \label{Q NSS}
Suppose $J \thmin X$ for some graphs $J,X$. Must $G_\infty \cmin X$ hold?
\end{question}

Analogously to Question~\ref{Q NSS}, one can ask
\begin{question} \label{Q NSS 2}
Suppose $X$ is a counterexample to the coarse Menger conjecture. 
Must $G_\infty \cmin X$ hold?
\end{question}

Recall that the coarse Menger conjecture \cite{AHJKW,GeoPapMin} asserts that for every $d, k \in \N$, there exists some $r \in \N$ such that for every graph $G$ and two sets $X, Y \subseteq V(G)$ there are either $k$ paths between $X$ and $Y$ at distance at least $d$ from each other, or there is a set $Z \subseteq V(G)$ of size at most $k-1$ such that every $X$--$Y$ path in $G$ meets $B_G(Z, r)$.
This conjecture was disproven by Nguyen, Scott and Seymour \cite{NgScSeCou} using $G_\infty$. Hence, Question~\ref{Q NSS 2} asks whether this $G_\infty$ is essentially the only counterexample.

\bibliographystyle{plain}
\bibliography{collective}

\end{document}

%% file: defs.tex
\usepackage{amsthm,amssymb,amsmath,bbm,enumerate,graphicx,epsf,stmaryrd,accents}

\usepackage{authblk}

\newcommand{\comment}[1]{}
\newcommand{\COMMENT}[1]{}

\definecolor{darkgray}{rgb}{0.3,0.3,0.3}
\definecolor{lgreen}{rgb}{0.065, 0.225, 0.065}
\newcommand{\defi}[1]{{\color{darkgray}\emph{#1}}}
\newcommand{\defim}[1]{{\color{darkgray}{#1}}}

\comment{
	\begin{lemma}\label{}	
\end{lemma}
\begin{proof}

\end{proof}

\begin{theorem}\label{}
\end{theorem} 
\begin{proof} 	

\end{proof}

}

\newtheorem{proposition}{Proposition}[section]

\newtheorem{theorem}[proposition]{Theorem}
\newtheorem{corollary}[proposition]{Corollary}

\newtheorem{lemma}[proposition]{Lemma}
\newtheorem{observation}[proposition]{Observation}
\newtheorem{conjecture}{{Conjecture}}[section]

\newtheorem{question}[conjecture]{{Question}}

\newtheorem{examp}[proposition]{Example}

\newtheorem{configuration}[proposition]{Configuration}

\theoremstyle{definition}
\newtheorem{construction}[proposition]{Construction}

\newcommand{\kreis}[1]{\mathaccent"7017\relax #1}

\newcommand{\FIG}{0}

\ifnum \NOTESON = 1 \newcommand{\note}[1]{ 

\hspace*{-30pt}
	{\color{blue}  NOTE: \color{Turquoise}{\small  \tt \begin{minipage}[c]{1.1\textwidth}  #1 \end{minipage} \ignorespacesafterend }} 
	
	}
\else \newcommand{\note}[1]{} \fi

\newcommand{\afsubm}[1]{ \ifnum \Debug = 1 {\mymargin{#1}}
\fi} 

\ifnum \Debug = 1 
\else  \fi

\ifnum \FIG = 1 \newcommand{\fig}[1]{Figure ``{#1}''}
\else \newcommand{\fig}[1]{Figure~\ref{#1}} \fi

\ifnum \FIG = 1 
\else  \fi

\ifnum \Debug = 1 \usepackage[notref,notcite]{showkeys}
\fi

\ifnum \COLORON = 0 \renewcommand{\color}[1]{}
\fi

\newcommand{\N}{\ensuremath{\mathbb N}}
\newcommand{\R}{\ensuremath{\mathbb R}}

\newcommand{\cb}{\ensuremath{\mathcal B}}

\newcommand{\ce}{\ensuremath{\mathcal E}}

\newcommand{\cp}{\ensuremath{\mathcal P}}

\newcommand{\cu}{\ensuremath{\mathcal U}}
\newcommand{\cv}{\ensuremath{\mathcal V}}

\newcommand{\sm}{\backslash}

\newcommand{\rad}{\ensuremath{\text{rad}}}

\makeatletter
\DeclareRobustCommand{\cev}[1]{%
  \mathpalette\do@cev{#1}%
}
\newcommand{\do@cev}[2]{%
  \fix@cev{#1}{+}%
  \reflectbox{$\m@th#1\vec{\reflectbox{$\fix@cev{#1}{-}\m@th#1#2\fix@cev{#1}{+}$}}$}%
  \fix@cev{#1}{-}%
}
\newcommand{\fix@cev}[2]{%
  \ifx#1\displaystyle
    \mkern#23mu
  \else
    \ifx#1\textstyle
      \mkern#23mu
    \else
      \ifx#1\scriptstyle
        \mkern#22mu
      \else
        \mkern#22mu
      \fi
    \fi
  \fi
}

\makeatother

\newcommand{\nin}{\ensuremath{{n\in\N}}}

\newcommand{\pth}[2]{\ensuremath{#1}\text{--}\ensuremath{#2}~path}

\newcommand{\g}{\ensuremath{G\ }}
\newcommand{\G}{\ensuremath{G}}

\newcommand{\Lr}[1]{Lemma~\ref{#1}}

\newcommand{\Tr}[1]{Theorem~\ref{#1}}
\newcommand{\Trs}[1]{Theorems~\ref{#1}}
\newcommand{\Sr}[1]{Section~\ref{#1}}

\newcommand{\Prr}[1]{Pro\-position~\ref{#1}}

\newcommand{\Cnr}[1]{Con\-jecture~\ref{#1}}
\newcommand{\Or}[1]{Observation~\ref{#1}}

\renewcommand{\iff}{if and only if}
\newcommand{\fe}{for every}

\newcommand{\st}{such that}

\newcommand{\ti}{there is}

\newcommand{\obda}{without loss of generality}

\newcommand{\leth}{large enough that}

\newcommand{\labtequ}[2]{
 \begin{equation} \label{#1} 	\begin{minipage}[c]{0.9\textwidth}  #2 \end{minipage} \ignorespacesafterend \end{equation} } 

\newcommand{\labtequtag}[3]{%
 \begin{equation} \label{#1} 	\begin{minipage}[c]{0.9\textwidth}  #2 \end{minipage} \ignorespacesafterend \tag{#3} \end{equation} }

\newcommand{\mymargin}[1]{
 \ifnum \Debug = 1
  \marginpar{%
    \begin{minipage}{\marginparwidth}\small%
      \begin{flushleft}%
        {\color{blue}#1}%
      \end{flushleft}%
   \end{minipage}%
  }%
 \fi
}%

\newcommand{\extras}[1]{
 \ifnum \Debug = 1
\section{Extras} #1
 \fi
}%

\newcommand{\mySection}[2]{}



%% file: NSSGraph.tex
\scalebox{0.72}{%
\begin{tikzpicture}[scale=1/2,auto=left]

\tikzstyle{every node}=[inner sep=1.5pt, fill=black,circle,draw]
\node[blue,ultra thick] (v1) at (1,0) {};
\node[blue, ultra thick] (v34) at (34,0) {};
\draw[dotted, thick] (v1)--(v34);
\node[blue,ultra thick] (v2) at (2,0) {};
\node (v3) at (3,0) {};
\node (v4) at (4,0) {};
\node (v5) at (5,0) {};
\node (v6) at (6,0) {};
\node (v7) at (7,0) {};
\node (v8) at (8,0) {};
\node (v9) at (9,0) {};
\node (v10) at (10,0) {};
\node (v11) at (11,0) {};
\node (v12) at (12,0) {};
\node (v13) at (13,0) {};
\node (v14) at (14,0) {};
\node (v15) at (15,0) {};
\node (v16) at (16,0) {};
\node (v17) at (17,0) {};
\node (v18) at (18,0) {};
\node (v19) at (19,0) {};
\node (v20) at (20,0) {};
\node (v21) at (21,0) {};
\node (v22) at (22,0) {};
\node (v23) at (23,0) {};
\node (v24) at (24,0) {};
\node (v25) at (25,0) {};
\node (v26) at (26,0) {};
\node (v27) at (27,0) {};
\node (v28) at (28,0) {};
\node (v29) at (29,0) {};
\node (v30) at (30,0) {};
\node (v31) at (31,0) {};
\node (v32) at (32,0) {};
\node[blue, ultra thick] (v33) at (33,0) {};

\node (u2) at (2.5,1) {};
\node (u4) at (4.5,1) {};
\node (u6) at (6.5,1) {};
\node (u8) at (8.5,1) {};
\node (u10) at (10.5,1) {};
\node (u12) at (12.5,1) {};
\node (u14) at (14.5,1) {};
\node (u16) at (16.5,1) {};
\node (u18) at (18.5,1) {};
\node (u20) at (20.5,1) {};
\node (u22) at (22.5,1) {};
\node (u24) at (24.5,1) {};
\node (u26) at (26.5,1) {};
\node (u28) at (28.5,1) {};
\node (u30) at (30.5,1) {};
\node (u32) at (32.5,1) {};

\draw[dotted,thick] (u2) to [bend right=20] (v1);
\draw[dotted,thick] (u4) to [bend right=20] (v3);
\draw[dotted,thick] (u6) to [bend right=20] (v5);
\draw[dotted,thick] (u8) to [bend right=20] (v7);
\draw[dotted,thick] (u10) to [bend right=20] (v9);
\draw[dotted,thick] (u12) to [bend right=20] (v11);
\draw[dotted,thick] (u14) to [bend right=20] (v13);
\draw[dotted,thick] (u16) to [bend right=20] (v15);
\draw[dotted,thick] (u18) to [bend right=20] (v17);
\draw[dotted,thick] (u20) to [bend right=20] (v19);
\draw[dotted,thick] (u22) to [bend right=20] (v21);
\draw[dotted,thick] (u24) to [bend right=20] (v23);
\draw[dotted,thick] (u26) to [bend right=20] (v25);
\draw[dotted,thick] (u28) to [bend right=20] (v27);
\draw[dotted,thick] (u30) to [bend right=20] (v29);
\draw[dotted,thick] (u32) to [bend right=20] (v31);

\draw[dotted,thick] (u2) to [bend left=20] (v4);
\draw[dotted,thick] (u4) to [bend left=20] (v6);
\draw[dotted,thick] (u6) to [bend left=20] (v8);
\draw[dotted,thick] (u8) to [bend left=20] (v10);
\draw[dotted,thick] (u10) to [bend left=20] (v12);
\draw[dotted,thick] (u12) to [bend left=20] (v14);
\draw[dotted,thick] (u14) to [bend left=20] (v16);
\draw[dotted,thick] (u16) to [bend left=20] (v18);
\draw[dotted,thick] (u18) to [bend left=20] (v20);
\draw[dotted,thick] (u20) to [bend left=20] (v22);
\draw[dotted,thick] (u22) to [bend left=20] (v24);
\draw[dotted,thick] (u24) to [bend left=20] (v26);
\draw[dotted,thick] (u26) to [bend left=20] (v28);
\draw[dotted,thick] (u28) to [bend left=20] (v30);
\draw[dotted,thick] (u30) to [bend left=20] (v32);
\draw[dotted,thick] (u32) to [bend left=20] (v34);

\node (t3) at (3.5,2) {};
\node (t7) at (7.5,2) {};
\node (t11) at (11.5,2) {};
\node (t15) at (15.5,2) {};
\node (t19) at (19.5,2) {};
\node (t23) at (23.5,2) {};
\node (t27) at (27.5,2) {};
\node (t31) at (31.5,2) {};

\draw (u2) -- (t3)--(u4);
\draw (u6) -- (t7)--(u8);
\draw (u10) -- (t11)--(u12);
\draw (u14) -- (t15)--(u16);
\draw (u18) -- (t19)--(u20);
\draw (u22) -- (t23)--(u24);
\draw (u26) -- (t27)--(u28);
\draw (u30) -- (t31)--(u32);

\node (s5) at (5.5,3) {};
\node (s13) at (13.5,3) {};
\node (s21) at (21.5,3) {};
\node (s29) at (29.5,3) {};

\draw (t3) -- (s5)--(t7);
\draw (t11) -- (s13)--(t15);
\draw (t19) -- (s21)--(t23);
\draw (t27) -- (s29)--(t31);

\node (r9) at (9.5,4) {};
\node (r25) at (25.5,4) {};

\draw (s5) -- (r9)--(s13);
\draw (s21) -- (r25)--(s29);

\node (q17) at (17.5,5) {};
\draw (r9) -- (q17)--(r25);

\tikzstyle{every node}=[]
\draw[below] (q17) node []           {$\RR(G)$};
\draw[blue,below left] (v1) node []           {\small{$\SO(G)$}};
\draw[blue,below] (v2) node []           {\small{$\SI(G)$}};
\draw[blue,below] (v33) node []           {\small{$\TI(G)$}};
\draw[blue,below right] (v34) node []           {\small{$\TO(G)$}};

\end{tikzpicture}
}

%% file: NSSGraph_Subpyramid.tex
\begin{tikzpicture}[scale=0.3,auto=left]

\node[inner sep=2, blue,fill=blue,circle,draw] (v1) at (2,0) {};
\node[inner sep=2, blue,fill=blue,circle,draw] (v2) at (4,0) {};
\node[inner sep=2, blue,fill=blue,circle,draw] (v9) at (18,0) {};
\node[inner sep=2, blue,fill=blue,circle,draw] (v10) at (20,0) {};

\tikzstyle{every node}=[inner sep=1.5pt, fill=black,circle,draw]
\node[blue,fill=blue] (v3) at (6,0) {};
\node[blue,fill=blue] (v4) at (8,0) {};
\node[blue,fill=blue] (v5) at (10,0) {};
\node[blue,fill=blue] (v6) at (12,0) {};
\node[blue,fill=blue] (v7) at (14,0) {};
\node[blue,fill=blue] (v8) at (16,0) {};
\node (v11) at (22,0) {};
\node (v12) at (24,0) {};
\node (v13) at (26,0) {};
\node (v14) at (28,0) {};
\node (v15) at (30,0) {};
\node (v16) at (32,0) {};
\node (v17) at (34,0) {};
\node (v18) at (36,0) {};

\draw[blue,dotted, thick] (v1)--(v10);
\draw[dotted, thick] (v10)--(v18);

\node[blue,fill=blue] (u2) at (5,1.5) {};
\node[blue,fill=blue] (u4) at (9,1.5) {};
\node[blue,fill=blue] (u6) at (13,1.5) {};
\node[blue,fill=blue] (u8) at (17,1.5) {};
\node (u10) at (21,1.5) {};
\node (u12) at (25,1.5) {};
\node (u14) at (29,1.5) {};
\node (u16) at (33,1.5) {};

\draw[blue,dotted,thick] (u2) to [bend right=20] (v1);
\draw[blue,dotted,thick] (u4) to [bend right=20] (v3);
\draw[blue,dotted,thick] (u6) to [bend right=20] (v5);
\draw[blue,dotted,thick] (u8) to [bend right=20] (v7);
\draw[dotted,thick] (u10) to [bend right=20] (v9);
\draw[dotted,thick] (u12) to [bend right=20] (v11);
\draw[dotted,thick] (u14) to [bend right=20] (v13);
\draw[dotted,thick] (u16) to [bend right=20] (v15);

\draw[blue,dotted,thick] (u2) to [bend left=20] (v4);
\draw[blue,dotted,thick] (u4) to [bend left=20] (v6);
\draw[blue,dotted,thick] (u6) to [bend left=20] (v8);
\draw[blue,dotted,thick] (u8) to [bend left=20] (v10);
\draw[dotted,thick] (u10) to [bend left=20] (v12);
\draw[dotted,thick] (u12) to [bend left=20] (v14);
\draw[dotted,thick] (u14) to [bend left=20] (v16);
\draw[dotted,thick] (u16) to [bend left=20] (v18);

\node[blue,fill=blue] (t3) at (7,2.5) {};
\node[blue,fill=blue] (t7) at (15,2.5) {};
\node (t11) at (23,2.5) {};
\node (t15) at (31,2.5) {};

\draw[blue] (u2) -- (t3)--(u4);
\draw[blue] (u6) -- (t7)--(u8);
\draw (u10) -- (t11)--(u12);
\draw (u14) -- (t15)--(u16);

\node[blue,fill=blue] (s5) at (11,3.5) {};
\node (s13) at (27,3.5) {};

\draw[blue] (t3) -- (s5)--(t7);
\draw (t11) -- (s13)--(t15);

\node (r9) at (19,5) {};

\draw (s5) -- (r9)--(s13);

\tikzstyle{every node}=[]
\draw[below] (r9) node []           {\footnotesize{$\RR(G)$}};
\draw[blue,below] (v1) node []           {\footnotesize{$\SSS(\SP(G))$}};

\draw[blue,above left] (s5) node []           {\footnotesize{\footnotesize{$\RR(\SP(G))$}}};
\draw[above right] (s13) node []           {\footnotesize{\footnotesize{$\RR(\TP(G))$}}};
\draw[below] (v10) node []           {\footnotesize{$\blue{\TTT(\SP(G))}=\SSS(\TP(G))$}};

\end{tikzpicture}

%% file: LinkagesNested.tex
\scalebox{0.85}{%
\begin{tikzpicture}[scale=1/2,auto=left]

\tikzstyle{every node}=[inner sep=1.5pt, fill=black,circle,draw]
\node[blue] (v1) at (1,0) {};
\node[red] (v2) at (2,0) {};
\node[red] (v3) at (3,0) {};
\node[red] (v4) at (4,0) {};
\node[red] (v5) at (5,0) {};
\node[red] (v6) at (6,0) {};
\node[red] (v7) at (7,0) {};
\node[red] (v8) at (8,0) {};
\node[red] (v9) at (9,0) {};
\node[blue] (v10) at (10,0) {};

\draw[dotted, thick] (v1)--(v2);
\draw[red,thick] (v2)--(v9);
\draw[dotted,thick] (v9)--(v10);

\node[blue] (u2) at (2.5,1) {};
\node (u4) at (4.5,1) {};
\node (u6) at (6.5,1) {};
\node[blue] (u8) at (8.5,1) {};

\draw[blue,thick] (u2) to [bend right=20] (v1);
\draw[dotted,thick] (u4) to [bend right=20] (v3);
\draw[dotted,thick] (u6) to [bend right=20] (v5);
\draw[dotted,thick] (u8) to [bend right=20] (v7);

\draw[dotted,thick] (u2) to [bend left=20] (v4);
\draw[dotted,thick] (u4) to [bend left=20] (v6);
\draw[dotted,thick] (u6) to [bend left=20] (v8);
\draw[blue,thick] (u8) to [bend left=20] (v10);

\node[blue] (t3) at (3.5,2) {};
\node[blue] (t7) at (7.5,2) {};

\draw (u2) -- (t3)--(u4);
\draw (u6) -- (t7)--(u8);

\node[blue] (s5) at (5.5,3) {};

\draw (t3) -- (s5)--(t7);

\draw[blue,thick] (u2)--(t3)--(s5)--(t7)--(u8);

\end{tikzpicture}
}

%% file: LinkagesCrossing.tex
\scalebox{0.75}{%
\begin{tikzpicture}[scale=1/2,auto=left]

\tikzstyle{every node}=[inner sep=1.5pt, fill=black,circle,draw]
\node[blue] (v1) at (1,0) {};
\node[red] (v2) at (2,0) {};
\node[red] (v3) at (3,0) {};
\node[red] (v4) at (4,0) {};
\node[red] (v5) at (5,0) {};
\node[red] (v6) at (6,0) {};
\node[red] (v7) at (7,0) {};
\node[red] (v8) at (8,0) {};
\node[red] (v9) at (9,0) {};
\node[blue] (v10) at (10,0) {};
\node[blue] (v11) at (11,0) {};
\node[blue] (v12) at (12,0) {};
\node[blue] (v13) at (13,0) {};
\node[blue] (v14) at (14,0) {};
\node[blue] (v15) at (15,0) {};
\node[blue] (v16) at (16,0) {};
\node[blue] (v17) at (17,0) {};
\node[red] (v18) at (18,0) {};

\draw[dotted, thick] (v1)--(v2);
\draw[red,thick] (v2)--(v9);
\draw[dotted,thick] (v9)--(v10);
\draw[blue,thick] (v10)--(v17);
\draw[dotted,thick] (v17)--(v18);

\node[blue] (u2) at (2.5,1) {};
\node (u4) at (4.5,1) {};
\node (u6) at (6.5,1) {};
\node[blue] (u8) at (8.5,1) {};
\node[red] (u10) at (10.5,1) {};
\node (u12) at (12.5,1) {};
\node (u14) at (14.5,1) {};
\node[red] (u16) at (16.5,1) {};

\draw[blue,thick] (u2) to [bend right=20] (v1);
\draw[dotted,thick] (u4) to [bend right=20] (v3);
\draw[dotted,thick] (u6) to [bend right=20] (v5);
\draw[dotted,thick] (u8) to [bend right=20] (v7);
\draw[red,thick] (u10) to [bend right=20] (v9);
\draw[dotted,thick] (u12) to [bend right=20] (v11);
\draw[dotted,thick] (u14) to [bend right=20] (v13);
\draw[dotted,thick] (u16) to [bend right=20] (v15);

\draw[dotted,thick] (u2) to [bend left=20] (v4);
\draw[dotted,thick] (u4) to [bend left=20] (v6);
\draw[dotted,thick] (u6) to [bend left=20] (v8);
\draw[blue,thick] (u8) to [bend left=20] (v10);
\draw[dotted,thick] (u10) to [bend left=20] (v12);
\draw[dotted,thick] (u12) to [bend left=20] (v14);
\draw[dotted,thick] (u14) to [bend left=20] (v16);
\draw[red,thick] (u16) to [bend left=20] (v18);

\node[blue] (t3) at (3.5,2) {};
\node[blue] (t7) at (7.5,2) {};
\node[red] (t11) at (11.5,2) {};
\node[red] (t15) at (15.5,2) {};

\draw (u2) -- (t3)--(u4);
\draw (u6) -- (t7)--(u8);
\draw (u10) -- (t11)--(u12);
\draw (u14) -- (t15)--(u16);

\node[blue] (s5) at (5.5,3) {};
\node[red] (s13) at (13.5,3) {};

\draw (t3) -- (s5)--(t7);
\draw (t11) -- (s13)--(t15);

\node (r9) at (9.5,4) {};

\draw (s5) -- (r9)--(s13);

\draw[blue,thick] (u2)--(t3)--(s5)--(t7)--(u8);
\draw[red,thick] (u10)--(t11)--(s13)--(t15)--(u16);

\end{tikzpicture}
}

%% file: PathsInNSSGraph.tex
\scalebox{0.78}{%
\begin{tikzpicture}[scale=1/2,auto=left]

\tikzstyle{every node}=[gray,inner sep=1.5pt, fill=gray,circle,draw]
\tikzset{every path/.style={draw=gray}}
\node (v1) at (1,0) {};
\node (v34) at (34,0) {};
\node[black,thick] (v9) at (9,0) {};
\node[black,thick] (v14) at (14,0) {};
\draw[dotted, black, ultra thick] (v9)--(v14);
\draw[dotted, thick] (v1)--(v9);
\draw[dotted, thick] (v14)--(v34);
\node (v2) at (2,0) {};
\node (v3) at (3,0) {};
\node (v4) at (4,0) {};
\node (v5) at (5,0) {};
\node (v6) at (6,0) {};
\node (v7) at (7,0) {};
\node (v8) at (8,0) {};
\node[blue,thick] (v10) at (10,0) {};
\node[black,thick] (v11) at (11,0) {};
\node[black,thick] (v12) at (12,0) {};
\node[blue,thick] (v13) at (13,0) {};
\node (v15) at (15,0) {};
\node (v16) at (16,0) {};
\node (v17) at (17,0) {};
\node (v18) at (18,0) {};
\node (v19) at (19,0) {};
\node (v20) at (20,0) {};
\node (v21) at (21,0) {};
\node (v22) at (22,0) {};
\node (v23) at (23,0) {};
\node (v24) at (24,0) {};
\node (v25) at (25,0) {};
\node (v26) at (26,0) {};
\node (v27) at (27,0) {};
\node (v28) at (28,0) {};
\node (v29) at (29,0) {};
\node (v30) at (30,0) {};
\node (v31) at (31,0) {};
\node (v32) at (32,0) {};
\node (v33) at (33,0) {};

\node (u2) at (2.5,1) {};
\node (u4) at (4.5,1) {};
\node (u6) at (6.5,1) {};
\node[blue,thick] (u8) at (8.5,1) {};
\node[black,thick] (u10) at (10.5,1) {};
\node[black,thick] (u12) at (12.5,1) {};
\node[blue,thick] (u14) at (14.5,1) {};
\node (u16) at (16.5,1) {};
\node (u18) at (18.5,1) {};
\node (u20) at (20.5,1) {};
\node (u22) at (22.5,1) {};
\node (u24) at (24.5,1) {};
\node (u26) at (26.5,1) {};
\node (u28) at (28.5,1) {};
\node (u30) at (30.5,1) {};
\node (u32) at (32.5,1) {};

\draw[dotted,thick] (u2) to [bend right=20] (v1);
\draw[dotted,thick] (u4) to [bend right=20] (v3);
\draw[dotted,thick] (u6) to [bend right=20] (v5);
\draw[dotted,thick] (u8) to [bend right=20] (v7);
\draw[dotted,black,ultra thick] (u10) to [bend right=20] (v9);
\draw[dotted,black,ultra thick] (u12) to [bend right=20] (v11);
\draw[blue,ultra thick] (u14) to [bend right=20] (v13);
\draw[dotted,thick] (u16) to [bend right=20] (v15);
\draw[dotted,thick] (u18) to [bend right=20] (v17);
\draw[dotted,thick] (u20) to [bend right=20] (v19);
\draw[dotted,thick] (u22) to [bend right=20] (v21);
\draw[dotted,thick] (u24) to [bend right=20] (v23);
\draw[dotted,thick] (u26) to [bend right=20] (v25);
\draw[dotted,thick] (u28) to [bend right=20] (v27);
\draw[dotted,thick] (u30) to [bend right=20] (v29);
\draw[dotted,thick] (u32) to [bend right=20] (v31);

\draw[dotted,thick] (u2) to [bend left=20] (v4);
\draw[dotted,thick] (u4) to [bend left=20] (v6);
\draw[dotted,thick] (u6) to [bend left=20] (v8);
\draw[blue,ultra thick] (u8) to [bend left=20] (v10);
\draw[dotted,black,ultra thick] (u10) to [bend left=20] (v12);
\draw[dotted,black,ultra thick] (u12) to [bend left=20] (v14);
\draw[dotted,thick] (u14) to [bend left=20] (v16);
\draw[dotted,thick] (u16) to [bend left=20] (v18);
\draw[dotted,thick] (u18) to [bend left=20] (v20);
\draw[dotted,thick] (u20) to [bend left=20] (v22);
\draw[dotted,thick] (u22) to [bend left=20] (v24);
\draw[dotted,thick] (u24) to [bend left=20] (v26);
\draw[dotted,thick] (u26) to [bend left=20] (v28);
\draw[dotted,thick] (u28) to [bend left=20] (v30);
\draw[dotted,thick] (u30) to [bend left=20] (v32);
\draw[dotted,thick] (u32) to [bend left=20] (v34);

\node (t3) at (3.5,2) {};
\node[blue] (t7) at (7.5,2) {};
\node[black,thick] (t11) at (11.5,2) {};
\node[blue] (t15) at (15.5,2) {};
\node (t19) at (19.5,2) {};
\node (t23) at (23.5,2) {};
\node (t27) at (27.5,2) {};
\node (t31) at (31.5,2) {};

\draw (u2) -- (t3)--(u4);
\draw (u6) -- (t7)--(u8);
\draw[black,very thick] (u10) -- (t11)--(u12);
\draw (u14) -- (t15)--(u16);
\draw (u18) -- (t19)--(u20);
\draw (u22) -- (t23)--(u24);
\draw (u26) -- (t27)--(u28);
\draw (u30) -- (t31)--(u32);

\node[blue] (s5) at (5.5,3) {};
\node[blue] (s13) at (13.5,3) {};
\node (s21) at (21.5,3) {};
\node (s29) at (29.5,3) {};

\draw (t3) -- (s5)--(t7);
\draw (t11) -- (s13)--(t15);
\draw (t19) -- (s21)--(t23);
\draw (t27) -- (s29)--(t31);

\node[blue] (r9) at (9.5,4) {};
\node (r25) at (25.5,4) {};

\draw (s5) -- (r9)--(s13);
\draw (s21) -- (r25)--(s29);

\node (q17) at (17.5,5) {};
\draw (r9) -- (q17)--(r25);

\draw[blue,ultra thick] (u8)--(t7)--(s5)--(r9)--(s13)--(t15)--(u14);

\tikzstyle{every node}=[]
\draw[above left] (t11) node {\footnotesize $\RR(\Delta)$};
\draw[above right] (u12) node {\large $\Delta$};

\end{tikzpicture}
}

%% file: K7.tex
\scalebox{0.75}{%
\begin{tikzpicture}[scale=1/2,auto=left]

\tikzstyle{every node}=[inner sep=1.5pt, fill=gray,circle,draw]
\tikzset{every path/.style={draw=gray}}
\node[ForestGreen,thick] (v1) at (1,0) {};
\node[RedViolet,thick] (v34) at (34,0) {};
\draw[dotted, thick] (v1)--(v34);
\node[LimeGreen,thick] (v2) at (2,0) {};
\node[LimeGreen,thick] (v3) at (3,0) {};
\node[ForestGreen,thick] (v4) at (4,0) {};
\node[ForestGreen,thick] (v5) at (5,0) {};
\node[LimeGreen,thick] (v6) at (6,0) {};
\node[LimeGreen,thick] (v7) at (7,0) {};
\node[ForestGreen,thick] (v8) at (8,0) {};
\node[ForestGreen,thick] (v9) at (9,0) {};
\node[LimeGreen,thick] (v10) at (10,0) {};
\node[LimeGreen,thick] (v11) at (11,0) {};
\node[cyan,thick] (v12) at (12,0) {};
\node[cyan,thick] (v13) at (13,0) {};
\node[orange,thick] (v14) at (14,0) {};
\node[orange,thick] (v15) at (15,0) {};
\node[orange,thick] (v16) at (16,0) {};
\node[orange,thick] (v17) at (17,0) {};
\node[cyan,thick] (v18) at (18,0) {};
\node[cyan,thick] (v19) at (19,0) {};
\node[magenta,thick] (v20) at (20,0) {};
\node[magenta,thick] (v21) at (21,0) {};
\node[RedViolet,thick] (v22) at (22,0) {};
\node[RedViolet,thick] (v23) at (23,0) {};
\node[magenta,thick] (v24) at (24,0) {};
\node[magenta,thick] (v25) at (25,0) {};
\node[RedViolet,thick] (v26) at (26,0) {};
\node[RedViolet,thick] (v27) at (27,0) {};
\node[magenta,thick] (v28) at (28,0) {};
\node[magenta,thick] (v29) at (29,0) {};
\node[RedViolet,thick] (v30) at (30,0) {};
\node[RedViolet,thick] (v31) at (31,0) {};
\node[magenta,thick] (v32) at (32,0) {};
\node[magenta,thick] (v33) at (33,0) {};

\node[ForestGreen,thick] (u2) at (2.5,1) {};
\node[LimeGreen,thick] (u4) at (4.5,1) {};
\node[ForestGreen,thick] (u6) at (6.5,1) {};
\node[LimeGreen,thick] (u8) at (8.5,1) {};
\node[ForestGreen,thick] (u10) at (10.5,1) {};
\node[orange,thick] (u12) at (12.5,1) {};
\node[cyan,thick] (u14) at (14.5,1) {};
\node[cyan,thick] (u16) at (16.5,1) {};
\node[orange,thick] (u18) at (18.5,1) {};
\node[RedViolet,thick] (u20) at (20.5,1) {};
\node[magenta,thick] (u22) at (22.5,1) {};
\node[RedViolet,thick] (u24) at (24.5,1) {};
\node[magenta,thick] (u26) at (26.5,1) {};
\node[RedViolet,thick] (u28) at (28.5,1) {};
\node[magenta,thick] (u30) at (30.5,1) {};
\node[RedViolet,thick] (u32) at (32.5,1) {};

\draw[ForestGreen,ultra thick] (u2) to [bend right=20] (v1);
\draw[LimeGreen,ultra thick] (u4) to [bend right=20] (v3);
\draw[ForestGreen,ultra thick] (u6) to [bend right=20] (v5);
\draw[LimeGreen,ultra thick] (u8) to [bend right=20] (v7);
\draw[ForestGreen,ultra thick] (u10) to [bend right=20] (v9);
\draw[dotted,thick] (u12) to [bend right=20] (v11);
\draw[cyan,ultra thick] (u14) to [bend right=20] (v13);
\draw[dotted,thick] (u16) to [bend right=20] (v15);
\draw[orange,ultra thick] (u18) to [bend right=20] (v17);
\draw[dotted,thick] (u20) to [bend right=20] (v19);
\draw[magenta,ultra thick] (u22) to [bend right=20] (v21);
\draw[RedViolet,ultra thick] (u24) to [bend right=20] (v23);
\draw[magenta,ultra thick] (u26) to [bend right=20] (v25);
\draw[RedViolet,ultra thick] (u28) to [bend right=20] (v27);
\draw[magenta,ultra thick] (u30) to [bend right=20] (v29);
\draw[RedViolet,ultra thick] (u32) to [bend right=20] (v31);

\draw[ForestGreen,ultra thick] (u2) to [bend left=20] (v4);
\draw[LimeGreen,ultra thick] (u4) to [bend left=20] (v6);
\draw[ForestGreen,ultra thick] (u6) to [bend left=20] (v8);
\draw[LimeGreen,ultra thick] (u8) to [bend left=20] (v10);
\draw[dotted,thick] (u10) to [bend left=20] (v12);
\draw[orange,ultra thick] (u12) to [bend left=20] (v14);
\draw[dotted,thick] (u14) to [bend left=20] (v16);
\draw[cyan,ultra thick] (u16) to [bend left=20] (v18);
\draw[dotted,thick] (u18) to [bend left=20] (v20);
\draw[RedViolet,ultra thick] (u20) to [bend left=20] (v22);
\draw[magenta,ultra thick] (u22) to [bend left=20] (v24);
\draw[RedViolet,ultra thick] (u24) to [bend left=20] (v26);
\draw[magenta,ultra thick] (u26) to [bend left=20] (v28);
\draw[RedViolet,ultra thick] (u28) to [bend left=20] (v30);
\draw[magenta,ultra thick] (u30) to [bend left=20] (v32);
\draw[RedViolet,ultra thick] (u32) to [bend left=20] (v34);

\node[blue,thick] (t3) at (3.5,2) {};
\node (t7) at (7.5,2) {};
\node[orange,thick] (t11) at (11.5,2) {};
\node[cyan,thick] (t15) at (15.5,2) {};
\node[orange] (t19) at (19.5,2) {};
\node (t23) at (23.5,2) {};
\node (t27) at (27.5,2) {};
\node[blue,thick] (t31) at (31.5,2) {};

\draw (u2) -- (t3)--(u4);
\draw (u6) -- (t7)--(u8);
\draw (u10) -- (t11)--(u12);
\draw[cyan,ultra thick] (u14) -- (t15)--(u16);
\draw (u18) -- (t19)--(u20);
\draw (u22) -- (t23)--(u24);
\draw (u26) -- (t27)--(u28);
\draw (u30) -- (t31)--(u32);

\node[blue,thick] (s5) at (5.5,3) {};
\node[blue,thick] (s13) at (13.5,3) {};
\node (s21) at (21.5,3) {};
\node[blue,thick] (s29) at (29.5,3) {};

\draw (t3) -- (s5)--(t7);
\draw (t11) -- (s13)--(t15);
\draw (t19) -- (s21)--(t23);
\draw (t27) -- (s29)--(t31);

\node[blue,thick] (r9) at (9.5,4) {};
\node[blue,thick] (r25) at (25.5,4) {};

\draw[blue,ultra thick] (s5) -- (r9)--(s13);
\draw (s21) -- (r25)--(s29);

\node[blue,thick] (q17) at (17.5,5) {};
\draw[blue,ultra thick] (r9) -- (q17)--(r25);

\draw[blue, ultra thick] (t3)--(s5);
\draw[blue,ultra thick] (r25)--(s29)--(t31);
\draw[orange,ultra thick] (t11)--(u12);
\draw[orange,ultra thick] (t19)--(u18);

\draw (q17)--(13,5.5);
\draw[LimeGreen,ultra thick] (0.5,1) to [bend left=20] (v2);
\draw[magenta,ultra thick] (34.5,1) to [bend right=20] (v33);

\draw[ForestGreen,ultra thick] (0.5,0)--(v1);
\draw[LimeGreen,ultra thick] (v2)--(v3);
\draw[ForestGreen,ultra thick] (v4)--(v5);
\draw[LimeGreen,ultra thick] (v6)--(v7);
\draw[ForestGreen,ultra thick] (v8)--(v9);
\draw[LimeGreen,ultra thick] (v10)--(v11);
\draw[cyan,ultra thick] (v12)--(v13);
\draw[orange,ultra thick] (v14)--(v17);
\draw[cyan,ultra thick] (v18)--(v19);
\draw[magenta,ultra thick] (v20)--(v21);

\draw[magenta,ultra thick] (v20)--(v21);
\draw[RedViolet,ultra thick] (v22)--(v23);
\draw[magenta,ultra thick] (v24)--(v25);
\draw[RedViolet,ultra thick] (v26)--(v27);
\draw[magenta,ultra thick] (v28)--(v29);
\draw[RedViolet,ultra thick] (v30)--(v31);
\draw[magenta,ultra thick] (v32)--(v33);
\draw[RedViolet,ultra thick] (v34)--(34.5,0);

\tikzstyle{every node}=[]
\draw[LimeGreen] (6.5,-0.6) node [] {\large $1$};
\draw[ForestGreen] (8.5,-0.6) node [] {\large $2$};
\draw[cyan] (12.5,-0.6) node [] {\large $3$};
\draw[orange] (15.5,-0.6) node [] {\large $4$};
\draw[magenta] (24.5,-0.6) node [] {\large $5$};
\draw[RedViolet] (26.5,-0.6) node [] {\large $6$};
\draw[blue] (17.5,4.2) node [] {\large $7$};

\end{tikzpicture}
}

%% file: K7_2.tex
\scalebox{0.75}{%
\begin{tikzpicture}[scale=1/2,auto=left]

\tikzstyle{every node}=[gray,inner sep=1.5pt, fill=gray,circle,draw]
\tikzset{every path/.style={draw=gray}}

\path [fill=gray,opacity=0.3,draw=gray]
(2,0) to [bend left=15] (2.5,1) to [bend left=20]
 (3.5,2.5) to [bend left=10] (7.5,2.5) to [bend left=20] (8.5,1) to [bend left=20] (9,0) to [bend left=15] (2,0);

\path [fill=gray,opacity=0.8,draw=black]
(10,0) to [bend left=15] (10.5,1) to [bend left=20]
 (11.5,2.5) to [bend left=10] (15.5,2.5) to [bend left=20] (16.5,1) to [bend left=20] (17,0) to [bend left=15] (10,0);

 \path [fill=gray,opacity=0.3,draw=gray]
(18,0) to [bend left=15] (18.5,1) to [bend left=20]
 (19.5,2.5) to [bend left=10] (23.5,2.5) to [bend left=20] (24.5,1) to [bend left=20] (25,0) to [bend left=15] (18,0);

\path [fill=gray,opacity=0.3,draw=black]
(26,0) to [bend left=15] (26.5,1) to [bend left=20]
 (27.5,2.5) to [bend left=10] (31.5,2.5) to [bend left=20] (32.5,1) to [bend left=20] (33,0) to [bend left=15] (26,0);

\tikzstyle{every node}=[inner sep=1.5pt, fill=gray,circle,draw]
\node[ForestGreen,thick] (v1) at (1,0) {};
\node[RedViolet,thick] (v34) at (34,0) {};
\node[LimeGreen,thick] (v10) at (10,0) {};
\draw[dotted, thick] (v1)--(v10);
\node[LimeGreen,thick] (v2) at (2,0) {};
\node[LimeGreen,thick] (v5) at (5,0) {};
\node[ForestGreen,thick] (v6) at (6,0) {};
\node[ForestGreen] (v9) at (9,0) {};
\node[magenta,thick] (v17) at (17,0) {};
\draw[dotted,thick] (v17)--(v26);
\node[RedViolet,thick] (v18) at (18,0) {};
\node[RedViolet,thick] (v25) at (25,0) {};
\node[RedViolet,thick] (v26) at (26,0) {};
\node[RedViolet,thick] (v33) at (33,0) {};

\draw[ForestGreen,ultra thick] (0.5,0)--(v1);
\draw[LimeGreen,ultra thick] (v2)--(v5);
\draw[ForestGreen,ultra thick] (v6)--(v9);
\draw[RedViolet,ultra thick] (v18)--(34.5,0);

\node[ForestGreen,thick] (u2) at (2.5,1) {};
\node[LimeGreen,thick] (u8) at (8.5,1) {};
\node[ForestGreen,thick] (u10) at (10.5,1) {};
\node[RedViolet,thick] (u16) at (16.5,1) {};
\node[magenta,thick] (u18) at (18.5,1) {};
\node (u24) at (24.5,1) {};
\node (u26) at (26.5,1) {};
\node (u32) at (32.5,1) {};

\draw[ForestGreen,ultra thick] (u2) to [bend right=20] (v1);
\draw[ForestGreen,ultra thick] (u10) to [bend right=20] (v9);
\draw[magenta,ultra thick] (u18) to [bend right=20] (v17);
\draw[dotted,thick] (u26) to [bend right=20] (v25);

\draw[LimeGreen,ultra thick] (u8) to [bend left=20] (v10);
\draw[RedViolet,ultra thick] (u16) to [bend left=20] (v18);
\draw[dotted,thick] (u24) to [bend left=20] (v26);
\draw[dotted,thick] (u32) to [bend left=20] (v34);

\node[ForestGreen,thick] (t3) at (3.5,2.5) {};
\node[LimeGreen,thick] (t7) at (7.5,2.5) {};
\node[blue,thick] (t11) at (11.5,2.5) {};
\node[blue,thick] (t15) at (15.5,2.5) {};
\node[magenta,thick] (t19) at (19.5,2.5) {};
\node (t23) at (23.5,2.5) {};
\node (t27) at (27.5,2.5) {};
\node (t31) at (31.5,2.5) {};


\node[magenta,thick] (s5) at (5.5,3.5) {};
\node[blue,thick] (s13) at (13.5,3.5) {};
\node[magenta,thick] (s21) at (21.5,3.5) {};
\node (s29) at (29.5,3.5) {};

\draw (t3) -- (s5)--(t7);
\draw[blue,ultra thick] (t11) -- (s13)--(t15);
\draw (t19) -- (s21)--(t23);
\draw (t27) -- (s29)--(t31);

\node[magenta,thick] (r9) at (9.5,4.5) {};
\node[magenta,thick] (r25) at (25.5,4.5) {};

\draw (s5) -- (r9)--(s13);
\draw (s21) -- (r25)--(s29); 

\node[magenta,thick] (q17) at (17.5,5.5) {};
\draw[magenta,ultra thick] (r9) -- (q17)--(r25);

\draw[ForestGreen,ultra thick] (u2)--(t3)--(v6);
\draw[LimeGreen,ultra thick] (v5)--(t7)--(u8);
\draw[magenta,ultra thick] (u18)--(t19)--(s21)--(r25);
\draw[magenta,ultra thick] (s5)--(r9);

\draw[black,thick] (t3) to [bend left=10] (t7);
\draw[black,thick] (v2) to [bend right=15] (v9);
\draw[black,thick] (v2) to [bend left=15] (u2);
\draw[black,thick] (u2) to [bend left=20] (t3);
\draw[black,thick] (v9) to [bend right=15] (u8);
\draw[black,thick] (u8) to [bend right=20] (t7);

\draw[black,thick] (t11) to [bend left=10] (t15);
\draw[black,thick] (v10) to [bend right=15] (v17);
\draw[black,thick] (v10) to [bend left=15] (u10);
\draw[black,thick] (u10) to [bend left=20] (t11);
\draw[black,thick] (v17) to [bend right=15] (u16);
\draw[black,thick] (u16) to [bend right=20] (t15);

\draw[black,thick] (t19) to [bend left=10] (t23);
\draw[black,thick] (v18) to [bend right=15] (v25);
\draw[black,thick] (v18) to [bend left=15] (u18);
\draw[black,thick] (u18) to [bend left=20] (t19);
\draw[black,thick] (v25) to [bend right=15] (u24);
\draw[black,thick] (u24) to [bend right=20] (t23);

\draw[black,thick] (t27) to [bend left=10] (t31);
\draw[black,thick] (v26) to [bend right=15] (v33);
\draw[black,thick] (v26) to [bend left=15] (u26);
\draw[black,thick] (u26) to [bend left=20] (t27);
\draw[black,thick] (v33) to [bend right=15] (u32);
\draw[black,thick] (u32) to [bend right=20] (t31);

\draw (q17)--(13,6);
\draw[LimeGreen,ultra thick] (0.5,1) to [bend left=20] (v2);
\draw[dotted,thick] (34.5,1) to [bend right=20] (v33);

\tikzstyle{every node}=[]
\draw[LimeGreen] (3.5,0.5) node [] {\large $1$};
\draw[ForestGreen] (7.5,0.5) node [] {\large $2$};
\draw[blue] (14,4) node [] {\large $7$};
\draw[magenta] (17.5,4.8) node [] {\large $5$};
\draw[RedViolet] (21.5,0.5) node [] {\large $6$};

\end{tikzpicture}
}

%% file: K4t.tex
\scalebox{0.79}{%
\begin{tikzpicture}[scale=1/2,auto=left]

\tikzstyle{every node}=[gray,inner sep=1.5pt, fill=gray,circle,draw]
\tikzset{every path/.style={draw=gray}}
\node[blue,thick] (v1) at (1,0) {};
\node[ForestGreen,thick] (v34) at (34,0) {};
\draw[dotted, thick] (v1)--(v34);
\node[cyan,thick] (v2) at (2,0) {};
\node[cyan,thick] (v3) at (3,0) {};
\node[blue,thick] (v4) at (4,0) {};
\node[blue,thick] (v5) at (5,0) {};
\node[cyan,thick] (v6) at (6,0) {};
\node[cyan,thick] (v7) at (7,0) {};
\node[blue,thick] (v8) at (8,0) {};
\node[blue,thick] (v9) at (9,0) {};
\node[cyan,thick] (v10) at (10,0) {};
\node[cyan,thick] (v11) at (11,0) {};
\node[blue,thick] (v12) at (12,0) {};
\node[blue,thick] (v13) at (13,0) {};
\node[ForestGreen,thick] (v14) at (14,0) {};
\node[ForestGreen,thick] (v15) at (15,0) {};
\node[LimeGreen,thick] (v16) at (16,0) {};
\node[LimeGreen,thick] (v17) at (17,0) {};
\node[ForestGreen,thick] (v18) at (18,0) {};
\node[ForestGreen,thick] (v19) at (19,0) {};
\node[LimeGreen,thick] (v20) at (20,0) {};
\node[LimeGreen,thick] (v21) at (21,0) {};
\node[ForestGreen,thick] (v22) at (22,0) {};
\node[ForestGreen,thick] (v23) at (23,0) {};
\node[LimeGreen,thick] (v24) at (24,0) {};
\node[LimeGreen,thick] (v25) at (25,0) {};
\node[ForestGreen,thick] (v26) at (26,0) {};
\node[ForestGreen,thick] (v27) at (27,0) {};
\node[LimeGreen,thick] (v28) at (28,0) {};
\node[LimeGreen,thick] (v29) at (29,0) {};
\node[ForestGreen,thick] (v30) at (30,0) {};
\node[ForestGreen,thick] (v31) at (31,0) {};
\node[LimeGreen,thick] (v32) at (32,0) {};
\node[LimeGreen,thick] (v33) at (33,0) {};

\node[blue,thick] (u2) at (2.5,1) {};
\node[cyan,thick] (u4) at (4.5,1) {};
\node[blue,thick] (u6) at (6.5,1) {};
\node[cyan,thick] (u8) at (8.5,1) {};
\node[blue,thick] (u10) at (10.5,1) {};
\node[cyan,thick] (u12) at (12.5,1) {};
\node[LimeGreen,thick] (u14) at (14.5,1) {};
\node[ForestGreen,thick] (u16) at (16.5,1) {};
\node[LimeGreen,thick] (u18) at (18.5,1) {};
\node[ForestGreen,thick] (u20) at (20.5,1) {};
\node[LimeGreen,thick] (u22) at (22.5,1) {};
\node[ForestGreen,thick] (u24) at (24.5,1) {};
\node[LimeGreen,thick] (u26) at (26.5,1) {};
\node[ForestGreen,thick] (u28) at (28.5,1) {};
\node[LimeGreen,thick] (u30) at (30.5,1) {};
\node[ForestGreen,thick] (u32) at (32.5,1) {};

\draw[blue,ultra thick] (u2) to [bend right=20] (v1);
\draw[cyan,ultra thick] (u4) to [bend right=20] (v3);
\draw[blue,ultra thick] (u6) to [bend right=20] (v5);
\draw[cyan,ultra thick] (u8) to [bend right=20] (v7);
\draw[blue,ultra thick] (u10) to [bend right=20] (v9);
\draw[cyan,ultra thick] (u12) to [bend right=20] (v11);
\draw[dotted,thick] (u14) to [bend right=20] (v13);
\draw[ForestGreen,ultra thick] (u16) to [bend right=20] (v15);
\draw[LimeGreen,ultra thick] (u18) to [bend right=20] (v17);
\draw[ForestGreen,ultra thick] (u20) to [bend right=20] (v19);
\draw[LimeGreen,ultra thick] (u22) to [bend right=20] (v21);
\draw[ForestGreen,ultra thick] (u24) to [bend right=20] (v23);
\draw[LimeGreen,ultra thick] (u26) to [bend right=20] (v25);
\draw[ForestGreen,ultra thick] (u28) to [bend right=20] (v27);
\draw[LimeGreen,ultra thick] (u30) to [bend right=20] (v29);
\draw[ForestGreen,ultra thick] (u32) to [bend right=20] (v31);

\draw[blue,ultra thick] (u2) to [bend left=20] (v4);
\draw[cyan,ultra thick] (u4) to [bend left=20] (v6);
\draw[blue,ultra thick] (u6) to [bend left=20] (v8);
\draw[cyan,ultra thick] (u8) to [bend left=20] (v10);
\draw[blue,ultra thick] (u10) to [bend left=20] (v12);
\draw[dotted,thick] (u12) to [bend left=20] (v14);
\draw[LimeGreen,ultra thick] (u14) to [bend left=20] (v16);
\draw[ForestGreen,ultra thick] (u16) to [bend left=20] (v18);
\draw[LimeGreen,ultra thick] (u18) to [bend left=20] (v20);
\draw[ForestGreen,ultra thick] (u20) to [bend left=20] (v22);
\draw[LimeGreen,ultra thick] (u22) to [bend left=20] (v24);
\draw[ForestGreen,ultra thick] (u24) to [bend left=20] (v26);
\draw[LimeGreen,ultra thick] (u26) to [bend left=20] (v28);
\draw[ForestGreen,ultra thick] (u28) to [bend left=20] (v30);
\draw[LimeGreen,ultra thick] (u30) to [bend left=20] (v32);
\draw[ForestGreen,ultra thick] (u32) to [bend left=20] (v34);

\draw[cyan,ultra thick] (v2)--(v3);
\draw[blue,ultra thick] (v4)--(v5);
\draw[cyan,ultra thick] (v6)--(v7);
\draw[blue,ultra thick] (v8)--(v9);
\draw[cyan,ultra thick] (v10)--(v11);
\draw[blue,ultra thick] (v12)--(v13);
\draw[ForestGreen,ultra thick] (v14)--(v15);
\draw[LimeGreen,ultra thick] (v16)--(v17);
\draw[ForestGreen,ultra thick] (v18)--(v19);
\draw[LimeGreen,ultra thick] (v20)--(v21);
\draw[ForestGreen,ultra thick] (v22)--(v23);
\draw[LimeGreen,ultra thick] (v24)--(v25);
\draw[ForestGreen,ultra thick] (v26)--(v27);
\draw[LimeGreen,ultra thick] (v28)--(v29);
\draw[ForestGreen,ultra thick] (v30)--(v31);
\draw[LimeGreen,ultra thick] (v32)--(v33);

\node[magenta,thick] (t3) at (3.5,2) {};
\node[magenta,thick] (t7) at (7.5,2) {};
\node[orange,thick] (t11) at (11.5,2) {};
\node[orange,thick] (t15) at (15.5,2) {};
\node (t19) at (19.5,2) {};
\node (t23) at (23.5,2) {};
\node[magenta,thick] (t27) at (27.5,2) {};
\node[magenta,thick] (t31) at (31.5,2) {};

\draw (u2) -- (t3)--(u4);
\draw (u6) -- (t7)--(u8);
\draw (u10) -- (t11)--(u12);
\draw (u14) -- (t15)--(u16);
\draw (u18) -- (t19)--(u20);
\draw (u22) -- (t23)--(u24);
\draw (u26) -- (t27)--(u28);
\draw (u30) -- (t31)--(u32);

\node[magenta,thick] (s5) at (5.5,3) {};
\node[orange,thick] (s13) at (13.5,3) {};
\node (s21) at (21.5,3) {};
\node[magenta,thick] (s29) at (29.5,3) {};

\draw[magenta,ultra thick] (t3) -- (s5)--(t7);
\draw[orange,ultra thick] (t11) -- (s13)--(t15);
\draw (t19) -- (s21)--(t23);
\draw[magenta,ultra thick] (t27) -- (s29)--(t31);

\node[magenta,thick] (r9) at (9.5,4) {};
\node[magenta,thick] (r25) at (25.5,4) {};

\draw (s5) -- (r9)--(s13);
\draw (s21) -- (r25)--(s29);

\draw[magenta,ultra thick] (r9)--(s5);
\draw[magenta,ultra thick] (r25)--(s29);

\node[magenta,thick] (q17) at (17.5,5) {};
\draw[magenta,ultra thick] (r9) -- (q17)--(r25);

\tikzstyle{every node}=[]

\end{tikzpicture}
}

%% file: K4tFat.tex
\scalebox{0.79}{%
\begin{tikzpicture}[scale=1/2,auto=left]

\tikzstyle{every node}=[gray,inner sep=1.5pt, fill=gray,circle,draw]
\tikzset{every path/.style={draw=gray}}
\node[LimeGreen,thick] (v1) at (1,0) {};
\node[ForestGreen,thick] (v34) at (34,0) {};
\draw[dotted, thick] (v1)--(v34);
\node[ForestGreen,thick] (v2) at (2,0) {};
\node[ForestGreen,thick] (v3) at (3,0) {};
\node[ForestGreen,thick] (v4) at (4,0) {};
\node[ForestGreen,thick] (v5) at (5,0) {};
\node[LimeGreen,thick] (v6) at (6,0) {};
\node[LimeGreen,thick] (v7) at (7,0) {};
\node[ForestGreen,thick] (v8) at (8,0) {};
\node[ForestGreen,thick] (v9) at (9,0) {};
\node[LimeGreen,thick] (v10) at (10,0) {};
\node[LimeGreen,thick] (v11) at (11,0) {};
\node[ForestGreen,thick] (v12) at (12,0) {};
\node[ForestGreen,thick] (v13) at (13,0) {};
\node[LimeGreen,thick] (v14) at (14,0) {};
\node[LimeGreen,thick] (v15) at (15,0) {};
\node[ForestGreen,thick] (v16) at (16,0) {};
\node[ForestGreen,thick] (v17) at (17,0) {};
\node[LimeGreen,thick] (v18) at (18,0) {};
\node[LimeGreen,thick] (v19) at (19,0) {};
\node[LimeGreen,thick] (v20) at (20,0) {};
\node[LimeGreen,thick] (v21) at (21,0) {};
\node[ForestGreen,thick] (v22) at (22,0) {};
\node[ForestGreen,thick] (v23) at (23,0) {};
\node[LimeGreen,thick] (v24) at (24,0) {};
\node[LimeGreen,thick] (v25) at (25,0) {};
\node[ForestGreen,thick] (v26) at (26,0) {};
\node[ForestGreen,thick] (v27) at (27,0) {};
\node[LimeGreen,thick] (v28) at (28,0) {};
\node[LimeGreen,thick] (v29) at (29,0) {};
\node[ForestGreen,thick] (v30) at (30,0) {};
\node[ForestGreen,thick] (v31) at (31,0) {};
\node[LimeGreen,thick] (v32) at (32,0) {};
\node[LimeGreen,thick] (v33) at (33,0) {};

\node[LimeGreen,thick] (u2) at (2.5,1) {};
\node[LimeGreen,thick] (u4) at (4.5,1) {};
\node[ForestGreen,thick] (u6) at (6.5,1) {};
\node[LimeGreen,thick] (u8) at (8.5,1) {};
\node[ForestGreen,thick] (u10) at (10.5,1) {};
\node[LimeGreen,thick] (u12) at (12.5,1) {};

\node[ForestGreen,thick] (u14) at (14.5,1) {};
\node[LimeGreen,thick] (u16) at (16.5,1) {};
\node[ForestGreen,thick] (u18) at (18.5,1) {};
\node[ForestGreen,thick] (u20) at (20.5,1) {};
\node[LimeGreen,thick] (u22) at (22.5,1) {};
\node[ForestGreen,thick] (u24) at (24.5,1) {};
\node[LimeGreen,thick] (u26) at (26.5,1) {};
\node[ForestGreen,thick] (u28) at (28.5,1) {};
\node[LimeGreen,thick] (u30) at (30.5,1) {};
\node[ForestGreen,thick] (u32) at (32.5,1) {};

\draw[LimeGreen,ultra thick] (u2) to [bend right=20] (v1);
\draw[dotted,thick] (u4) to [bend right=20] (v3);
\draw[ForestGreen,ultra thick] (u6) to [bend right=20] (v5);
\draw[LimeGreen,ultra thick] (u8) to [bend right=20] (v7);
\draw[ForestGreen,ultra thick] (u10) to [bend right=20] (v9);
\draw[LimeGreen,ultra thick] (u12) to [bend right=20] (v11);
\draw[ForestGreen,ultra thick] (u14) to [bend right=20] (v13);
\draw[LimeGreen,ultra thick] (u16) to [bend right=20] (v15);
\draw[ForestGreen,ultra thick] (u18) to [bend right=20] (v17);
\draw[dotted,thick] (u20) to [bend right=20] (v19);
\draw[LimeGreen,ultra thick] (u22) to [bend right=20] (v21);
\draw[ForestGreen,ultra thick] (u24) to [bend right=20] (v23);
\draw[LimeGreen,ultra thick] (u26) to [bend right=20] (v25);
\draw[ForestGreen,ultra thick] (u28) to [bend right=20] (v27);
\draw[LimeGreen,ultra thick] (u30) to [bend right=20] (v29);
\draw[ForestGreen,ultra thick] (u32) to [bend right=20] (v31);

\draw[dotted,thick] (u2) to [bend left=20] (v4);
\draw[LimeGreen,ultra thick] (u4) to [bend left=20] (v6);
\draw[ForestGreen,ultra thick] (u6) to [bend left=20] (v8);
\draw[LimeGreen,ultra thick] (u8) to [bend left=20] (v10);
\draw[ForestGreen,ultra thick] (u10) to [bend left=20] (v12);
\draw[LimeGreen,ultra thick] (u12) to [bend left=20] (v14);
\draw[ForestGreen,ultra thick] (u14) to [bend left=20] (v16);
\draw[LimeGreen,ultra thick] (u16) to [bend left=20] (v18);
\draw[dotted,thick] (u18) to [bend left=20] (v20);
\draw[ForestGreen,ultra thick] (u20) to [bend left=20] (v22);
\draw[LimeGreen,ultra thick] (u22) to [bend left=20] (v24);
\draw[ForestGreen,ultra thick] (u24) to [bend left=20] (v26);
\draw[LimeGreen,ultra thick] (u26) to [bend left=20] (v28);
\draw[ForestGreen,ultra thick] (u28) to [bend left=20] (v30);
\draw[LimeGreen,ultra thick] (u30) to [bend left=20] (v32);
\draw[ForestGreen,ultra thick] (u32) to [bend left=20] (v34);

\draw[ForestGreen,ultra thick] (v2)--(v5);
\draw[LimeGreen,ultra thick] (v6)--(v7);
\draw[ForestGreen,ultra thick] (v8)--(v9);
\draw[LimeGreen,ultra thick] (v10)--(v11);
\draw[ForestGreen,ultra thick] (v12)--(v13);
\draw[LimeGreen,ultra thick] (v14)--(v15);
\draw[ForestGreen,ultra thick] (v16)--(v17);
\draw[LimeGreen,ultra thick] (v18)--(v21);
\draw[ForestGreen,ultra thick] (v22)--(v23);
\draw[LimeGreen,ultra thick] (v24)--(v25);
\draw[ForestGreen,ultra thick] (v26)--(v27);
\draw[LimeGreen,ultra thick] (v28)--(v29);
\draw[ForestGreen,ultra thick] (v30)--(v31);
\draw[LimeGreen,ultra thick] (v32)--(v33);

\node[LimeGreen,thick] (t3) at (3.5,2) {};
\node (t7) at (7.5,2) {};
\node (t11) at (11.5,2) {};
\node (t15) at (15.5,2) {};
\node[ForestGreen,thick] (t19) at (19.5,2) {};
\node (t23) at (23.5,2) {};
\node (t27) at (27.5,2) {};
\node (t31) at (31.5,2) {};

\draw[LimeGreen,ultra thick] (u2) -- (t3)--(u4);
\draw (u6) -- (t7)--(u8);
\draw (u10) -- (t11)--(u12);
\draw (u14) -- (t15)--(u16);
\draw[ForestGreen,ultra thick] (u18) -- (t19)--(u20);
\draw (u22) -- (t23)--(u24);
\draw (u26) -- (t27)--(u28);
\draw (u30) -- (t31)--(u32);

\node[black,thick] (s5) at (5.5,3) {};
\node (s13) at (13.5,3) {};
\node[black,thick] (s21) at (21.5,3) {};
\node (s29) at (29.5,3) {};

\draw (t3) -- (s5)--(t7);
\draw (t11) -- (s13)--(t15);
\draw (t19) -- (s21)--(t23);
\draw (t27) -- (s29)--(t31);

\node[magenta,thick] (r9) at (9.5,4) {};
\node[magenta,thick] (r25) at (25.5,4) {};

\draw (s5) -- (r9)--(s13);
\draw (s21) -- (r25)--(s29);

\draw[black,ultra thick] (t3) -- (s5) -- (r9);
\draw[black,ultra thick] (t19) -- (s21) -- (r25);

\node[magenta,thick] (q17) at (17.5,5) {};
\draw[magenta,ultra thick] (r9) -- (q17)--(r25);

\end{tikzpicture}
}

%% file: CSSProof1.tex
\scalebox{0.78}{%
\begin{tikzpicture}[scale=1/2,auto=left]

\tikzstyle{every node}=[gray,inner sep=1.5pt, fill=gray,circle,draw]
\node[black, very thick] (v1) at (1,0) {};
\node[black,very thick] (v9) at (9,0) {};
\node[black,very thick] (v29) at (29,0) {};
\node[black,very thick] (v34) at (34,0) {};
\node[black,very thick] (v2) at (2,0) {};
\node[black,very thick] (v3) at (3,0) {};
\node[black,very thick] (v4) at (4,0) {};
\node[black,very thick] (v5) at (5,0) {};
\node[black,very thick] (v6) at (6,0) {};
\node[black,very thick] (v7) at (7,0) {};
\node[black,very thick] (v8) at (8,0) {};
\node[black,very thick] (v10) at (10,0) {};
\node (v11) at (11,0) {};
\draw[black,dotted,ultra thick] (v1)--(v10);
\node[black,very thick] (v25) at (25,0) {};
\draw[dotted,thick] (v10)--(v25);
\node (v12) at (12,0) {};
\node (v13) at (13,0) {};
\node (v14) at (14,0) {};
\node (v15) at (15,0) {};
\node (v16) at (16,0) {};
\node (v17) at (17,0) {};
\node (v18) at (18,0) {};
\node (v19) at (19,0) {};
\node (v20) at (20,0) {};
\node (v21) at (21,0) {};
\node (v22) at (22,0) {};
\node (v23) at (23,0) {};
\node (v24) at (24,0) {};
\node[black,very thick] (v26) at (26,0) {};
\node[black,very thick] (v27) at (27,0) {};
\node[black,very thick] (v28) at (28,0) {};
\node[black,very thick] (v30) at (30,0) {};
\node[black,very thick] (v31) at (31,0) {};
\node[black,very thick] (v32) at (32,0) {};
\node[black,very thick] (v33) at (33,0) {};
\draw[black,dotted,ultra thick] (v25)--(v34);

\node[black,very thick] (u2) at (2.5,1.5) {};
\node[black,very thick] (u4) at (4.5,1.5) {};
\node[black,very thick] (u6) at (6.5,1.5) {};
\node[black,very thick] (u8) at (8.5,1.5) {};
\node (u10) at (10.5,1.5) {};
\node (u12) at (12.5,1.5) {};
\node (u14) at (14.5,1.5) {};
\node (u16) at (16.5,1.5) {};
\node (u18) at (18.5,1.5) {};
\node (u20) at (20.5,1.5) {};
\node (u22) at (22.5,1.5) {};
\node (u24) at (24.5,1.5) {};
\node[black,very thick] (u26) at (26.5,1.5) {};
\node[black,very thick] (u28) at (28.5,1.5) {};
\node[black,very thick] (u30) at (30.5,1.5) {};
\node[black,very thick] (u32) at (32.5,1.5) {};

\draw[black,dotted,ultra thick] (u2) to [bend right=20] (v1);
\draw[black,dotted,ultra thick] (u4) to [bend right=20] (v3);
\draw[black,dotted,ultra thick] (u6) to [bend right=20] (v5);
\draw[black,dotted,ultra thick] (u8) to [bend right=20] (v7);
\draw[dotted,thick] (u10) to [bend right=20] (v9);
\draw[dotted,thick] (u12) to [bend right=20] (v11);
\draw[dotted,thick] (u14) to [bend right=20] (v13);
\draw[dotted,thick] (u16) to [bend right=20] (v15);
\draw[dotted,thick] (u18) to [bend right=20] (v17);
\draw[dotted,thick] (u20) to [bend right=20] (v19);
\draw[dotted,thick] (u22) to [bend right=20] (v21);
\draw[dotted,thick] (u24) to [bend right=20] (v23);
\draw[black,dotted,ultra thick] (u26) to [bend right=20] (v25);
\draw[black,dotted,ultra thick] (u28) to [bend right=20] (v27);
\draw[black,dotted,ultra thick] (u30) to [bend right=20] (v29);
\draw[black,dotted,ultra thick] (u32) to [bend right=20] (v31);

\draw[black,dotted,ultra thick] (u2) to [bend left=20] (v4);
\draw[black,dotted,ultra thick] (u4) to [bend left=20] (v6);
\draw[black,dotted,ultra thick] (u6) to [bend left=20] (v8);
\draw[black,dotted,ultra thick] (u8) to [bend left=20] (v10);
\draw[dotted,thick] (u10) to [bend left=20] (v12);
\draw[dotted,thick] (u12) to [bend left=20] (v14);
\draw[dotted,thick] (u14) to [bend left=20] (v16);
\draw[dotted,thick] (u16) to [bend left=20] (v18);
\draw[dotted,thick] (u18) to [bend left=20] (v20);
\draw[dotted,thick] (u20) to [bend left=20] (v22);
\draw[dotted,thick] (u22) to [bend left=20] (v24);
\draw[dotted,thick] (u24) to [bend left=20] (v26);
\draw[black,dotted,ultra thick] (u26) to [bend left=20] (v28);
\draw[black,dotted,ultra thick] (u28) to [bend left=20] (v30);
\draw[black,dotted,ultra thick] (u30) to [bend left=20] (v32);
\draw[black,dotted,ultra thick] (u32) to [bend left=20] (v34);

\node[black,very thick] (t3) at (3.5,3) {};
\node[black,very thick] (t7) at (7.5,3) {};
\node (t11) at (11.5,3) {};
\node (t15) at (15.5,3) {};
\node (t19) at (19.5,3) {};
\node (t23) at (23.5,3) {};
\node[black,very thick] (t27) at (27.5,3) {};
\node[black,very thick] (t31) at (31.5,3) {};

\draw[black,very thick] (u2) -- (t3)--(u4);
\draw[black,very thick] (u6) -- (t7)--(u8);
\draw (u10) -- (t11)--(u12);
\draw (u14) -- (t15)--(u16);
\draw (u18) -- (t19)--(u20);
\draw (u22) -- (t23)--(u24);
\draw[black,very thick] (u26) -- (t27)--(u28);
\draw[black,very thick] (u30) -- (t31)--(u32);

\node[black,very thick] (s5) at (5.5,4.5) {};
\node (s13) at (13.5,4.5) {};
\node (s21) at (21.5,4.5) {};
\node[black,very thick] (s29) at (29.5,4.5) {};

\draw[black,very thick] (t3) -- (s5)--(t7);
\draw (t11) -- (s13)--(t15);
\draw (t19) -- (s21)--(t23);
\draw[black,very thick] (t27) -- (s29)--(t31);

\node (r9) at (9.5,6) {};
\node (r25) at (25.5,6) {};

\draw (s5) -- (r9)--(s13);
\draw (s21) -- (r25)--(s29);

\node (q17) at (17.5,7.5) {};
\draw (r9) -- (q17)--(r25);

\tikzstyle{every node}=[]

\draw[above left] (t3) node []           {\Large $\Delta_0$};
\draw[above right] (t31) node []           {\Large $\Delta_1$};
\draw (17.5,5.5) node [] {\LARGE $\Delta$};

\end{tikzpicture}
}

%% file: CSSProof2.tex
\scalebox{0.78}{%
\begin{tikzpicture}[scale=1/2,auto=left]

\draw[gray,very thick] (9.5,6) to [bend right=20] (1,0);
\draw [gray,very thick] plot [smooth, tension=0.4] coordinates {(9.5,6) (17.5,7.5) (25.5,6)};
\draw[gray,very thick] (25.5,6) to [bend left=20] (34,0);
\draw[gray,very thick] (1,0)--(34,0);

\draw[black,very thick] (29.2,4.5) to [bend right=8] (24.7,0.2);
\draw[black,very thick] (29.2,4.5) to [bend left=8] (33.7,0.2);
\draw[black,very thick] (24.7,0.2)--(33.7,0.2);

\draw[black,very thick] (s5)+(0.3,0) to [bend right=8] (1.3,0.2);
\draw[black,very thick] (s5)+(0.3,0) to [bend left=8] (10.3,0.2);
\draw[black,very thick] (1.3,0.2)--(10.3,0.2);

\draw[blue,very thick] (2.4,0.7) circle (0.4);
\draw[blue,very thick] (3.6,0.7) circle (0.4);

\draw[blue,very thick] (9.2,0.7) circle (0.4);
\draw[blue,very thick] (8,0.7) circle (0.4);

\draw[blue,very thick] (5.8,3.85) circle (0.4);

\draw[blue,very thick] (25.8,0.7) circle (0.4);
\draw[blue,very thick] (27,0.7) circle (0.4);

\draw[blue,very thick] (32.6,0.7) circle (0.4);
\draw[blue,very thick] (31.4,0.7) circle (0.4);

\draw[blue,very thick] (29.2,3.85) circle (0.4);

\draw[black, thick] (17.5,7) to [bend right=10] (5.8,4);
\draw[black, thick] (17.5,7) to [bend left=10] (29.2,4);

\draw [red,thick] plot [smooth, tension=0.4] coordinates {(8,0.9) (9,1.2) (13.8,2.3) (19,0.5) (25.6,0.7)};

\draw [Green,thick] plot [smooth, tension=0.4] coordinates {(27,0.9) (26,1.2) (22,2.3) (16,0.5) (9.4,0.7)};

\draw[dashed,thick] (1,3.7) to [bend right=6] (34,3.8);

\draw[very thick] (17.5,6.5) circle (0.8);

\tikzstyle{every node}=[]
\draw (6,1.8) node {\Large $\nabla_0$};
\draw (29.4,1.8) node {\Large $\nabla_1$};
\draw (23.1,5.5) node {\large $P$};
\draw[red] (10.8,2.3) node {\large $P_0$};
\draw[Green] (24.5,2.3) node {\large $P_1$};

\draw[blue] (10,-0.45) node {\footnotesize $\TO(\nabla_0)$};
\draw[blue] (7.4,-0.45) node {\footnotesize $\TI(\nabla_0)$};

\draw[blue] (25.2,-0.45) node {\footnotesize $\SO(\nabla_1)$};
\draw[blue] (27.8,-0.45) node {\footnotesize $\SI(\nabla_1)$};

\draw[blue] (4.6,4.6) node {\footnotesize $\RR(\nabla_0)$};
\draw[blue] (30.4,4.6) node {\footnotesize $\RR(\nabla_1)$};

\draw (33.5,4.3) node {\large $B$};

\draw (17.5,4.9) node {\large $\RR(\nabla)$};

\end{tikzpicture}
}

%% file: CSSProof3.tex
\scalebox{0.78}{%
\begin{tikzpicture}[scale=1/2,auto=left]

\draw[gray,very thick] (9.5,6) to [bend right=20] (1,0);
\draw [gray,very thick] plot [smooth, tension=0.4] coordinates {(9.5,6) (17.5,7.5) (25.5,6)};
\draw[gray,very thick] (25.5,6) to [bend left=20] (34,0);
\draw[gray,very thick] (1,0)--(34,0);

\draw[black,very thick] (29.2,4.5) to [bend right=8] (24.7,0.2);
\draw[black,very thick] (29.2,4.5) to [bend left=8] (33.7,0.2);
\draw[black,very thick] (24.7,0.2)--(33.7,0.2);

\draw[black,very thick] (s5)+(0.3,0) to [bend right=8] (1.3,0.2);
\draw[black,very thick] (s5)+(0.3,0) to [bend left=8] (10.3,0.2);
\draw[black,very thick] (1.3,0.2)--(10.3,0.2);

\draw[black,very thick] (24,3.5) to [bend left=8] (28.3,0.2);
\draw[black,very thick] (24,3.5) to [bend right=8] (19.7,0.2);
\draw[black,very thick] (19.7,0.2)--(28.3,0.2);

\draw[blue,very thick] (2.4,0.7) circle (0.4);
\draw[blue,very thick] (3.6,0.7) circle (0.4);

\draw[blue,very thick] (9.2,0.7) circle (0.4);
\draw[blue,very thick] (8,0.7) circle (0.4);

\draw[blue,very thick] (5.8,3.85) circle (0.4);

\draw[blue,very thick] (25.8,0.7) circle (0.4);
\draw[blue,very thick] (27,0.7) circle (0.4);

\draw[blue,very thick] (32.6,0.7) circle (0.4);
\draw[blue,very thick] (31.4,0.7) circle (0.4);

\draw[blue,very thick] (29.2,3.85) circle (0.4);

\draw[blue,very thick] (21,0.7) circle (0.4);
\draw[blue,very thick] (22.2,0.7) circle (0.4);

\draw[black, thick] (17.5,7) to [bend right=10] (5.8,4);
\draw[black, thick] (17.5,7) to [bend left=10] (29.2,4);

\draw [red,thick] plot [smooth, tension=0.4] coordinates {(8,0.9) (9,1.2) (16,2.2) (20.9,1.3) (22.4,0.7) (23.8,0.6) (25.5,1.7) (27.1,0.9)};

\draw [Green,thick] plot [smooth, tension=0.4] coordinates {(9.4,0.7) (20.2,0.7) (22.2,1.7) (24.2,0.6) (25.7,0.7)};

\draw[dashed,thick] (1,3.55) to [bend right=6] (34,3.4);

\tikzstyle{every node}=[]
\draw (6,1.8) node {\Large $\nabla_0$};
\draw (29.4,1.8) node {\Large $\nabla_1$};
\draw (17.5,6.4) node {\large $P$};
\draw[red] (10.8,2.1) node {\large $P_0$};
\draw[Green] (14,1.2) node {\large $P_1$};
\draw (24,1.9) node {\Large $\nabla'$};
\draw (33.5,4) node {\large $B$};

\end{tikzpicture}
}

%% file: CSSProof4.tex
\scalebox{0.78}{%
\begin{tikzpicture}[scale=1/2,auto=left]

\draw[gray,very thick] (9.5,6) to [bend right=20] (1,0);
\draw [gray,very thick] plot [smooth, tension=0.4] coordinates {(9.5,6) (17.5,7.5) (25.5,6)};
\draw[gray,very thick] (25.5,6) to [bend left=20] (34,0);
\draw[gray,very thick] (1,0)--(34,0);

\draw[black,very thick] (29.2,4.5) to [bend right=8] (24.7,0.2);
\draw[black,very thick] (29.2,4.5) to [bend left=8] (33.7,0.2);
\draw[black,very thick] (24.7,0.2)--(33.7,0.2);

\draw[black,very thick] (s5)+(0.3,0) to [bend right=8] (1.3,0.2);
\draw[black,very thick] (s5)+(0.3,0) to [bend left=8] (10.3,0.2);
\draw[black,very thick] (1.3,0.2)--(10.3,0.2);

\draw[black,very thick] (24,3.5) to [bend left=8] (28.3,0.2);
\draw[black,very thick] (24,3.5) to [bend right=8] (19.7,0.2);
\draw[black,very thick] (19.7,0.2)--(28.3,0.2);

\draw[blue,very thick] (2.4,0.7) circle (0.4);
\draw[blue,very thick] (3.6,0.7) circle (0.4);

\draw[blue,very thick] (9.2,0.7) circle (0.4);
\draw[blue,very thick] (8,0.7) circle (0.4);

\draw[blue,very thick] (5.8,3.85) circle (0.4);

\draw[blue,very thick] (25.8,0.7) circle (0.4);
\draw[blue,very thick] (27,0.7) circle (0.4);

\draw[blue,very thick] (32.6,0.7) circle (0.4);
\draw[blue,very thick] (31.4,0.7) circle (0.4);

\draw[blue,very thick] (29.2,3.85) circle (0.4);

\draw[blue,very thick] (21,0.7) circle (0.4);
\draw[blue,very thick] (22.2,0.7) circle (0.4);

\draw[black, thick] (17.5,7) to [bend right=10] (5.8,4);
\draw[black, thick] (17.5,7) to [bend left=10] (29.2,4);

\draw [red,thick] plot [smooth, tension=0.4] coordinates {(8,0.9) (9,1.2) (16,2.2) (20.9,1.3) (22.4,0.7) (25.7,0.7)};

\draw [Green,thick] plot [smooth, tension=0.4] coordinates {(9.4,0.7) (20.2,0.7) (24,2.85) (27,0.9)};

\draw[dashed,thick] (1,3.55) to [bend right=6] (34,3.4);

\tikzstyle{every node}=[]
\draw (6,1.8) node {\Large $\nabla_0$};
\draw (29.4,1.8) node {\Large $\nabla_1$};
\draw (17.5,6.4) node {\large $P$};
\draw[red] (10.8,2.1) node {\large $P'_0$};
\draw[Green] (14,1.2) node {\large $P'_1$};
\draw (24,1.8) node {\Large $\nabla'$};
\draw (33.5,4) node {\large $B$};

\end{tikzpicture}
}

%% file: TreeDecomp.tex
\scalebox{0.78}{%
\begin{tikzpicture}[scale=1/2,auto=left]

\tikzstyle{every node}=[gray,inner sep=1.5pt, fill=gray,circle,draw]
\tikzset{every path/.style={draw=gray}}
\node[blue,very thick] (v1) at (1,0) {};
\node[blue,very thick] (v18) at (18,0) {};
\draw[dotted,black,very thick] (v1)--(v18);
\node (v34) at (34,0) {};
\draw[dotted,thick] (v18)--(v34);
\node[blue,very thick] (v2) at (2,0) {};
\node[black] (v3) at (3,0) {};
\node[black] (v4) at (4,0) {};
\node[black] (v5) at (5,0) {};
\node[black] (v6) at (6,0) {};
\node[black] (v7) at (7,0) {};
\node[black] (v8) at (8,0) {};
\node[red,very thick] (v9) at (9,0) {};
\node[red,very thick] (v10) at (10,0) {};
\node[black] (v11) at (11,0) {};
\node[black] (v12) at (12,0) {};
\node[black] (v13) at (13,0) {};
\node[black] (v14) at (14,0) {};
\node[black] (v15) at (15,0) {};
\node[black] (v16) at (16,0) {};
\node[blue,very thick] (v17) at (17,0) {};
\node (v19) at (19,0) {};
\node (v20) at (20,0) {};
\node (v21) at (21,0) {};
\node (v22) at (22,0) {};
\node (v23) at (23,0) {};
\node (v24) at (24,0) {};
\node (v25) at (25,0) {};
\node (v26) at (26,0) {};
\node (v27) at (27,0) {};
\node (v28) at (28,0) {};
\node (v29) at (29,0) {};
\node (v30) at (30,0) {};
\node (v31) at (31,0) {};
\node (v32) at (32,0) {};
\node (v33) at (33,0) {};

\node[black] (u2) at (2.5,1) {};
\node[black] (u4) at (4.5,1) {};
\node[black] (u6) at (6.5,1) {};
\node[black] (u8) at (8.5,1) {};
\node[black] (u10) at (10.5,1) {};
\node[black] (u12) at (12.5,1) {};
\node[black] (u14) at (14.5,1) {};
\node[black] (u16) at (16.5,1) {};
\node (u18) at (18.5,1) {};
\node (u20) at (20.5,1) {};
\node (u22) at (22.5,1) {};
\node (u24) at (24.5,1) {};
\node (u26) at (26.5,1) {};
\node (u28) at (28.5,1) {};
\node (u30) at (30.5,1) {};
\node (u32) at (32.5,1) {};

\draw[dotted,black,very thick] (u2) to [bend right=20] (v1);
\draw[dotted,black,very thick] (u4) to [bend right=20] (v3);
\draw[dotted,black,very thick] (u6) to [bend right=20] (v5);
\draw[dotted,black,very thick] (u8) to [bend right=20] (v7);
\draw[dotted,black,very thick] (u10) to [bend right=20] (v9);
\draw[dotted,black,very thick] (u12) to [bend right=20] (v11);
\draw[dotted,black,very thick] (u14) to [bend right=20] (v13);
\draw[dotted,black,very thick] (u16) to [bend right=20] (v15);
\draw[dotted,thick] (u18) to [bend right=20] (v17);
\draw[dotted,thick] (u20) to [bend right=20] (v19);
\draw[dotted,thick] (u22) to [bend right=20] (v21);
\draw[dotted,thick] (u24) to [bend right=20] (v23);
\draw[dotted,thick] (u26) to [bend right=20] (v25);
\draw[dotted,thick] (u28) to [bend right=20] (v27);
\draw[dotted,thick] (u30) to [bend right=20] (v29);
\draw[dotted,thick] (u32) to [bend right=20] (v31);

\draw[dotted,black,very thick] (u2) to [bend left=20] (v4);
\draw[dotted,black,very thick] (u4) to [bend left=20] (v6);
\draw[dotted,black,very thick] (u6) to [bend left=20] (v8);
\draw[dotted,black,very thick] (u8) to [bend left=20] (v10);
\draw[dotted,black,very thick] (u10) to [bend left=20] (v12);
\draw[dotted,black,very thick] (u12) to [bend left=20] (v14);
\draw[dotted,black,very thick] (u14) to [bend left=20] (v16);
\draw[dotted,black,very thick] (u16) to [bend left=20] (v18);
\draw[dotted,thick] (u18) to [bend left=20] (v20);
\draw[dotted,thick] (u20) to [bend left=20] (v22);
\draw[dotted,thick] (u22) to [bend left=20] (v24);
\draw[dotted,thick] (u24) to [bend left=20] (v26);
\draw[dotted,thick] (u26) to [bend left=20] (v28);
\draw[dotted,thick] (u28) to [bend left=20] (v30);
\draw[dotted,thick] (u30) to [bend left=20] (v32);
\draw[dotted,thick] (u32) to [bend left=20] (v34);

\node[black] (t3) at (3.5,2) {};
\node[black] (t7) at (7.5,2) {};
\node[black] (t11) at (11.5,2) {};
\node[black] (t15) at (15.5,2) {};
\node (t19) at (19.5,2) {};
\node (t23) at (23.5,2) {};
\node (t27) at (27.5,2) {};
\node (t31) at (31.5,2) {};

\draw[black,thick] (u2) -- (t3)--(u4);
\draw[black,thick] (u6) -- (t7)--(u8);
\draw[black,thick] (u10) -- (t11)--(u12);
\draw[black,thick] (u14) -- (t15)--(u16);
\draw (u18) -- (t19)--(u20);
\draw (u22) -- (t23)--(u24);
\draw (u26) -- (t27)--(u28);
\draw (u30) -- (t31)--(u32);

\node[red,very thick] (s5) at (5.5,3) {};
\node[red,very thick] (s13) at (13.5,3) {};
\node (s21) at (21.5,3) {};
\node (s29) at (29.5,3) {};

\draw[thick, black] (t3) -- (s5)--(t7);
\draw[thick, black] (t11) -- (s13)--(t15);
\draw (t19) -- (s21)--(t23);
\draw (t27) -- (s29)--(t31);

\node[blue,very thick] (r9) at (9.5,4) {};
\node (r25) at (25.5,4) {};

\draw[black,thick] (s5) -- (r9)--(s13);
\draw (s21) -- (r25)--(s29);

\node (q17) at (17.5,5) {};
\draw (r9) -- (q17)--(r25);

\tikzstyle{every node}=[]
\draw[gray,above left] (q17) node []           {$r$};
\draw[blue,above left] (r9) node []           {$x$};
\draw (9.5,2.5) node [] {\large $\Delta_x$};

\end{tikzpicture}
}

%% file: Setting.tex
\scalebox{0.78}{%
\begin{tikzpicture}[scale=1/2,auto=left]

\tikzset{every path/.style={draw=gray}}
\tikzstyle{every node}=[gray,inner sep=1.5pt, fill=gray,circle,draw]
\node (v1) at (1,0) {};
\node (v9) at (9,0) {};
\draw[dotted,thick] (v1)--(v9);
\node (v34) at (34,0) {};
\node (v2) at (2,0) {};
\node (v3) at (3,0) {};
\node (v4) at (4,0) {};
\node (v5) at (5,0) {};
\node (v6) at (6,0) {};
\node (v7) at (7,0) {};
\node (v8) at (8,0) {};
\node (v10) at (10,0) {};
\node[black,thick] (v11) at (11,0) {};
\node[black,thick] (v12) at (12,0) {};
\node[black,thick] (v13) at (13,0) {};
\node[black,thick] (v14) at (14,0) {};
\node[black,thick] (v15) at (15,0) {};
\node[black,thick] (v16) at (16,0) {};
\node[black] (v17) at (17,0) {};
\node[black] (v18) at (18,0) {};
\draw[dotted,thick] (v18)--(v34);
\node (v19) at (19,0) {};
\node (v20) at (20,0) {};
\node (v21) at (21,0) {};
\node (v22) at (22,0) {};
\node (v23) at (23,0) {};
\node (v24) at (24,0) {};
\node (v25) at (25,0) {};
\node (v26) at (26,0) {};
\node (v27) at (27,0) {};
\node (v28) at (28,0) {};
\node (v29) at (29,0) {};
\node (v30) at (30,0) {};
\node (v31) at (31,0) {};
\node (v32) at (32,0) {};
\node (v33) at (33,0) {};

\draw[black,dotted, ultra thick] (v9)--(v18);

\node (u2) at (2.5,1.5) {};
\node (u4) at (4.5,1.5) {};
\node (u6) at (6.5,1.5) {};
\node (u8) at (8.5,1.5) {};
\node[black,thick] (u10) at (10.5,1.5) {};
\node[black,thick] (u12) at (12.5,1.5) {};
\node[black,thick] (u14) at (14.5,1.5) {};
\node[black,thick] (u16) at (16.5,1.5) {};
\node (u18) at (18.5,1.5) {};
\node (u20) at (20.5,1.5) {};
\node (u22) at (22.5,1.5) {};
\node (u24) at (24.5,1.5) {};
\node (u26) at (26.5,1.5) {};
\node (u28) at (28.5,1.5) {};
\node (u30) at (30.5,1.5) {};
\node (u32) at (32.5,1.5) {};

\draw[dotted,thick] (u2) to [bend right=20] (v1);
\draw[dotted,thick] (u4) to [bend right=20] (v3);
\draw[dotted,thick] (u6) to [bend right=20] (v5);
\draw[dotted,thick] (u8) to [bend right=20] (v7);
\draw[black,dotted,ultra thick] (u10) to [bend right=20] (v9);
\draw[black,dotted,ultra thick] (u12) to [bend right=20] (v11);
\draw[black,dotted,ultra thick] (u14) to [bend right=20] (v13);
\draw[black,dotted,ultra thick] (u16) to [bend right=20] (v15);
\draw[dotted,thick] (u18) to [bend right=20] (v17);
\draw[dotted,thick] (u20) to [bend right=20] (v19);
\draw[dotted,thick] (u22) to [bend right=20] (v21);
\draw[dotted,thick] (u24) to [bend right=20] (v23);
\draw[dotted,thick] (u26) to [bend right=20] (v25);
\draw[dotted,thick] (u28) to [bend right=20] (v27);
\draw[dotted,thick] (u30) to [bend right=20] (v29);
\draw[dotted,thick] (u32) to [bend right=20] (v31);

\draw[dotted,thick] (u2) to [bend left=20] (v4);
\draw[dotted,thick] (u4) to [bend left=20] (v6);
\draw[dotted,thick] (u6) to [bend left=20] (v8);
\draw[dotted,thick] (u8) to [bend left=20] (v10);
\draw[black,dotted,ultra thick] (u10) to [bend left=20] (v12);
\draw[black,dotted,ultra thick] (u12) to [bend left=20] (v14);
\draw[black,dotted,ultra thick] (u14) to [bend left=20] (v16);
\draw[black,dotted,ultra thick] (u16) to [bend left=20] (v18);
\draw[dotted,thick] (u18) to [bend left=20] (v20);
\draw[dotted,thick] (u20) to [bend left=20] (v22);
\draw[dotted,thick] (u22) to [bend left=20] (v24);
\draw[dotted,thick] (u24) to [bend left=20] (v26);
\draw[dotted,thick] (u26) to [bend left=20] (v28);
\draw[dotted,thick] (u28) to [bend left=20] (v30);
\draw[dotted,thick] (u30) to [bend left=20] (v32);
\draw[dotted,thick] (u32) to [bend left=20] (v34);

\node (t3) at (3.5,3) {};
\node (t7) at (7.5,3) {};
\node[black,thick] (t11) at (11.5,3) {};
\node[black,thick] (t15) at (15.5,3) {};
\node (t19) at (19.5,3) {};
\node (t23) at (23.5,3) {};
\node (t27) at (27.5,3) {};
\node (t31) at (31.5,3) {};

\draw (u2) -- (t3)--(u4);
\draw (u6) -- (t7)--(u8);
\draw[black,very thick] (u10) -- (t11)--(u12);
\draw[black,very thick] (u14) -- (t15)--(u16);
\draw (u18) -- (t19)--(u20);
\draw (u22) -- (t23)--(u24);
\draw (u26) -- (t27)--(u28);
\draw (u30) -- (t31)--(u32);

\node (s5) at (5.5,4.5) {};
\node[black,thick] (s13) at (13.5,4.5) {};
\node (s21) at (21.5,4.5) {};
\node (s29) at (29.5,4.5) {};

\draw (t3) -- (s5)--(t7);
\draw[black,very thick] (t11) -- (s13)--(t15);
\draw (t19) -- (s21)--(t23);
\draw (t27) -- (s29)--(t31);

\node (r9) at (9.5,6) {};
\node (r25) at (25.5,6) {};

\draw (s5) -- (r9)--(s13);
\draw (s21) -- (r25)--(s29);

\node (q17) at (17.5,7.5) {};
\draw (r9) -- (q17)--(r25);

\draw[cyan,fill] (v9) circle (0.25);
\draw[orange,fill] (v10) circle (0.25);
\draw[magenta,very thick] (v17) circle (0.3);
\draw[magenta,very thick] (v18) circle (0.3);
\draw[fill,LimeGreen] (s13) circle (.25);
\draw[blue,very thick] (13.5,2.6) circle (1);
\draw[blue,very thick] (18,5) circle (1.5);

\tikzstyle{every node}=[]
\draw[LimeGreen,above right] (13.8,4.4) node [] {\large $2$};
\draw[cyan,below] (9,-0.3) node []     {\large $i$};
\draw[orange,below] (10,-0.3) node []     {\large $j$};
\draw[magenta,below right] (17,-0.3) node []   {\large $h$};
\draw[black] (10.4,3.5) node []   {\LARGE $\Delta$};
\draw[blue] (13.5,2.6) node []   {\large $1$};
\draw[blue] (18,5) node []   {\large $6$};

\end{tikzpicture}
}

%% file: K3t_1.tex
\scalebox{0.75}{%
\begin{tikzpicture}[scale=1/2,auto=left]

\tikzstyle{every node}=[inner sep=1.5pt, fill=black,circle,draw]
\node[orange,ultra thick] (v1) at (1,0) {};
\node[blue,ultra thick] (v18) at (18,0) {};
\draw[dotted, thick] (v1)--(v18);
\node (v2) at (2,0) {};
\node (v3) at (3,0) {};
\node (v4) at (4,0) {};
\node (v5) at (5,0) {};
\node (v6) at (6,0) {};
\node (v7) at (7,0) {};
\node (v8) at (8,0) {};
\node (v9) at (9,0) {};
\node (v10) at (10,0) {};
\node (v11) at (11,0) {};
\node (v12) at (12,0) {};
\node (v13) at (13,0) {};
\node (v14) at (14,0) {};
\node (v15) at (15,0) {};
\node (v16) at (16,0) {};
\node[Green,ultra thick] (v17) at (17,0) {};

\node (u2) at (2.5,1) {};
\node (u4) at (4.5,1) {};
\node (u6) at (6.5,1) {};
\node (u8) at (8.5,1) {};
\node (u10) at (10.5,1) {};
\node (u12) at (12.5,1) {};
\node (u14) at (14.5,1) {};
\node[blue] (u16) at (16.5,1) {};

\draw[dotted,thick] (u2) to [bend right=20] (v1);
\draw[dotted,thick] (u4) to [bend right=20] (v3);
\draw[dotted,thick] (u6) to [bend right=20] (v5);
\draw[dotted,thick] (u8) to [bend right=20] (v7);
\draw[dotted,thick] (u10) to [bend right=20] (v9);
\draw[dotted,thick] (u12) to [bend right=20] (v11);
\draw[dotted,thick] (u14) to [bend right=20] (v13);
\draw[dotted,thick] (u16) to [bend right=20] (v15);

\draw[dotted,thick] (u2) to [bend left=20] (v4);
\draw[dotted,thick] (u4) to [bend left=20] (v6);
\draw[dotted,thick] (u6) to [bend left=20] (v8);
\draw[dotted,thick] (u8) to [bend left=20] (v10);
\draw[dotted,thick] (u10) to [bend left=20] (v12);
\draw[dotted,thick] (u12) to [bend left=20] (v14);
\draw[dotted,thick] (u14) to [bend left=20] (v16);
\draw[blue,ultra thick] (u16) to [bend left=20] (v18);

\node (t3) at (3.5,2) {};
\node (t7) at (7.5,2) {};
\node (t11) at (11.5,2) {};
\node[blue] (t15) at (15.5,2) {};

\draw (u2) -- (t3)--(u4);
\draw (u6) -- (t7)--(u8);
\draw (u10) -- (t11)--(u12);
\draw (u14) -- (t15)--(u16);

\node (s5) at (5.5,3) {};
\node[blue] (s13) at (13.5,3) {};

\draw (t3) -- (s5)--(t7);
\draw (t11) -- (s13)--(t15);

\node[blue,ultra thick] (r9) at (9.5,4) {};

\draw (s5) -- (r9)--(s13);

\draw[blue,ultra thick] (u16) -- (t15) -- (s13) -- (r9);

\tikzstyle{every node}=[]
\draw[orange] (-1,0.5) node [] {$\SO(G_t) \in B_1$};
\draw[Green] (17,-0.7) node [] {$\TI(G_t) \in B_2$};
\draw[blue] (20,0.5) node [] {$\TO(G_t) \in B_3$};
\draw[blue] (9.5,3.1) node [] {$\RR(G_t) \in B_3$};

\end{tikzpicture}
}

%% file: K3t_2.tex
\scalebox{0.73}{%
\begin{tikzpicture}[scale=1/2,auto=left]

 \path [fill=gray,opacity=0.3,draw=black]
(20.5,4) to [bend right=10] (16,0) to [bend right=5]
 (25,0) to [bend right=10] (20.5,4);

\tikzstyle{every node}=[inner sep=1.5pt, fill=black,circle,draw]
\node[orange,ultra thick] (v1) at (1,0) {};
\node[blue,ultra thick] (v34) at (34,0) {};
\draw[dotted, thick] (v1) -- (v34);
\node[orange,ultra thick] (v17) at (16,0) {};
\node[blue,ultra thick] (v18) at (19,0) {};
\node[Green,ultra thick] (v25) at (25,0) {};
\node[Green,ultra thick] (v29) at (30,0) {};
\node[Green,ultra thick] (v33) at (32,0) {};

\node[blue,ultra thick] (r1) at (17.5,10) {};
\node[blue,ultra thick] (s9) at (9.5,8) {};
\node[blue,ultra thick] (s27) at (25,8) {};

\node[magenta,ultra thick] (t1) at (23.188,6.4) {};
\node[magenta,ultra thick] (x1) at (22.292,5.6) {};
\node[magenta,ultra thick] (x2) at (24.084,7.2) {};

\node[orange,ultra thick] (n1) at (21.396,4.8) {};

\node[orange,ultra thick] (u1) at (20.5,4) {};

\draw[blue,ultra thick] (v18) -- (s9) -- (r1) -- (s27) -- (v34);
\draw[thick,dotted] (v1) -- (s9);
\draw[orange,ultra thick] (v1) -- (v17) -- (u1) -- (n1);
\draw[dotted,thick] (u1) -- (v25);
\draw[Green,ultra thick] (v25) -- (v33);

\draw[magenta,ultra thick] (x1) -- (x2);

\draw[cyan,ultra thick] (n1) -- (x1);
\draw[cyan,ultra thick] (t1) -- (v29);
\draw[cyan,ultra thick] (x2) -- (s27);

\tikzstyle{every node}=[]
\draw[orange] (1,-0.8) node [] {$\SO(G_{t+1})$};
\draw[blue] (34.9,-0.8) node [] {$\TO(G_{t+1})$};
\draw[Green] (31.7,-0.8) node [] {$\TI(G_{t+1})$};
\draw[blue] (17.5,9) node [] {$\RR(G_{t+1})$};

\draw[orange] (9,1) node [] {\Large $B'_1$};
\draw[blue] (30.5,5) node [] {\Large $B'_3$};
\draw[Green] (26.5,1) node [] {\Large $B'_2$};
\draw[magenta] (22.4,6.8) node [] {\Large $B^{t+1}$};

\draw[orange] (16,-0.8) node [] {$\SO(\Delta'_1)$};
\draw[blue] (19,-0.8) node [] {$\SI(\Delta'_1)$};
\draw[Green] (25,-0.8) node [] {$\TO(\Delta'_1)$};

\draw[gray] (9.5,5) node [] {\LARGE $\Delta_0$};
\draw[gray] (26,5) node [] {\LARGE $\Delta_1$};
\draw[gray] (20.5,1.5) node [] {\LARGE $\Delta'_1$};

\end{tikzpicture}
}

%% file: FatK5_1.tex
\scalebox{0.75}{%
\begin{tikzpicture}[scale=1/2,auto=left]

\tikzstyle{every node}=[inner sep=1.5pt, fill=black,circle,draw]
\node[orange,thick] (v1) at (1,0) {};
\node[LimeGreen,thick] (v34) at (34,0) {};
\draw[dotted, thick] (v1)--(v34);
\node[cyan,thick] (v2) at (2,0) {};
\node[cyan,thick] (v3) at (3,0) {};
\node[cyan,thick] (v4) at (4,0) {};
\node[cyan,thick] (v17) at (17,0) {};
\node[cyan,thick] (v18) at (18,0) {};
\node[cyan,thick] (v27) at (26.5,0) {};
\node[cyan,thick] (v29) at (29,0) {};
\node[magenta,thick] (v30) at (30,0) {};
\node[magenta,thick] (v31) at (31,0) {};
\node[magenta,thick] (v32) at (32,0) {};
\node[magenta,thick] (v33) at (33,0) {};

\node[orange,thick] (u2) at (2.5,1) {};
\node[blue,thick] (u16) at (16.5,1) {};
\node (u18) at (18.5,1) {};
\node[LimeGreen,thick] (u27) at (28,1) {};
\node[LimeGreen,thick] (u32) at (32.5,1) {};

\draw[orange, ultra thick] (u2) to [bend right=20] (v1);
\draw[dotted, thick] (u18) to [bend right=20] (v17);
\draw[dotted, thick] (u27) to [bend right=20] (v27);
\draw[dotted, thick] (u32) to [bend right=20] (v31);

\draw[dotted, thick] (u16) to [bend left=20] (v18);
\draw[LimeGreen,ultra thick] (u32) to [bend left=20] (v34);

\node[blue,thick] (r9) at (9.5,4) {};
\node[blue,thick] (r25) at (25.5,4) {};

\node[blue,thick] (q17) at (17.5,5) {};
\draw[blue,ultra thick] (r9) -- (q17);
\draw[blue,ultra thick] (q17)--(r25);

\node[blue,thick] (x1) at (7,3.4) {};
\draw[dashed,blue, ultra thick] (r9) to [bend right=8] (x1);
\draw[dashed,thick] (x1) to [bend right=8] (u2);
\draw[dashed,blue, ultra thick] (r9) to [bend left=10] (u16);
\draw[dashed,thick] (r25) to [bend right=10] (u18);

\node[blue,thick] (x2) at (28,3.4) {};
\node[LimeGreen,thick] (s29) at (30.2,2.4) {};
\draw[blue,dashed,ultra thick] (r25) to [bend left=8] (x2);
\draw[dashed, thick] (x2) to [bend left=8] (s29);
\draw[LimeGreen,dashed,ultra thick] (s29) to [bend left=8] (u32);
\draw[LimeGreen,dashed,ultra thick] (s29) to [bend right=8] (u27);

\draw[dotted, thick] (0.5,1) to [bend left=20] (v2);
\draw[magenta,ultra thick] (34.5,1) to [bend right=20] (v33);

\draw[orange,ultra thick] (0.5,0)--(v1);

\draw[dotted, thick] (v34)--(34.5,0);

\draw[cyan, ultra thick] (v2) -- (v29);
\draw[magenta, ultra thick] (v30)--(v33);

\end{tikzpicture}
}

%% file: FatK5_2.tex
\scalebox{0.75}{%
\begin{tikzpicture}[scale=1/2,auto=left]

\path [fill=gray,opacity=0.3,draw=black]
(12.5,2) to [bend right=15] (11,0) to [bend right=20]
 (14,0) to [bend right=10] (12.5,2);

\tikzstyle{every node}=[inner sep=1.5pt, fill=black,circle,draw]
\node[orange,thick] (v1) at (1,0) {};
\node[orange,thick] (v34) at (34,0) {};
\draw[dotted, thick] (v1)--(v34);
\node[orange,thick] (v11) at (11,0) {};
\node[cyan,thick] (v12) at (12,0) {};
\node[magenta,thick] (v13) at (13,0) {};
\node[LimeGreen,thick] (v14) at (14,0) {};
\node[orange,thick] (v15) at (15,0) {};
\node[orange,thick] (v17) at (17,0) {};
\node[orange,thick] (v18) at (18,0) {};

\node[orange,thick] (u2) at (2.5,1) {};
\node[blue,thick] (u12) at (12.5,2) {};
\node[magenta,thick] (u14) at (14.5,1) {};
\node[magenta,thick] (u16) at (16.5,1) {};
\node (u18) at (18.5,1) {};
\node[orange,thick] (u32) at (32.5,1) {};

\draw[orange,ultra thick] (u2) to [bend right=20] (v1);
\draw[magenta,ultra thick] (u14) to [bend right=20] (v13);

\draw[dotted,thick] (u18) to [bend right=20] (v17);

\draw[dotted, thick] (u16) to [bend left=20] (v18);
\draw[orange,ultra thick] (u32) to [bend left=20] (v34);

\node[magenta,thick] (t11) at (11.5,3.5) {};
\node[magenta,thick] (t15) at (15.5,3.5) {};

\node[magenta,thick] (s13) at (13.5,5) {};

\draw[magenta, ultra thick] (t11) -- (s13);
\draw[magenta,ultra thick] (s13)--(t15);

\node[orange,thick] (r9) at (9.5,7) {};
\node[orange,thick] (r25) at (25.5,7) {};

\node[orange,thick] (q17) at (17.5,8) {};
\draw[orange,ultra thick] (r9) -- (q17)--(r25);

\draw[orange,ultra thick] (v15) -- (v34);
\draw[orange, ultra thick] (v1) -- (v11);


\draw[dashed,thick] (r9) to [bend left=10] (s13);
\draw[magenta,dashed,ultra thick] (t15) to [bend left=10] (u16);
\draw[dashed,thick] (r25) to [bend right=20] (u18);
\draw[dashed,orange,ultra thick] (r9) to [bend right=20] (u2);
\draw[dashed,orange,ultra thick] (r25) to [bend left=20] (u32);

\draw[dashed,thick](t11) to [bend left=10] (u12);
\draw[magenta,dashed,ultra thick] (t15) to [bend right=10] (u14);

\tikzstyle{every node}=[]

\end{tikzpicture}
}